\documentclass{amsart}

\usepackage{amssymb}
\usepackage{amsmath}
\usepackage{amsthm}
\usepackage{color}
\usepackage{euscript}
\usepackage[linktoc=page,colorlinks,citecolor = green!50!black,linkcolor= blue,urlcolor = blue]{hyperref}
\usepackage{enumitem}
\usepackage{tikz-cd} 
\usepackage{mathrsfs}
\usepackage[all]{xy}
\usepackage{verbatim}

\theoremstyle{plain}
\newtheorem{thm}{Theorem}[section]
\newtheorem{lem}[thm]{Lemma}

\newtheorem*{conj}{Conjecture}
\newtheorem{cor}[thm]{Corollary}
\newtheorem{prop}[thm]{Proposition}

\theoremstyle{definition}
\newtheorem{defn}[thm]{Definition}
\newtheorem{exmp}[thm]{Example}

\theoremstyle{definition}

\newtheorem{rk}[thm]{Remark}
\newtheorem*{notation}{Notation}

\newcommand{\C}{\mathbb{C}}
\newcommand{\Z}{\mathbb{Z}}

\newcommand{\Om}{\mathcal{O} }
\newcommand{\A}{\mathcal{A}}

\newcommand{\T}{\EuScript{T}}
\newcommand{\scrG}{\EuScript{G}}
\newcommand{\cG}{\mathcal{G}}
\newcommand{\scrP}{\EuScript{P}}
\newcommand{\tr}{\mathfrak{tr}}
\newcommand{\R}{\mathcal{R}}
\newcommand{\Tt}{\tilde{\EuScript{T}}}
\newcommand{\B}{\mathcal{B}}
\newcommand{\G}{\mathbb{G}_m}
\newcommand{\fM}{\mathfrak{M}}
\newcommand{\baU}{\bar{U}}
\newcommand{\tU}{\tilde{U}}
\newcommand{\tV}{\tilde{V}}
\newcommand{\scrC}{\EuScript{C}}
\newcommand{\scrF}{\EuScript{F}}
\newcommand{\precon}{D\mkern-12mu/\mkern2mu}
\newcommand{\quabla}{\nabla\mkern-12mu/\mkern2mu}
\newcommand{\cW}{\mathcal{W}}
\newcommand{\sheafhom}{\mathcal{H} \kern -.5pt \mathit{om}}

\numberwithin{equation}{section}
\numberwithin{figure}{section}
\title{Dynamical invariants of mapping torus categories}
\author{Yusuf Bar{\i}\c{s} Kartal}
\date{}
\begin{document}
\begin{abstract}
	This paper describes constructions in homological
	algebra that are part of a strategy whose goal is to understand and classify symplectic mapping tori. More precisely, given a dg category and an auto-equivalence, satisfying certain assumptions, we introduce a category $M_\phi$-called the mapping torus category- that describes the wrapped Fukaya category of an open symplectic mapping torus. Then we define a family of bimodules on a natural deformation of $M_\phi$, uniquely characterize it and using this, we distinguish $M_\phi$ from the mapping torus category of the identity. The proof of the equivalence of $M_\phi$ with wrapped Fukaya category is proven in a different paper (\cite{nextpaper}).
\end{abstract}
	\maketitle

\tableofcontents
\parskip1em
\parindent0em	
\section{Introduction} \label{sec:intro}
\subsection{Motivation from symplectic geometry}\label{subsec:motiv}
Let $M^{2n}$ be a Weinstein manifold and let $\phi$ be a symplectomorphism. For simplicity, assume $\phi$ acts as the identity on the boundary and it is exact with respect to boundary. Associated to this data one can construct the open symplectic mapping torus as \begin{equation}T_\phi^{2n+2}:=(M\times \mathbb{R}\times S^1/(x,t,s)\sim(\phi(x),t+1,s ))\setminus\{[(x,t,s)]:t=0,s=1 \} \end{equation}
This is a symplectic fibration over punctured torus $T_0=T^2\setminus\{* \}$ with monodromy as shown in Figure \ref{figure:sympmt}. It can be shown to carry a Liouville structure and its contact boundary at infinity is isomorphic to that of $T_0\times M$, in other words the boundary of the mapping torus of identity.

One would like to distinguish the fillings $T_\phi$ and $T_0\times M$, when $\phi$ is not Hamiltonian isotopic to identity. An attempt can be made as follows: Assume the fillings are the same. Consider the partial compactification 
\begin{equation}\overline{T}_\phi:=M\times \mathbb{R}\times S^1/(x,t,s)\sim(\phi(x),t+1,s ) \end{equation}
Assume we are able to identify $\overline T_\phi$ with $\overline T_{1_M}=T^2\times M$. Every circle action on $T^2$ lifts to a circle action on $T^2\times M$; however, this is not the case with $\overline T_\phi$. Indeed, the flow of the obvious lift of $\partial_t$ at time $t=1$ gives us the symplectomorphism \begin{equation}\label{fiberwisesymp}[(x,t,s)]\mapsto[(\phi^{-1}(x),t,s)] \end{equation}which is different from the identity. In other words, it seems there are ``more circle actions'' on $T^2\times M$ and its flux group is bigger. The first and major limitation of this approach is our inability to identify the partial compactifications. Second limitation is even if one successfully runs the above program and rigorously computes the flux groups, they would only be able to conclude fiberwise $\phi$, the inverse of the symplectomorphism (\ref{fiberwisesymp}),
is Hamiltonian isotopic to identity. We do not know how to conclude the same for $\phi$ acting on $M$.
\begin{figure}\centering
	\includegraphics[height=4 cm]{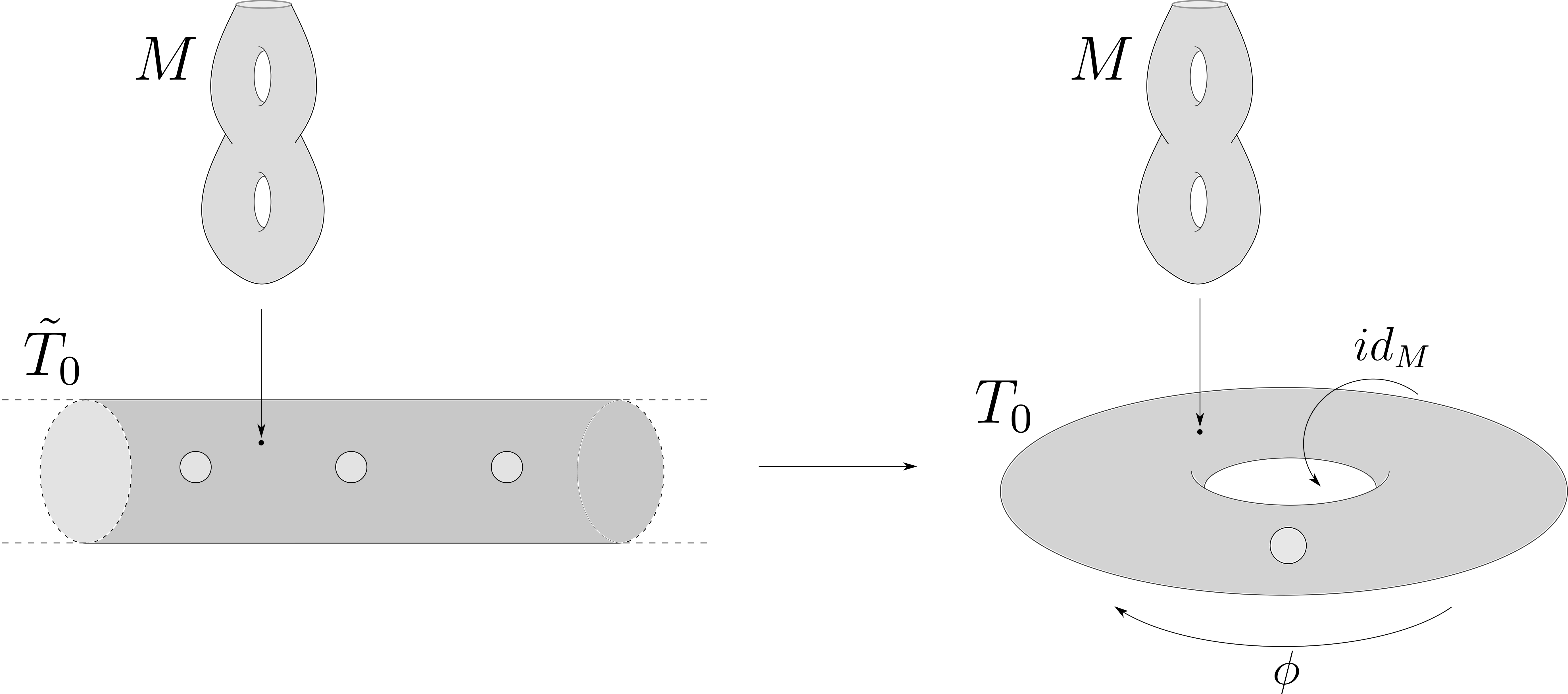}
	\caption{$T_\phi$ and its $\Z$-fold cover $\tilde T_0\times M$}
	\label{figure:sympmt}
\end{figure}

One can instead try to follow an algebraic counterpart of the same idea, namely by executing this procedure at the level of Fukaya categories. By \cite{fukdef}, a compactification of a symplectic manifold by adding a divisor gives rise to a deformation of its Fukaya category, called the relative Fukaya category. For instance, by \cite{lekper}, the (compact) Fukaya category of the punctured torus $T_0$ is derived equivalent to the perfect complexes over the nodal elliptic curve $\T_0$, and the relative Fukaya category corresponding to partial compactification $T_0\subset T^2$ is equivalent to perfect complexes over the Tate family, which deforms $\T_0$. ``The generic fiber'' of this deformation is equivalent to Fukaya category of $T^2$. Therefore, instead of considering partial compactifications of $T_\phi$, one can consider deformations of its Fukaya category, which admit a classification in terms of invariants such as Hochschild cohomology. In the example above, the category corresponding to $T_0\subset T^2$ is essentially the unique deformation of the Fukaya category of $T_0$. One then runs a categorical incarnation of the dynamical idea above that goes back to \cite{flux}. Namely, instead of counting circle actions on the manifold, one counts periodic ``flow lines'' on its Fukaya category. A periodic flow line is a family of derived equivalences of the category. Since derived equivalences are realized as convolution with an invertible bimodule (its ``Fourier-Mukai kernel''), one can instead consider periodic families of invertible bimodules. Seidel \cite{flux}, makes this notion precise and considers such families parametrized by elliptic curves as a categorical replacement for the circle actions. 

We run a similar idea on the Fukaya category of $T_\phi$. Fukaya category associated to a symplectic manifold is a category whose objects consist of Lagrangian manifolds. It is hard to construct interesting compact Lagrangians in $T_\phi$; for instance, if there are no $\phi$-invariant Lagrangians in $M$, then one does not have a Lagrangian in $T_\phi$ that is fibered over a longitude of the torus in Figure \ref{figure:introinclusion}. On the other hand, for any Lagrangian $L\subset M$ and any open arc $\ell\subset T_0$, one can construct a Lagrangian in $T_\phi$ fibered over $\ell$ and has fiber $L$, but this Lagrangian is non-compact. As a result, one needs to consider the version of Fukaya categories that contain non-compact Lagrangians as well, called the wrapped Fukaya category $\cW(T_\phi)$ of $T_\phi$. Another advantage of working with wrapped Fukaya categories is the strong properties they satisfy, such as homological smoothness. Note that one does not have an analogue of the relative Fukaya category for wrapped Fukaya categories. However, we will take this as heuristic and work with deformations of the category directly.

Computations with Fukaya categories can be challenging. For instance, algebro-geometric machinery for deformation theory computations is not available. Taking homological mirror symmetry as the guiding principle, we give a semi-algebro geometric description of the wrapped Fukaya category of $T_\phi$, which simplifies various computations, construction of explicit deformations, as well as families of bimodules over the deformations, which we see as the categorical flow lines. More precisely, we describe the wrapped Fukaya category of $T_\phi$ in terms of the wrapped Fukaya category of $M$ and the action of $\phi$ on the category. The construction is very general: to every $A_\infty$-category $\A$ and auto-equivalence-- still denoted by $\phi$-- we associate a category $M_\phi$, called \textbf{the mapping torus category of $\phi$}. The goal of this paper is to run the program above in purely algebraic terms to distinguish $M_\phi$ from $M_{1_\A}$, the mapping torus of the identity, under various assumptions on $\A$. Note that $M_\phi$ will be a category ``over $D^b(Coh(\T_0))$'', rather than $Perf(\T_0)$, mirroring the use of non-compact Lagrangians. 

An heuristic picture of the mapping torus category from algebraic geometry is as follows: given variety $X$ and automorphism $\phi_X$, one can construct an algebraic space by taking the product $\mathbb{P}^1\times X$ and gluing $\{0\}\times X$ to $\{\infty\}\times X$ after twisting by $\phi_X$ (i.e. by identifying $(0,x)$ with $(\infty,\phi_X(x))$ for all $x\in X$). This is a fibration over the nodal elliptic curve $\T_0$, and deforms compatibly with Tate family deforming $\T_0$. This deformation can be thought as the mirror analogue of the partial compactification $T_\phi\subset \overline{T}_\phi$. The generic fiber of this deformation is a rigid analytic mapping torus fibered over the generic fiber of the Tate family. For instance, when $\phi_X=1_X$, this is the product of an elliptic curve with $X$ (after appropriate base change). One has an analytic $\mathbb{G}_m$-action on the rigid analytic mapping torus, and it is periodic for $1_X$, whereas it is either non-periodic or has a larger period for other $\phi_X$. In other words, trivial mapping torus has more periodic actions, and this lets one to distinguish the non-trivial and trivial mapping tori. It is clear that before the deformation, there is no periodic action, as $\T_0$ is singular. Hence, one has to pass to this deformation, analogous to one's need to pass to partial compactification of $T_\phi$ to see any circle actions. This is the mirror analogue of the above program. We categorify this idea to distinguish $M_\phi$ from $M_{1_\A}$. In other words, we construct deformation of these categories over the formal disc. Algebraic computations allow us to classify this deformation as the unique one that is non-trivial in the first order. 

As mentioned, the idea to use ``periodic flow lines'' to distinguish Fukaya categories of symplectic manifolds is introduced in \cite{flux}. There the author uses the ``categorical flux'' to distinguish Fukaya categories of compact mapping tori (similar to $\overline{T}_\phi$), as well as some related manifolds that could not be distinguished by other means. There are some differences between \cite{flux} and our work: for instance, in our case, the periodic dynamics is broken for $M_\phi$ itself, but is only recovered once we pass to deformations. Another difference is that in \cite{flux}, the author does not construct ``flow lines'' and check if they are periodic. Instead, he realizes ''periodic flow lines'' as families parametrized by elliptic curves, where the elliptic curve is allowed to vary. In our case, we construct families parametrized by a formal model for an annulus in the analytic multiplicative group $\G^{an}$, i.e. by a formal scheme whose generic fiber is the annulus in $\G^{an}$. Therefore, as we characterize families up to changes in the special fiber, they can actually be thought as families parametrized by (an annulus in) the analytic multiplicative group. Another difference between \cite{flux} and the present work, at a more technical level, is the non-properness of the categories we consider. The tools introduced in \cite{flux} are for proper categories; more precisely, this is needed to show the uniqueness of ``flow lines''. We go around this problem by assuming the input category $\A$, despite being non-proper, is proper in each degree. This implies the same for $M_\phi$. 
\subsection{Categorical construction and the statement of the main theorem}\label{subsec:introcatcon}
Let $\A$ be an $A_\infty$-category over $\C$ and $\phi$ be an auto-equivalence, i.e. an $A_\infty$-functor $\phi:\A\rightarrow \A $ such that $H^*(\phi):H^*(\A)\rightarrow H^*(\A)$ is an equivalence. For simplicity assume $\A$ is a dg category and $\phi$ is a dg functor acting bijectively on objects and hom-complexes of $\A$. Based on this, we construct an $A_\infty$ category $M_\phi$ over $\C$, and we call it \textbf{the mapping torus category of $\phi$}. 

Briefly, the construction goes as follows. Consider the universal cover of the Tate curve $\Tt_0$ whose definition will be recalled in Section \ref{subsec:Tate} (also see Figure \ref{figure:projtate}). It is a nodal infinite chain of projective lines parametrized by $i\in\Z$, and it admits a translation automorphism $\tr$ which moves one projective line to the next. Consider the bounded derived category of coherent sheaves supported on finitely many projective lines, denoted by $D^b(Coh_p(\Tt_0))$. We will construct a dg category $\Om(\Tt_0)_{dg}$ whose triangulated envelope is a dg enhancement of $D^b(Coh_p(\Tt_0))$. Moreover, it admits a strict dg autoequivalence, still denoted by $\tr$, which lifts $\tr_*$. Then, $\tr\otimes \phi$ endows $\Om(\Tt_0)_{dg}\otimes\A$ with a $\Z$-action, and we define the mapping torus category as \begin{equation}\label{intromtdef}
M_\phi:=(\Om(\Tt_0)_{dg}\otimes\A)\#\Z \end{equation}
The smash product with $\Z$, whose definition will be recalled in Section \ref{sec:construction}, corresponds geometrically  to taking the quotient by the $\Z$-action. 

The following example justifies the terminology ``mapping torus category'' from an algebro-geometric perspective:
\begin{exmp}\label{exmpalggeo}
Let $\A$ be a dg model for $D^b(Coh(X))$, where $X$ is a variety over $\C$ and $\phi=(\phi_X)_*$ for an automorphism $\phi_X\curvearrowright X$. Consider the algebraic space \begin{equation}\label{eq:algtori}\Tt_0\times X/(t,x)\sim (\tr(t),\phi_X(x)) \end{equation}We expect $tw^\pi(M_\phi)$- idempotent completed twisted (triangulated) envelope- to be a dg enhancement of bounded derived category of coherent sheaves on this algebraic space (see also \cite[Remark 3.21]{nextpaper}). 
We showed this in the case $X=Spec(\C)$, when the construction gives the nodal elliptic curve $\T_0$ (see Figure \ref{figure:projtate}). See Lemma \ref{categorycomparisonlemma} for this result. This algebro-geometric version of the mapping torus that we have introduced informally earlier provides another motivation for the categorical construction.
\end{exmp}
\begin{rk}
The informal mirror symmetry motivation for the construction of $M_\phi$ is as follows: one knows by \cite{lekpol} and \cite{lekper} that the nodal elliptic curve $\T_0$ is  mirror dual to $T_0$. $T_\phi$ is obtained as a quotient of $\tilde T_0\times M$, where $\tilde T_0$ is an infinitely punctured cylinder that is covering $T_0$ (see Figure \ref{figure:sympmt}). Heuristically, one can think of $\tilde T_0$ as a mirror to $\Tt_0$. Assume $X$ is mirror to Weinstein manifold $M^{2n}$, and an automorphism of $X$, denoted by $\phi_X$, corresponds to $\phi$. A natural proposed mirror for $T_\phi$ is the algebraic space (\ref{eq:algtori}). $M_\phi$ is a straightforward categorification of the construction in Example \ref{exmpalggeo}.
\end{rk}
\begin{exmp}If $\phi=1_\A$, $M_\phi$ is Morita equivalent to $\scrC oh(\T_0)\otimes\A$, where $\scrC oh(\T_0)$ is a dg model for $D^b(Coh(\T_0))$. Thus, the category of perfect modules over $M_\phi$ is equivalent to the category of perfect modules over $\scrC oh(\T_0)\otimes\A$.
\end{exmp}
\begin{figure}\centering
	\includegraphics[height=2 cm]{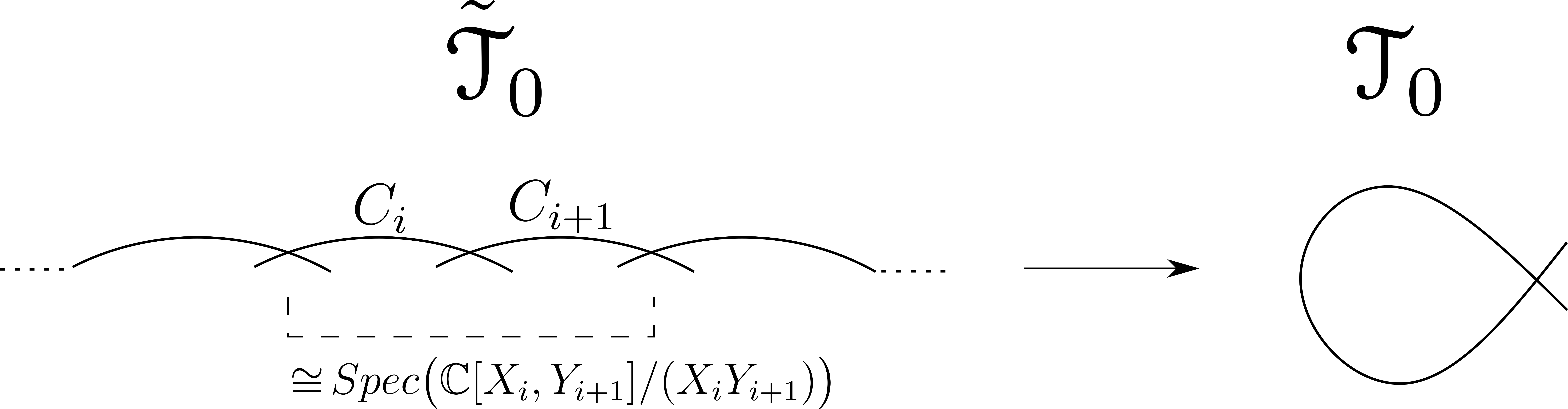}
	\caption{The nodal elliptic curve $\T_0$, and its $\Z$-fold covering $\Tt_0$ }
	\label{figure:projtate}
\end{figure}
We will assume the following conditions hold throughout the paper:
\begin{enumerate}[label=\textbf{C.\arabic*}]
\item\label{C1} $\A$ is (homologically) smooth, see \cite{koso} for a definition
\item\label{C2} $\A$ is proper in each degree and bounded below, i.e. $H^*(hom_\A(x,y))=0$ is finite dimensional in each degree and vanishes for $*\ll 0$ for any $x,y\in Ob(\A)$
\item\label{C3} $HH^i(\A)$, the $i^{th}$ Hochschild cohomology group of $\A$, is $0$ for $i<0$ and is isomorphic to $\C$ for $i=0$
\end{enumerate}
Based on this $M_\phi$ will be shown to satisfy \ref{C1}-\ref{C3} as well. 

Now we can state our main theorem:
\begin{thm}\label{mainthm}
Let $\A$ be as above and assume further that $HH^1(\A)=HH^2(\A)=0$. Assume $M_\phi$ is Morita equivalent to $M_{1_\A}$. Then, $\phi\simeq 1_\A$.
\end{thm}
\subsection{Sketch of the proof}
The proof goes as follows. Assume $M_\phi$ is Morita equivalent to $M_{1_\A}$. The notion of Morita equivalence will be recalled later in Definition \ref{moritadefn}, but we remark that this is equivalent to equivalence of derived categories for $A_\infty$-categories over $\C$. To any categorical mapping torus one can associate a natural formal deformation (with curvature) over the topological local ring $R=\C[[q]]$. We denote this deformation by $M_\phi^R$ (resp. $M_{1_\A}^R$). Its explicit construction is as follows. There exists a natural smoothing of $\Tt_0$, denoted by $\Tt_R$ (see Figure \ref{figure:incltate}). To this we associate a curved dg category, denoted by $\Om(\Tt_R)_{cdg}$, and then apply the same construction as (\ref{intromtdef}) replacing $\Om(\Tt_0)_{dg}$ by $\Om(\Tt_R)_{cdg}$. The deformations $M_\phi^R$ and $M_{1_\A}^R$ have no a priori relation to the Morita equivalence; however, $HH^2(M_\phi)\cong HH^2(M_{1_\A})\cong \C$, under the assumptions of the theorem and the construction. Hence, there is only one formal deformation that is non-trivial at first order (up to reparametrization). Thus, we may assume without loss of generality that the Morita equivalence deforms to a Morita equivalence between $M_\phi^R$ and $M_{1_\A}^R$. 

$M_\phi^R$ (resp. $M_\phi$) carries a canonical $\G(R)$ (resp. $\G$)-action for which the infinitesimal action makes sense (i.e. one can differentiate the action, see Definition \ref{defn:prorat}). Infinitesimal action gives a class $\gamma_\phi^R\in HH^1(M_\phi^R)$ (resp. $\gamma_\phi\in HH^1(M_\phi)$). The action can be considered as a family of $M_\phi^R$-bimodules which is parametrized by the formal spectrum of $\C[t,t^{-1}][[q]]$, and which ``follows'' the class $1\otimes \gamma_\phi^R$ along $t\partial_t$ direction. This family can be considered as a ``short flow line'' for $1\otimes \gamma_\phi^R$, and we extend it to a ``longer flow line'', i.e. to a family over the formal spectrum of $\C[u,t][[q]]/(ut-q)$. This is the formal completion of $\{ut=0\}\subset \mathbb A^2_\C$ and contains the formal spectrum $Spf(\C[t,t^{-1}][[q]])$ as a formal open subscheme, where the inclusion is induced by $t\mapsto t,u\mapsto qt^{-1}$. See Figure \ref{figure:introinclusion}.
\begin{figure}
\centering
\includegraphics[height=4cm]{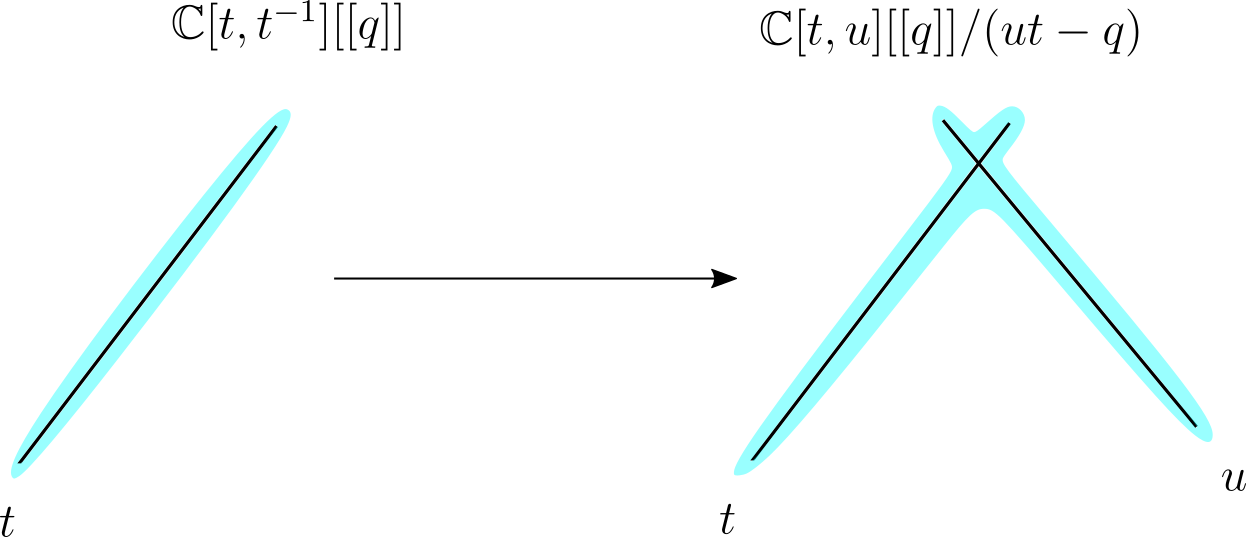}
\caption{Inclusion of $\C[t,t^{-1}][[q]]$ into $\C[u,t][[q]]/(ut-q)$}
\label{figure:introinclusion}
\end{figure}
To construct the extended family we consider a formal subscheme $\cG_R\subset \Tt_R\times\Tt_R\times Spf(\C[u,t][[q]]/(ut-q) )$ with the following properties:
\begin{enumerate}
\item it is flat over $Spf(\C[u,t][[q]]/(ut-q) )$
\item it restricts to the graph of $\widehat{\mathbb{G}_{m,R}}$-action (see Remark \ref{actionremark}) over the formal spectrum $Spf(\C[t,t^{-1}][[q]])$
\item it restricts to graph of composition of the inverse action with backwards translation $\tr^{-1}$ over $Spf(\C[u,u^{-1}][[q]])$
\end{enumerate}
In particular, we obtain the diagonal over the $R$-point $t=1$ and the graph of backwards translation over $u=1$. We turn $\cG_R$ into a family of bimodules over $M_\phi^R$ by defining an $\Om(\Tt_R)_{cdg}$-$\Om(\Tt_R)_{cdg}$-bimodule \begin{equation}``(\scrF,\scrF')\mapsto hom_{\Tt_R\times\Tt_R}(q^*\scrF,p^*\scrF'\otimes\cG_R )" \end{equation} and showing it naturally descends to $M_\phi^R=(\Om(\Tt_R)_{cdg}\otimes\A )\#\Z$. After some other technical replacements, we obtain a family $\scrG_R^{sf}$ of bimodules over $M_\phi^R$ parametrized by $\C[u,t][[q]]/(ut-q)$ satisfying properties \ref{G1}-\ref{G3} below for $\gamma=\gamma_\phi^R$ and which restricts to ``fiberwise $\phi$'' at $u=1$, i.e. to the bimodule corresponding to descent of auto-equivalence $1\otimes\phi$ on $\Om(\Tt_R)_{cdg}\otimes\A$ to $M_\phi^R$. Hence, if we can show families constructed in this way correspond to each other under the Morita equivalence between $M_\phi^R$ and $M_{1_\A}^R$, this would imply the triviality of ``fiberwise $\phi$'' and therefore triviality of $\phi$, finishing the proof of the theorem.

For this, we would first need to show the classes $\gamma_\phi^R$ and  $\gamma_{1_\A}^R$ correspond to each other under the isomorphism $HH^1(M_\phi^R)\cong HH^1(M_{1_\A}^R)$ induced by the Morita equivalence. To achieve this, we prove in Section \ref{sec:rank2} that these classes fall into natural rank $2$ lattices inside $HH^1(M_\phi^R)\cong \C^2$ resp. $HH^1(M_{1_\A}^R)\cong \C^2$ that are matched by the Morita equivalence, and show in Section \ref{sec:symmetries} that the symmetries of $M_{1_\A}^R$ induce $SL_2(\Z)$ symmetry on the lattice. Hence, we can use these categorical symmetries to fix the initial Morita equivalence so that the classes $\gamma_\phi^R$ and  $\gamma_{1_\A}^R$ match. 

Given this result, one would only need to prove a general theorem that states once the class $\gamma$ is fixed, the family is uniquely characterized by the following axioms:
\begin{enumerate}[label=\textbf{G.\arabic*}]
\item\label{G1} The restriction $\fM|_{q=0}$ is coherent. This is equivalent to its representability by an object of $tw^\pi(\B_0\otimes\B_0^{op}\otimes ``\mathcal{C}oh(A)")$. See Definition \ref{compactfamdef}.
\item The restriction $\fM|_{t=1}$ is isomorphic to the diagonal bimodule over $\B$.
\item The family follows the class $1\otimes \gamma\in HH^1(\B^e)$.
\end{enumerate}
This is achieved in Theorem \ref{uniquenesstheorem}, namely we show that two families satisfying \ref{G1}-\ref{G3} are quasi-isomorphic up to $q$-torsion. The proof of Theorem \ref{uniquenesstheorem} relies on two things: the ideas in \cite{flux}, which we recall in Section \ref{subsec:familyreview}, and the algebra/geometry of modules over $\C[u,t][[q]]/(ut-q)$ which carry connections along the derivation $t\partial_t-u\partial_u$. More explicitly, given two such family $\scrG_1$ and $\scrG_2$, we show the hom-complexes in the category of families involving them are chain complexes of $\C[u,t][[q]]/(ut-q)$-modules carrying such connections in each degree that commute with the differentials. Hence, degree $0$ homomorphisms in the cohomological category give rank $1$ modules with connection, and we show in Appendix \ref{sec:modules} that such modules are free up to $q$-torsion. Following this line of ideas we prove the isomorphism $\scrG_1|_{t=1}\simeq \scrG_2|_{t=1}$ extends over $\C[u,t][[q]]/(ut-q)$ to an isomorphism up to $q$-torsion. This completes the proof. 

Now, let us detail the moral idea given in Section \ref{subsec:introcatcon} for the algebro-geometric minded reader. Consider the algebro-geometric torus given in Example \ref{exmpalggeo}. It has a natural deformation \begin{equation}\mathfrak Y=\Tt_R\times X/(t,x)\sim (\tr(t),\phi_X(x)) \end{equation}which is a fibration over the formal smoothing $\T_R=\Tt_R/t\sim\tr(t)$. Its generic fiber $\mathfrak Y_{\C((q))}$ (in the sense of Raynaud, see \cite[Section 5]{Temkin2015}) gives \begin{equation}``\mathbb{G}_{m,\C((q))}^{an}\times X/(t,x)\sim (qt,\phi_X(x))"\end{equation}a rigid analytic version of $\C^*\times X/(t,x)\sim (q_0t,\phi_X(x))$, where $|q_0|<1$. There is an action of $\mathbb{G}_{m,\C((q))}^{an}$ on this rigid analytic space; however, it descends to an action of the elliptic curve $\mathbb{G}_{m,\C((q))}^{an}/q$ if and only if $\phi_X=1_X$. In other words, the trivial mapping torus will be distinguished from the others in that the restriction of the graph of the action to $z=q\in \mathbb{G}_{m,\C((q))}^{an}$ is the diagonal of $\mathfrak Y_{\C((q))}$ while in general it is the graph of fiberwise $\phi_X$. This action can be thought as analogous to the flow of a vector field. The uniqueness of the family is an analogue of the uniqueness of the flow of a vector field. This is more explicit if we consider the philosophy of Raynaud and realize rigid analytic spaces as formal schemes over $R=\C[[q]]$ up to admissable blow-ups in the special fiber $q=0$. In particular, the family $\scrG_R^{sf}$ obtained from the graph \begin{equation}\cG_R\subset \Tt_R\times \Tt_R\times Spf(\C[u,t][[q]]/(ut-q) )\end{equation} morally corresponds to such a degeneration of the graph of action, restricted to a smaller annulus in $\mathbb{G}_{m,\C((q))}^{an}$ afterwards. 
\begin{rk}
As we have explained briefly, the proof can also be thought as an algebraic version of the argument in Section \ref{subsec:motiv}. The deformation $M_\phi^R$ is analogous to partial compactification $\overline{T}_\phi$ (see also \cite{fukdef}). The Hochschild cohomology class $\gamma_\phi^R$ is an algebraic analogue of the (lift of) vector field $\partial_t$, and the family $\scrG_R^{sf}$ is the analogue of its flow. Hence, the restriction of this family to $u=1$ is analogous to time $1$-flow of $\partial_t$ (time $(-1)$-flow to be precise), giving us ``fiberwise $\phi$'' in both cases. The problem of concluding the triviality of $\phi$ from the triviality of fiberwise $\phi$ has an easy solution in categorical version. 
\end{rk}
\subsection{Outline}
In Sections \ref{sec:tate} and \ref{sec:dgmodel} we review the construction of $\Tt_0$, $\Tt_R$, and present the dg model $\Om(\Tt_0)_{dg}$ and its deformation $\Om(\Tt_R)_{cdg}$. In Section \ref{sec:construction} we review the smash products and define $M_\phi$ and $M_\phi^R$. Section \ref{sec:hoch} is dedicated to computation of Hochschild cohomology and its results will be referred in other computations later. In Section \ref{sec:family}, we construct the family $\scrG_R^{sf}$ and prove it satisfies desired properties. This section also contains a brief review of families. Sections \ref{sec:rank2} and \ref{sec:symmetries} provide us the statements we need to fix the image of $\gamma_\phi\in HH^1(M_\phi)\cong\C^2$ under the Morita equivalence. In Section \ref{sec:rank2}, we show that the classes that are obtained as the infinitesimal action of a $\G$ (resp. $\G(R)$)-action on $M_\phi$ (resp. $M_\phi^R$) form a copy of $\Z^2$ inside $HH^1(M_\phi)$ (resp. $HH^1(M_\phi^R)\cong R^2$) generated by basis elements. This ``cocharacter lattice'' is obviously preserved under Morita equivalences, and Section \ref{sec:symmetries} provides us symmetries of the categories acting transitively on primitive elements of the lattice. In Section \ref{sec:unique}, we finally conclude the proof of uniqueness(up to $q$-torsion) of families satisfying \ref{G1}-\ref{G3} and the proof of Theorem \ref{mainthm}. In the final section, Section \ref{sec:another}, we relate the growth rates of $rk(HH^*(M_\phi^R,\Phi^k_f))$, where $\Phi^k_f$ is the bimodule kernel of fiberwise $\phi^k$, to growth rates for $\phi$. In Appendix \ref{sec:modules}, we prove some results (such as freeness up to $q$-torsion) for finitely generated modules over $A_R=\C[u,t][[q]]/(ut-q)$ with connections along $t\partial_t-u\partial_u$. 
\subsection{Applications and generalizations}
In \cite{nextpaper}, we show that $M_\phi$ is actually derived equivalent to $\cW(T_\phi)$ with a canonical grading. Using this we find pairs of Liouville manifolds that can be distinguished by Theorem \ref{mainthm}, but that have the same topology, contact boundary, symplectic cohomology groups, etc. Indeed, these examples are obtained by attaching subcritical handles to $T_\phi$ and $T_{1_M}$. One can kill the first cohomology of $T_\phi$, resp. $T_{1_M}$ by this process without changing wrapped Fukaya category. Theorem \ref{mainthm} implies $\cW(T_\phi)$ and $\cW(T_0\times M)$ -endowed with canonical gradings- are not derived equivalent (one can grade these categories in different ways, but after killing the first cohomology the grading is unique). 
Note vanishing of first cohomology implies that arguments involving flux cannot be used either. 
%

We believe Theorem \ref{mainthm} can be generalized as follows:
\begin{conj}\label{betterconjecture}
Assume $\A$ is as in Theorem \ref{mainthm}. Let $\phi$ and $\phi'$ be two auto-equivalences satisfying the stated conditions and assume $M_\phi$ and $M_{\phi'}$ are Morita equivalent. Then $\phi$ and $\phi'$ have the same order.
\end{conj}
This produces an infinite family of pairwise non-equivalent categories; therefore, by the results of \cite{nextpaper}, an infinite family of pairwise non-symplectomorphic Liouville manifolds. To prove Conjecture \ref{betterconjecture}, one only needs to prove existence of families like $\scrG_\phi$ that follow different (primitive) classes in the natural rank $2$ lattice in $HH^1(M_\phi^R)$. Then, one obtains ``flow lines along any class in the rank two lattice''. More precisely, these are families parametrized by formal schemes possibly different from $Spf(A_R)$ (it is $Spf(A_R)$ when the class is primitive in the lattice, note we do not define families with more general bases). The order of $\phi$ would be the index of the subgroup of the lattice given by the elements for which the restriction of the corresponding flow line to ``other end'' is the diagonal. More concretely, given $\gamma=k\gamma_0$, where $\gamma_0$ is a primitive class, construct a family parametrized by $Spf(A_R)$ that follow $\gamma_0$ and that restrict to diagonal at $t=1$ and to $\Psi_1$ at $u=1$. Inductively construct a family parametrized by $A_R$ that follow $\gamma_0$ and that restricts to $\Psi_1$ at $t=1$ and to $\Psi_2$ at $u=1$, \dots, a family that restricts to $\Psi_{k-1}$ at $t=1$ and to $\Psi_k$ at $u=1$. Then, the elements $k\gamma_0$ such that $\Psi_k$ is quasi-isomorphic to diagonal bimodule form a subgroup of the lattice whose index is given by the order of $\phi$. This subgroup is intrinsic to $M_\phi$ and should be thought as an analogue of the flux group.

A generalization in an orthogonal direction is the following: one can construct a version of open symplectic mapping torus for two commuting symplectomorphisms: namely, given $\phi$ and $\psi$ acting on $M$, construct $T_{\phi,\psi}$ as the quotient of $(M\times \mathbb{R}\times \mathbb{R})\setminus (M\times \Z\times \Z)$ by the relations 
\begin{equation}
(x,t,s)\sim (\phi(x),t+1,s)\sim (\psi(x),t,s+1)
\end{equation}
This generalizes $T_\phi$ and it still carries the structure of a symplectic fibration over $T_0=T^2\setminus\{*\}$ with non-trivial monodromy in both $\partial_t$ and $\partial_s$ directions. One can easily produce an analogous algebraic model for this construction. First build a model for the doubly infinite cover $(\mathbb{R}\times\mathbb{R})\setminus(\Z\times\Z)\to T_0$ as follows: $\Om(\Tt_0)_{dg}$ contains a non-full quasi-equivalent subcategory $\Om(\Tt_0)_{dg}^{eval}$, on which the $\G$-action is rational. Consider the category $\tilde{\tilde{\Om}}$ whose objects are given as the pairs $(\scrF,i)$, where $\scrF\in ob(\Om(\Tt_0)_{dg})$ and $i\in \Z$. Let $\tilde{\tilde{\Om}}((\scrF,i),(\scrF',i') )$ be defined as the weight $i'-i$ part of $\Om(\Tt_0)_{dg}^{eval}(\scrF,\scrF')$ with respect to the rational $\G$-action. This category carries a strict action of $\Z\times\Z$, where the first action is induced by $\tr$ and the second by $(L,i)\mapsto(L,i+1)$. Clearly, $\tilde{\tilde{\Om}}\#\Z$ (for the second $\Z$-action) is equivalent to $\Om(\Tt_0)_{dg}^{eval}$ (hence to $\Om(\Tt_0)_{dg}$). Then, given $\A$ as before with strictly commuting strict dg auto-equivalences $\phi$ and $\psi$, define $M_{\phi,\psi}$ to be $(\tilde{\tilde{\Om}}\otimes \A)\# (\Z\times\Z )$. 
We expect this category to be equivalent to $\cW(T_{\phi,\psi})$ with a canonical grading. We believe the following generalization of Theorem \ref{mainthm} and Conjecture \ref{betterconjecture} holds:
\begin{conj}\label{conjecture2}
$M_{\phi,\psi}$ is Morita equivalent to $M_{\phi',\psi'}$ if and only if the abelian subgroups $\langle\phi,\psi\rangle$ and $\langle\phi',\psi'\rangle$ of $Auteq(D^\pi(\A))$ are the same.
\end{conj}
The intuition for this conjecture is similar. We believe this conjecture can be proven using similar steps, although we have not checked this. 

One can also consider the category $M_{\phi[m],[n]}$, where $[m]$ and $[n]$ are the shift functors. An easy version of  Conjecture \ref{conjecture2} would state that $M_{\phi[m],[n]}$ is not Morita equivalent to $M_{[m'],[n']}$, unless $\phi$ is quasi-isomorphic to a shift functor itself. Indeed, such a result can presumably be proven applying the results of this paper. First, observe $M_{\phi[m],[n]}$  can be obtained from $M_\phi$ by modifying the grading. Therefore, the same holds for Hochschild cohomology groups and in particular, one can show that $HH^*(M_{\phi[m],[n]})$ remains the same.
One can attempt to construct a family following the class corresponding to $\gamma_\phi$ via change of grading. $\G\times\G$ still acts on $M_{\phi[m],[n]}$, and  there exists a rank $2$ lattice consisting of classes followed by cocharacters of this action, as in Section \ref{sec:rank2}. The action of $SL_2(\Z)$-elements on $M_{1_\A}$ (see the action constructed in Section \ref{sec:symmetries}) turn into isomorphisms of $M_{[m'],[n']}$ and $M_{[m''],[n'']}$, for when $gcd(m,n)=gcd(m',n')$ (i.e. when they are in the same $SL_2(\Z)$-orbit). Hence, starting with a Morita equivalence $M_{\phi[m],[n]}\simeq M_{[m'],[n']}$, one can switch to a Morita equivalence with $M_{[m''],[n'']}$ that preserves the canonical first Hochschild cohomology class. Their deformations match as before, and the families following degree $1$ classes also match again. The same argument works, but this time only to conclude that $\phi$ is equivalent to a shift functor. The importance of this generalization is that one obtains all possible gradings of $\cW(T_\phi)$ by changing $[m]$ and $[n]$; therefore, proving that $T_0\times M$ and $T_\phi$ are not symplectomorphic (as opposed to graded symplectomorphic, note grading becomes canonical after some subcritical handle attachment, i.e. one concludes that handle attached manifolds are not symplectomorphic without any need to use this generalization).
\subsection*{Notational remarks}
$R$ will always denote $\C[[q]]$ with the $q$-adic topology. Similarly, $A_R=\C[u,t][[q]]/(ut-q)$ with the $q$-adic topology and $A=\C[u,t]/(ut)=A_R/(q)$. $Spf(B)$ denotes the formal spectrum of a complete topological ring $B$ equipped with $I$-adic topology for an ideal $I\subset B$. This is a ringed space whose underlying topological space is $Spec(B/I)$(which is homeomorphic to $Spec(B/I^m)$ for any $m>0$) and whose ring of global functions is the topological ring $\lim\limits_{\longleftarrow} B/I^n$. For more details see \cite{bosch}. Note, in our paper most formal affine schemes are completions of varieties along a closed subvariety. 

Constructions/concepts over $R=\C[[q]]$ are implicitly assumed to be $q$-adically completed and continuous. This applies to categories over $R$, Hochschild cochains $CC^*(\B)$ of such categories, and to tensor products of topological complete modules over $R$. For instance if $M$ and $N$ are such modules, $M\otimes N$ refers to $M\hat\otimes_R N$, which is the $q$-adic completion of $M\otimes N$. If $M$ is over $R$ and $N$ is over $\C$, $M\otimes N$ refers to $q$-adic completion of $M\otimes_\C N$. We also mostly drop the subscripts of tensor products from the notation. Similarly, the base of products of schemes or formal schemes are written only when it is unclear (for instance $\Tt_R\times \Tt_R$ refers to fiber product over $Spf(R)$). 

We have elaborated on the definition of $\Tt_0$ in Section \ref{sec:tate} (see also Figure \ref{figure:incltate}). Indeed one can take \begin{equation}\T_0:=\Tt_0/t\sim \tr(t)
\end{equation}
as the definition. For an explicit equation defining $\T_0$, see \cite{lekper}.

Given an ordinary algebra $B$, $\scrC_{dg}(B)$ denotes the dg category of chain complexes over $B$.

Given dg categories $\B$ and $\B'$, we can their tensor product category as a category with objects $Ob(\B)\times Ob(\B')$. Let $b\times b'$ denote the corresponding object of $\B\otimes \B'$ for given $b\in Ob(\B),b'\in Ob(\B')$. Morphisms satisfy 
\begin{equation}
(\B\otimes \B')(b_1\times b'_1,b_2\times b_2')=\B(b_1,b_2)\otimes \B'(b_1',b_2')
\end{equation}
See \cite{kellerdg} for more details.

For a given $A_\infty$-category $\B$, $tw(\B)$ stands for the category of twisted complexes over $\B$ and $tw^\pi(\B)$ stands for the split-closure (a.k.a. idempotent completion) of $tw(\B)$. For a definition see \cite[Chapter I.3,I.4]{seidelbook}. $D^\pi(\B)$ stands for the triangulated category $H^0(tw^\pi(\B))$. 
A dg/$A_\infty$ enhancement of a triangulated category $D$ is a dg/$A_\infty$ category $\B$ such that $D$ is equivalent to $H^0(\B)$ as a triangulated category. 

By generation, we mean split generation unless specified otherwise. See \cite[Chapter I.4]{seidelbook}. We used the notations $CC^*(\B)$ and $CC^*(\B,\B)$ interchangeably. They both stand for the Hochschild complex of an $A_\infty$-category $\B$. See \cite{seidelK3}, \cite{categoricaldynamics}. The notation $Bimod(\B,\B')$ is used to mean the dg category of $A_\infty$-bimodules over $\B$-$\B'$. There is a functor \begin{equation}
Bimod(\B,\B)\rightarrow \scrC_{dg}(\C)\atop \fM\mapsto CC^*(\B,\fM)
\end{equation} which is naturally quasi-isomorphic to Yoneda functor of the diagonal bimodule. In the case of an $A_\infty$-algebra over $\C$, $CC^*(\B,\fM)$ has underlying graded vector space $\bigoplus_{i\geq 0} hom_\C(\B^{\otimes i},\fM)[-i]=hom_\C(T\B[1],\fM)$, where $T\B[1]=\bigoplus_{i\geq 0}\B^{\otimes i}[i]$(which is also defined in Section \ref{sec:family}). We note that this direct sum means each degree of each summand is summed separately. Also, as remarked before the constructions take place in the category of completed $R$-modules in the case $\B$ is a curved category over $R$. For instance, $hom(\B^{\otimes i},\fM)$ only involves convergent sums of continuous homomorphisms and direct sums are assumed to be $q$-adically completed. For the differential of $CC^*(\B,\fM)$, which involves $\mu_\B$ and $\mu_\fM$, see \cite[Remark 9.2]{categoricaldynamics}.

For more homological algebra preliminaries see \cite{kellerdg},  \cite{seidelK3},\cite{seidelbook}, \cite{categoricaldynamics}.
\subsection*{Acknowledgements} 
I would like to first thank to my advisor Paul Seidel, who introduced me to this problem and without his guidance this work would not be possible. I would also like to thank Padmavathi Srinivasan for illuminating conversations on the Tate curve and Neron models, to Yank$\i$ Lekili for conversations on their work \cite{lekper} and \cite{lekpol} and to Dhruv Ranganathan for conversations on the Tate curve as a rigid model for $\mathbb G_{m,K}^{an}$. Finally, I would like to thank the referee for reading the paper, and useful suggestions. This work was partially supported by NSF grant DMS-1500954 and by the Simons Foundation (through a Simons Investigator award).
\section{The universal cover of the Tate curve}\label{sec:tate}
\subsection{Reminder on the construction of $\Tt_R$}\label{subsec:Tate}
We first review the construction of $\Tt_R$ following \cite{lekper}. We slightly change the notation. Recall $R=\C[[q]]$ endowed with $q$-adic topology.  

Given $i\in\Z$, let $\baU_{i+1/2}$ denote $Spec\big(\C[q][X_i,Y_{i+1}]/(X_i Y_{i+1}-q) \big)$. It is a scheme over $Spec(\C[q])$, and it is isomorphic to $\mathbb{A}^2_\C$ as a scheme over $\C$. Moreover, \begin{equation} \baU_{i+1/2}[X_i^{-1}]\cong Spec\big(\C[q][X_i,X_i^{-1}]\big) \end{equation} is isomorphic to \begin{equation} \baU_{i-1/2}[Y_i^{-1}]\cong Spec\big(\C[q][Y_i,Y_i^{-1}] \big) \end{equation} as a scheme over $\C[q]$. Denote this scheme by $\bar{V}_i.$ The isomorphism is given by the coordinate change $X_i\leftrightarrow Y_i^{-1}$. In other words, the coordinates $X_i$ and $Y_i$ satisfy $X_iY_i=1$ on $\bar{V}_i$. 

By using the identifications $\baU_{i+1/2}[X_i^{-1}]\cong  \baU_{i-1/2}[Y_i^{-1}]$, we can glue $\baU_{i+1/2}$, $i\in\Z$. Hence, we obtain a scheme over $Spec(\C[q])$, which we denote by $\Tt_{\C[q]}$. It is not Noetherian and it is covered by charts $\baU_{i+1/2}$, $i\in\Z$. 

Note, there is a $\mathbb{G}_{m,\C[q]}$-action over $\C[q]$ on this scheme. Locally, the action is given by \begin{equation}\label{eq:firstaction}
Y_{i+1}\mapsto tY_{i+1} \text{ and } X_i\mapsto t^{-1}X_i\end{equation} where $t$ is the coordinate of $\mathbb{G}_{m,\C[q]}$.

We will mainly be interested in \begin{equation}\Tt_0:=\Tt_{\C[q]}|_{q=0}=\Tt_{\C[q]}\times_{Spec(\C[q])} Spec(\C[q]/(q))\end{equation} and its formal completion inside $\Tt_{\C[q]}$. We denote this formal completion by \begin{equation}\Tt_R:= \Tt_{\C[q]}\times_{Spec(\C[q])} Spf(R)\end{equation} where the fiber product is taken with respect to the obvious morphism \begin{equation}Spf(R)\rightarrow Spec(\C[q])\end{equation}
(recall that $Spf(R)$ denotes the formal spectrum of the topological ring $R=\C[[q]]$). 

Let $U_{i+1/2}:=\baU_{i+1/2}|_{q=0}$ and $\tU_{i+1/2}:=\baU_{i+1/2}\times_{Spec(\C[q])} Spf(\C[[q]])$. In the coordinates above,  \begin{equation}U_{i+1/2}=Spec(\C[X_i,Y_{i+1}]/(X_iY_{i+1}))\end{equation} and \begin{equation}\tU_{i+1/2}=Spf(\C[X_i,Y_{i+1}][[q]]/(X_iY_{i+1}-q))\end{equation} respectively. In the latter, the formal spectrum is taken with respect to $q$-adic topology. Let \begin{equation}\label{vitilde}
\tV_i:=\tU_{i-1/2}\cap\tU_{i+1/2}=\bar V_i\times_{Spec(\C[q])} Spf(\C[[q]])
\end{equation}
\begin{notation}
Let $j_{U_{i+1/2}}$, resp. $j_{V_i}$ denote the inclusion of the open set $U_{i+1/2}$ resp. $V_i$. Similarly, let $j_{\tU_{i+1/2}}$ and $j_{\tV_i}$ denote the open inclusions into $\Tt_R$.
\end{notation}
\begin{figure}
	\centering
	\includegraphics[height=2cm]{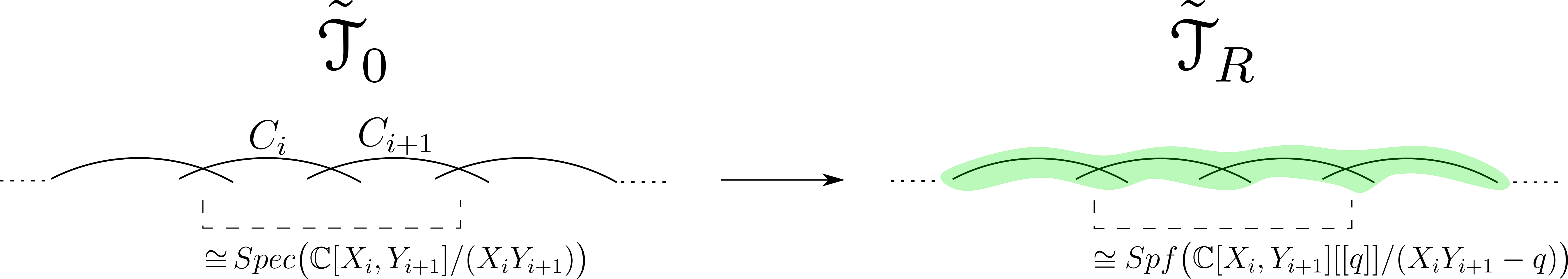}
	\caption{The inclusion of $\Tt_0$ into $\Tt_R$}
	\label{figure:incltate}
\end{figure}
\begin{rk}
It is easy to see that $\Tt_0$ is an infinite chain of projective lines. Let $C_i$ denote the projective line given as the union of $\{(X_{i-1},Y_i)\in U_{i-1/2}: X_{i-1}=0  \}$ and $\{(X_i,Y_{i+1})\in U_{i+1/2}: Y_{i+1}=0  \}$. Its affine charts have coordinates $X_i$ and $Y_i$ satisfying $X_iY_i=1$ on the overlap $V_i:=\bar{V}_i|_{q=0}\subset C_i$. See Figure \ref{figure:incltate}.
\end{rk}
\begin{defn}Define the \textbf{translation automorphism} on $\Tt_R$, resp. $\Tt_0$ to be the automorphism given by the local transformations $\tU_{i-1/2}\rightarrow \tU_{i+1/2}$, resp. $U_{i-1/2}\rightarrow U_{i+1/2}$ given by
	\begin{equation}X_i\mapsto X_{i-1}, Y_{i+1}\mapsto Y_i \end{equation} on the coordinate rings. Denote both of them by $\tr$.
\end{defn}
\begin{rk}\label{actionremark}
By restricting the $\mathbb{G}_{m,\C[q]}$-action in (\ref{eq:firstaction}) along $Spf(R)\rightarrow Spec(\C[q])$, we obtain an action of $\widehat{\G}:=Spf(\C[t,t^{-1}][[q]])$ on $\Tt_R$ in the category of formal schemes over $R$. Similarly, by restricting the $\mathbb{G}_{m,\C[q]}$-action along $0:Spec(\C)\rightarrow Spec(\C[q])$, we obtain an action of $\G:=\mathbb{G}_{m,\C}$ on $\Tt_0$ in the category of schemes over $\C$.
\end{rk}
\subsection{Multiplication graph of $\Tt_R$}
Raynaud's insight provided a picture of (some) rigid analytic spaces over $\C((q))$ as generic fibers of formal schemes over $\C[[q]]$. In this point of view, the analytification of $\mathbb{G}_{m,\C((q))}$ can be obtained as the generic fiber of $\Tt_R$. But, the analytification $\mathbb{G}_{m,\C((q))}^{an}$ is a group and this suggests finding a morphism of formal schemes \begin{equation}\Tt_R\times\Tt_R\rightarrow \Tt_R \end{equation} giving the group multiplication \begin{equation}\mathbb{G}_{m,\C((q))}^{an}\times\mathbb{G}_{m,\C((q))}^{an}\rightarrow\mathbb{G}_{m,\C((q))}^{an}\end{equation} in the generic fiber. This could be possible after admissible blow-ups on the special fiber of $\Tt_R\times\Tt_R$, but instead, we will write an explicit formal subscheme of $\Tt_R\times\Tt_R\times\Tt_R$ over $Spf(R)$, which presumably gives the graph of multiplication when the generic fiber functor is applied. We emphasize that we will not show this and there will be no formal references to Raynaud's view or to rigid analytic spaces, as it is not needed for our purposes. Interested reader may see \cite{bosch} or \cite{Temkin2015}) for more details. 

\begin{defn}\label{largegraph}
Let $\mathcal{G}_{l,R}$ be the formal subscheme of $\Tt_R\times\Tt_R\times\Tt_R$ locally given by the following equations
\begin{equation}\label{tabular1}
\begin{tabular}{lr}
	$Y_i(1)Y_j(2)=Y_{i+j}(3),$ & $Y_j(2)X_{i+j}(3)=X_i(1)$ \\
	$Y_i(1)X_{i+j}(3)=X_j(2),$  & $Y_j(2)=X_i(1)Y_{i+j}(3)$ \\
	$Y_i(1)=X_j(2)Y_{i+j}(3),$  & $X_{i+j}(3)=X_i(1)X_j(2)$
\end{tabular} 
\end{equation}
and by the equations\begin{equation}\label{tabular2}
\begin{tabular}{lr}
$Y_i(1)Y_j(2)X_{i+j}(3)=1$&$X_i(1)X_j(2)Y_{i+j}(3)=1$ 
\end{tabular}
\end{equation}
Here, $X_i(1), Y_i(1)$ are the local coordinates of the first component, $X_i(2), Y_i(2)$ are of the second and $X_i(3), Y_i(3)$ are of the third. For fixed $i$ and $j$, each of these equations make sense only on one chart of type $\tU_{k+1/2}\times\tU_{l+1/2}\times\tU_{m+1/2}$. Hence, $\mathcal{G}_{l,R}$ is the formal subscheme given on the chart $\tU_{k+1/2}\times\tU_{l+1/2}\times\tU_{m+1/2}$ by all equations listed in (\ref{tabular1}) and (\ref{tabular2})  (as $i$ and $j$ varies) that make sense on this chart. If none of these makes sense (i.e. for all equations as above there is at least one local coordinate involved in the equation and that is not defined on the chart), we take the subscheme to be empty on that chart.
\end{defn}
\begin{exmp}
For instance $Y_i(1)Y_j(2)=Y_{i+j}(3)$ makes sense on $\tU_{i-1/2}\times\tU_{j-1/2}\times\tU_{i+j-1/2}$ and $Y_j(2)X_{i+j}(3)=X_i(1)$ makes sense on $\tU_{i+1/2}\times\tU_{j-1/2}\times\tU_{i+j+1/2}$. The other equations that make sense on $\tU_{i-1/2}\times\tU_{j-1/2}\times\tU_{i+j-1/2}$ are $Y_j(2)X_{i+j-1}(3)=X_{i-1}(1)$ and $Y_i(1)X_{i+j-1}(3)=X_{j-1}(2)$.
\end{exmp}
\begin{rk}
There is an $S_3$-symmetry of the coordinates preserving equations, which would become more obvious after the coordinate change \begin{equation}X_i(3)\leftrightarrow Y_{-i}(3),X_{-i}(3)\leftrightarrow Y_{i}(3) \end{equation} After the coordinate change, the symmetry is given by permuting the components of $\Tt_R\times \Tt_R\times \Tt_R$.
\end{rk}
We still need to check:
\begin{lem}
Equations (\ref{tabular1}) and (\ref{tabular2}) give a well-defined formal subscheme of $\Tt_R\times\Tt_R\times\Tt_R$.
\end{lem}
\begin{proof}We need to check the formal subschemes match in the intersections of charts $\tU_{k'+1/2}\times\tU_{l'+1/2}\times\tU_{m'+1/2}$ and $\tU_{k''+1/2}\times\tU_{l''+1/2}\times\tU_{m''+1/2}$. Assuming the intersection is non-empty and charts are different, we see that $k'\neq k''$, $l'\neq l''$ or $m'\neq m''$. Without loss of generality assume $l'\neq l''$, $l'=-1$ and $l''=0$. Hence, their intersection lives inside 
\begin{equation}
\Tt_R\times \tV_0\times \Tt_R=\Tt_R\times (\tU_{-1/2}\cap\tU_{1/2})\times \Tt_R
\end{equation}
Notice that the intersection of the subscheme defined on a specific chart $\tU_{k+1/2}\times\tU_{l+1/2}\times\tU_{m+1/2}$ with $\Tt_R\times \tV_0\times \Tt_R$ is the same as the graph of the action \begin{equation}\Tt_R\times Spf(\C[t,t^{-1}][[q]])\rightarrow \Tt_R \end{equation} intersected with that chart. The action is still locally given by \begin{equation}Y_{i+1}\mapsto tY_{i+1} \text{ and } X_i\mapsto t^{-1}X_i\end{equation} and we identify $\tV_0$ with $Spf(\C[t,t^{-1}][[q]])$ by putting $t=Y_0$.

Hence, the restriction of the graphs defined on $\tU_{k'+1/2}\times\tU_{l'+1/2}\times\tU_{m'+1/2}$ or $\tU_{k''+1/2}\times\tU_{l''+1/2}\times\tU_{m''+1/2}$ can be obtained by restricting the graph of the action above to $(\tU_{k'+1/2}\times\tU_{l'+1/2}\times\tU_{m'+1/2}) \cap(\tU_{k''+1/2}\times\tU_{l''+1/2}\times\tU_{m''+1/2})$. This implies they are the same.
\end{proof}
We will confine ourselves to $\cG_{l,R}\cap \Tt_R\times \tU_{-1/2}\times\Tt_R$. Put $u=X_{-1},t=Y_0$ and put $X_i=X_i(1),X_i'=X_i(3),Y_{i+1}=Y_{i+1}(1),Y_{i+1}'=Y_{i+1}(3)$. Moreover, we interchange the second and third coordinates to obtain a formal subscheme  $\cG_R\subset\Tt_R\times\Tt_R\times Spf(\C[u,t][[q]]/(ut-q))$, where the formal spectrum is taken with respect to $q$-adic topology. The topological algebra $\C[u,t][[q]]/(ut-q)$ will appear recurrently, so let us name it:
\begin{notation}$A_R:=\C[u,t][[q]]/(ut-q)$ with its $q$-adic topology and $A:=\C[u,t]/(ut)$.
\end{notation}
$\cG_R$ is given by the equations 
\begin{equation}\label{eq:graph1}
tY_{i+1}=Y_{i+1}',tX_i'=X_i,Y_{i+1}X_i'=u\text{ on }\tU_{i+1/2}\times\tU_{i+1/2}\times Spf(A_R)
\end{equation}
\begin{equation}\label{eq:graph2}
Y_{i+1}=uY_{i}',X_{i-1}'=uX_i,Y_{i}'X_i=t\text{ on }\tU_{i+1/2}\times\tU_{i-1/2}\times Spf(A_R) 
\end{equation}
\begin{rk}
Equations (\ref{eq:graph1}) and (\ref{eq:graph2}) are merely translations of the equations (\ref{tabular1}) into new variables, and the equations (\ref{tabular2}) are not needed for the definition. $\cG_R$ is covered by its open subschemes defined in (\ref{eq:graph1}) and (\ref{eq:graph2}).
\end{rk}
\begin{lem}\label{flatsub}
$\cG_R$ is flat over $A_R=\C[u,t][[q]]/(ut-q)$.\end{lem}
\begin{proof}
We show this only for the formal subscheme of $\tU_{i+1/2}\times\tU_{i+1/2}\times Spf(A_R)$ defined by (\ref{eq:graph1}). The part defined by (\ref{eq:graph2}) is similar. 

Notice the equations $tY_{i+1}=Y_{i+1}',tX_i'=X_i,Y_{i+1}X_i'=u$ define a subscheme of \begin{equation}\C[X_i,Y_{i+1}][q]/(X_iY_{i+1}-q)\times_{\C[q]}\C[X'_i,Y'_{i+1}][q]/(X'_iY'_{i+1}-q)\atop\times_{\C[q]}\C[u,t][q]/(ut-q) \end{equation} whose formal completion along $q=0$ gives (part of) $\cG_R$. Indeed, it is isomorphic to the subscheme of $Spec(\C[X_i,Y_{i+1},X_i',Y_{i+1}',u,t])$ given by the same equations(equations (\ref{eq:graph1}) imply $X_iY_{i+1}=X_i'Y_{i+1}'=ut$). As $tY_{i+1}=Y_{i+1}'$ and $tX_i'=X_i$, we can see it as the subscheme of $Spec(\C[Y_{i+1},X_i',u])\times Spec(\C[t])$ given by the equation $Y_{i+1}X_i'=u$. $Spec(\C[Y_{i+1},X_i',u]/(Y_{i+1}X_i'-u))$ is flat over $\C[u]$; hence, $Spec(\C[Y_{i+1},X_i',u,t]/(Y_{i+1}X_i'-u))$ is flat over $\C[u,t]\cong\C[u,t][q]/(ut-q)$ and so is its formal completion along $q=0$.
\end{proof}
\begin{rk}
In the same way, we can show $\cG_{l,R}$ is flat with respect to all three projections to $\Tt_R$. 
\end{rk}
\begin{notation}
Let $\cG:=\cG_R|_{q=0}\subset \Tt_0\times\Tt_0\times Spec(A)$. It follows from Lemma \ref{flatsub} that $\cG$ is flat over $A$.
\end{notation}
\section{A dg model for the universal cover of the Tate curve}\label{sec:dgmodel}
\subsection{The dg model $\Om(\Tt_0)_{dg}$}\label{subsec:dgmint}
In this section we construct a dg category $\Om(\Tt_0)_{dg}$ such that \begin{equation}D^\pi(\Om(\Tt_0)_{dg})\simeq D^b(Coh_p(\Tt_0))\end{equation} where $Coh_p(\Tt_0)$ is the abelian category of properly supported coherent sheaves $\Tt_0$. We will take $Ob(\Om(\Tt_0)_{dg}):=\{\Om_{C_i}(-1),\Om_{C_i}:i\in \Z \}$, where $\Om_{\C_i}$ denotes the structure sheaf of the closed subvariety $C_i$ and $\Om_{C_i}(-1)$ denotes the structure sheaf twisted by a smooth point on $C_i$(it does not matter which). First we show
\begin{lem}
$\{\Om_{C_i}(-1),\Om_{C_i}:i\in \Z \}$ generates $D^b(Coh_p(\Tt_0))$ as a triangulated category. 
\end{lem} 
\begin{proof}
It is enough to show that every $\scrF \in Coh_p(\Tt_0)$ is in the full subcategory generated by $\{\Om_{C_i}(-1),\Om_{C_i}:i\in \Z \}$. Let $i_n: C_n\rightarrow \Tt_0$ denote the inclusion for a given $n\in \Z$. Consider $\scrF\rightarrow i_{n*} i_n^* \scrF$, where $i_n^*$ refers to ordinary (not derived) pull-back. Note, $i_{n*}$ does not need to be derived as $i_n$ is affine. The sheaf $i_{n*} i_n^* \scrF$ is in the image $i_{n*}(D^b(Coh(C_n)))$, which is generated by $\Om_{C_n},\Om_{C_n}(-1)$ as $C_n\cong \mathbb{P}^1$, $\Om_{C_n}= i_{n*}\Om_{\mathbb{P}^1}$ and $\Om_{C_n}(-1)= i_{n*}\Om_{\mathbb{P}^1}(-1)$. Hence, to finish the proof, we only need to show the kernel and the cokernel of the map $\scrF\rightarrow \bigoplus_{n\in\Z} i_{n*} i_n^* \scrF$ are in this category. But, both the kernel and the cokernel are finite direct sums of coherent sheaves supported on the nodes. Any such coherent sheaf can be filtered so that the subquotients are isomorphic to the structure sheaves of the nodes. Hence, they can be seen as iterated extensions of the structure sheaves of the nodal points, and the structure sheaf of the node is in the essential image $i_{n*}(D^b(Coh(C_n)))$ (as the cokernel of a map $\Om_{C_n}(-1)\rightarrow \Om_{C_n}$). Hence, they are all in the triangulated subcategory generated by $\{\Om_{C_i}(-1),\Om_{C_i}:i\in \Z \}$.
\end{proof} 
To find an enhancement of $D^b(Coh_p(\Tt_0))$, we will closely follow \cite{lunts}. First some generalities:

Let X be a separated scheme over $\C$, which is locally of finite type. Let $\{U_\alpha\}$ be an open cover, where the index set is ordered. Assume, every quasi-compact subset intersects only finitely many $U_\alpha$. Let $\scrF$ be a sheaf on $X$; and for a given open subset $j:V\hookrightarrow X$ define \begin{equation}^V \scrF:=j_!j^*(\scrF)\end{equation}
Also define
\begin{equation}\mathcal{C}_!(\scrF):= \bigg\{\dots \rightarrow\prod_{\alpha_1<\alpha_2} \big(^{U_{\alpha_1}\cap{U_{\alpha_2}}  } \scrF\big)\rightarrow \prod_{\alpha} \big(^{U_\alpha}\scrF\big) \bigg\}\xrightarrow{\simeq} \scrF \end{equation}
For the differential of this complex and exactness see \cite{lunts} and references there-in. In our situation we will choose a cover so that triple intersections will be empty. The differential is given by maps  \begin{equation}\big(^{U_{\alpha_1}\cap{U_{\alpha_2}}  } \scrF\big)\rightarrow  \big(^{U_{\alpha_1}}\scrF\big)\times \big(^{U_{\alpha_2}}\scrF\big) \end{equation} on the factors, which are the differences of the natural maps $\big(^{U_{\alpha_1}\cap{U_{\alpha_2}}  } \scrF\big)\rightarrow  \big(^{U_{\alpha_i}}\scrF\big)$, $i=1,2$.

Now, assume the $U_\alpha$ are affine and their triple intersections are empty. We will modify the resolutions as follows: for each finite subset $I\subset\{\alpha \}$, fix a free resolution of $j_{U_I}^*\scrF$, where $j_{U_I}$ is the inclusion of $U_I=\bigcap_{\alpha\in I}U_\alpha $. This extends to a double resolution over $\mathcal{C}_!(\scrF)$, where $\mathcal{C}_!(\scrF)$ is assumed to lie in the horizontal direction. Take its total complex to obtain a resolution of $\scrF$ by sums of sheaves of the form $j_!(E)$, where $j: V\rightarrow X$ is an open embedding and $E$ is a vector bundle on $V$. We denote this bounded above complex of $\Om_X$-modules by $R(\scrF)$, suppressing the data of resolutions and maps between them in the notation.

From now on let $X=\Tt_0$ and the covering be $\{U_{i+1/2} \}_{i\in\Z}$. Consider $\Om_{C_i}(a)$, where $i,a\in \Z$. The complex $\mathcal{C}_!(\Om_{C_i}(a))$, as a graded sheaf, is a shifted sum of $j_{V_i,!}\Om_{V_i}$, $j_{U_{i-1/2},!}\Om_{C_i\cap U_{i-1/2}}$ and $j_{U_{i+1/2},!}\Om_{C_i\cap U_{i+1/2}}$. Note that to write it this way, we need to choose trivializations of $\Om_{C_i}(a)|_{V_{i}}$, $\Om_{C_i}(a)|_{U_{i-1/2}}$ and $\Om_{C_i}(a)|_{U_{i+1/2}}$. Choose them together so that $\tr$ moves the trivializations for $\Om_{C_i}(a)$ to these for $\Om_{C_i+1}(a)$. Under the natural isomorphism \begin{equation}U_{i+1/2}\cong Spec(\C[X_i,Y_{i+1}]/(X_iY_{i+1})) \end{equation} $\Om_{C_i\cap U_{i+1/2}}$ corresponds to the module $\C[X_i,Y_{i+1}]/(X_iY_{i+1},Y_{i+1})$. Similarly, $\Om_{C_i\cap U_{i-1/2}}$ corresponds to $\C[X_{i-1},Y_{i}]/(X_{i-1}Y_{i},X_{i-1})$. Let the free resolution of $\Om_{V_i}$ be the trivial one. Also, let the other resolutions be \begin{equation}\dots\xrightarrow{Y_{i}}\Om(U_{i-1/2})\xrightarrow{X_{i-1}}\Om(U_{i-1/2})\rightarrow\Om(U_{i-1/2})/(X_{i-1})\end{equation} \begin{equation}\dots\xrightarrow{X_i}\Om(U_{i+1/2})\xrightarrow{Y_{i+1}}\Om(U_{i+1/2})\rightarrow\Om(U_{i+1/2})/(Y_{i+1})\end{equation} The only non-zero horizontal arrow in the double resolution is \begin{equation}j_{V_i,!}\Om_{V_i}\rightarrow j_{U_{i-1/2},!}\Om_{U_{i-1/2}}\times j_{U_{i+1/2},!}\Om_{U_{i+1/2}} \end{equation} lifting \begin{equation}j_{V_i,!}\Om_{C_i\cap V_i}\rightarrow j_{U_{i-1/2},!}\Om_{C_i\cap U_{i-1/2}}\times j_{ U_{i+1/2},!}\Om_{C_i\cap U_{i+1/2}} \end{equation} It is determined by an element in $\C[X_i,X_i^{-1}]\times \C[X_i,X_i^{-1}]$. Choose the horizontal arrows simultaneously for all $i$ so that they are compatible with $\tr$, in the sense above (i.e. the chosen arrows for $C_i$ will move to $C_{i+1}$ under $\tr$).

In summary, applying the above procedure of finding double resolutions and totalizations, we find complexes of sheaves $R(\scrF)$ supported in non-positive degree and quasi-isomorphisms \begin{equation}R(\scrF)\xrightarrow{\simeq} \scrF \end{equation} 
\begin{defn} Let $\Om(\Tt_0)_{dg}$ be the full dg subcategory of complexes of $\Om_{\Tt_0}$-modules that is spanned by objects $R(\Om_{C_i}(-1))$ and $R(\Om_{C_i})$. We will denote these objects by $\Om_{C_i}(-1)$ and $\Om_{C_i}$ as well. 
\end{defn}
\begin{prop}
$tw^\pi(\Om(\Tt_0)_{dg})$ is a dg-enhancement of $D^b(Coh_p(\Tt_0))$.
\end{prop}
\begin{proof}
First, start by noting that $D^b(Coh_p(\Tt_0))$ is equivalent to $D^b_{coh,p}(\Om_{\Tt_0})$, the full subcategory of $D^b(\Om_{\Tt_0})$ spanned by objects whose hypercohomology sheaves are in $Coh_p(\Tt_0)$. This can be shown using \cite[Cor 3.4,Prop 3.5]{huybFM} and the fact that $D^b(Coh_p(\Tt_0))$ is a union of subcategories equivalent to derived categories of properly supported coherent sheaves on open Noetherian subschemes of $\Tt_0$.
Hence, we will actually work with the latter category.

We need to show the natural map \begin{equation}Hom_{K(\Om_{\Tt_0})}(R(\scrF),R(\scrF'))\rightarrow Hom_{D(\Om_{\Tt_0})}(R(\scrF),R(\scrF')) \end{equation} is an isomorphism. Here, $\scrF$ and $\scrF'$ are among $\{\Om_{C_i},\Om_{C_i}(-1):i\in\Z \}$ and $K(\Om_{\Tt_0})$ denotes the homotopy category of complexes of $\Om_{\Tt_0}$-modules. 

First note \begin{equation}\label{eq:dib}
Hom_{K(\Om_{\Tt_0})}(R(\scrF),\scrF')\simeq Hom_{D(\Om_{\Tt_0})}(R(\scrF),\scrF') 
\end{equation}
To see this choose a resolution $\scrF'\rightarrow I^\cdot$ by quasi-coherent sheaves that are injective as $\Om_{\Tt_0}$-modules. Then we know (see \cite[Tag 070G]{stacks-project}) \begin{equation}Hom_{D(\Om_{\Tt_0})}(R(\scrF),\scrF')\simeq Hom_{K(\Om_{\Tt_0})}(R(\scrF),I^\cdot) \end{equation} 
To show (\ref{eq:dib}), we only need the hom complex \begin{equation}hom^\cdot(R(\scrF),\scrF'\rightarrow I^\cdot ) \end{equation} to be acyclic. But this is the totalization of a double complex supported on bidegrees that is in a fixed translate of the first quadrant. 
Moreover, the rows of this double complex are shifted direct sums of complexes of type $hom^\cdot(j_!(E),\scrF'\rightarrow I^\cdot)\simeq hom^\cdot(E,j^*(\scrF'\rightarrow I^\cdot) )$, where $j$ is the open embedding of either $U_{i+1/2}$ or $V_i$ for some $i$, and $E$ is a vector bundle on it. 
Hence, the rows are acyclic and (\ref{eq:dib}) follows from the spectral sequence for the double complex. 

Hence, we only need to show \begin{equation}Hom_{K(\Om_{\Tt_0})}(R(\scrF),\scrF')\simeq Hom_{K(\Om_{\Tt_0})}(R(\scrF),R(\scrF') )\end{equation} or equivalently $hom^\cdot(R(\scrF),R(\scrF')\rightarrow\scrF')$ is acyclic. By Lemma \ref{homal} below the acyclicity of $hom^\cdot(j_!(E),R(\scrF')\rightarrow\scrF')$ is enough, where $j$ and $E$ are as in the above paragraph. Let $U$ denote the domain of $j$.

Without loss of generality, assume $E$ is the trivial line bundle on $U$. By the adjunction $j_!\vdash j^*$ \begin{equation}hom^\cdot(j_!(E),R(\scrF')\rightarrow\scrF')\simeq hom^\cdot(E,j^*(R(\scrF')\rightarrow\scrF'))\simeq\Gamma(j^*(R(\scrF')\rightarrow\scrF')) \end{equation} When $U=V_i$, $j^*(R(\scrF')\rightarrow\scrF')$ is an acyclic complex of coherent sheaves on $U$ and $\Gamma$, the global sections functor, preserves its acyclicity. When $U=U_{i+1/2}$, $\Gamma(j^*R(\scrF'))$ can be obtained as the totalization of a double complex resolving the complex $\Gamma (j^*\mathcal{C}_!(\scrF')  )$, whose explicit form is \begin{align*}\{\Gamma(j_{V_{i},!}j_{V_{i}}^*\scrF' )\times \Gamma(j_{V_{i+1},!}j_{V_{i+1}}^*\scrF' )\rightarrow\\  \Gamma(j_{V_{i},!}j_{V_{i}}^*\scrF' )\times \Gamma(j_{V_{i+1},!}j_{V_{i+1}}^*\scrF' )\times \Gamma(j_{U_{i+1/2},!}j_{U_{i+1/2}}^*\scrF' )\}=\\ \{\Gamma(j_{U_{i+1/2},!}j_{U_{i+1/2}}^*\scrF' )\}=\{\Gamma(j_{U_{i+1/2}}^*\scrF' )\} \end{align*}
The equation holds as $\Gamma(j_{V_{i},!}j_{V_{i}}^*\scrF' )=0$ for all $i$(which is true since $j_{V_i}^*\scrF'$ is locally free and $U_{i+1/2}$ is connected). 
This is still a resolution of $\Gamma(j_{U_{i+1/2}}^*\scrF' )$. Hence, being the totalization of a double complex resolving $\Gamma (j^*\mathcal{C}_!(\scrF')  )$, $\Gamma(j^*R(\scrF') )$ is another resolution of $\scrF'$ and \begin{equation}\Gamma(j^*(R(\scrF')\rightarrow \scrF' ) ) \end{equation} is an acyclic complex. This finishes the proof.
\end{proof}
\begin{lem}\label{homal}Let $C^\cdot$, $D^\cdot$ be bounded above complexes of objects of an abelian category. Assume for each $i$, $hom^\cdot(C^i,D^\cdot)$ is acyclic. Then the total hom complex $hom^\cdot (C^\cdot,D^\cdot)$ is also acyclic.\end{lem}
\begin{rk} $\tr_*$ gives an explicit dg quasi-equivalence of $\Om(\Tt_0)_{dg}$ that acts bijectively on the objects and hom-sets. We denote this dg functor by $\tr$ as well.  
\end{rk}
\begin{rk}
The complex $hom_{\Om(\Tt_0)_{dg}}^\cdot(\Om_{C_i}(a),\Om_{C_{i'}}(b))=0$, for $|i-i'|\geq 2$ and $a,b\in\{0,1\}$. Indeed, if $j_!(E)\neq 0$ and $j'_!(E')\neq 0$ appear in $R(\Om_{C_i}(a))$ and $R(\Om_{C_{i'}}(b))$ respectively, there is no way the domain of $j$ or $j'$ can contain the domain of the other; hence $Hom_{\Om_{\Tt_0}}(j_!(E),j'_!(E'))=0.$
\end{rk}
\subsection{$\G$-action on $\Om(\Tt_0)_{dg}$}\label{subsec:action}
Let $\scrF\in\{\Om_{C_i},\Om_{C_i}(-1):i\in\Z \}$. Put a $\G$-equivariant structure on $\scrF$. This makes every graded piece of $\mathcal{C}_!(\scrF)$ naturally a $\G$-equivariant sheaf, and the differential is $\G$-equivariant. Moreover, the double complex resolving it can be made $\G$-equivariant as well in each bidegree, so that both differentials are $\G$-equivariant.
Hence, $R(\scrF)\rightarrow\scrF$ is an equivariant resolution. Fix choices for each $i$ so that $\tr_*$ moves $\Om_{C_i}$ to $\Om_{C_{i+1}}$ as an equivariant sheaf and similarly for $\Om_{C_i}(-1)$ as well as the resolutions. Hence, we obtain an action of $\G$ on hom-sets of $\Om(\Tt_0)_{dg}$, so that the differential and the multiplication are equivariant. In other words, there exists a $\G$-action at the chain level on this category. 

Note, however the hom-sets $hom^\cdot_{\Om_{\Tt_0}}(R(\scrF),R(\scrF'))$ are not rational as representations of $\G$. Instead, they are products of countably many rational representations at each degree. Inspired by this define:
\begin{defn}
Let $\Om(\Tt_0)_{dg}^{eval}$ be the dg-subcategory of $\Om(\Tt_0)_{dg}$ with the same set of objects and with the morphisms given by the subspace of those in $\Om(\Tt_0)_{dg}$ that decompose into a finite sum of eigenvalues of $\G$-action. 
\end{defn}
\begin{prop}\label{eval}The inclusion $\Om(\Tt_0)_{dg}^{eval}\rightarrow \Om(\Tt_0)_{dg}$ is a quasi-equivalence.
\end{prop}
This follows from a simple lemma whose proof we skip:
\begin{lem}Let $(C^\cdot,d)$ be a chain complex satisfying
\begin{itemize}
	\item There is a $\G$-action on each $C^i$ and $d$ is equivariant 
	\item The induced action on $H^\cdot(C^\cdot)$ is rational, i.e. $H^\cdot(C^\cdot)$ admits a direct sum decomposition into eigenvalues of $\G$-action
	\item For each $i$, $C^i$ has a product decomposition \begin{equation}C^i=\prod_{k\in\Z} C^i\{k\}\end{equation} into rational representations, such that $d:C^i\rightarrow C^{i+1}$ is a product of equivariant maps \begin{equation}d_k:C^i\{k\}\rightarrow C^{i+1}\{k\}\times C^{i+1}\{k+1\}\end{equation}
\end{itemize}
Let $C^\cdot_{eval}$ be the subcomplex of $C^\cdot$ spanned by eigenvalues of $\G$-action. Then the inclusion $C^\cdot_{eval}\rightarrow C^\cdot$ is a quasi-isomorphism. 
\end{lem}
\subsection{A deformation of $\Om(\Tt_0)_{dg}$}\label{subsec:defoofom}
We have constructed a deformation of $\Tt_0$ in Section \ref{subsec:Tate}. In this subsection, we will use it to obtain a deformation of $\Om(\Tt_0)_{dg}$ to a curved $A_\infty$-category, which we denote by $\Om(\Tt_R)_{cdg}$. We will manage this by deforming the double complexes in Section \ref{subsec:dgmint} whose totalizations give $R(\Om_{C_i}(-1))$ and $R(\Om_{C_i})$. We will deform them to bigraded sheaves of $\Om_{\Tt_R}$-modules with two endomorphisms, one of degree $(1,0)$ and other of degree $(0,1)$ (in other words, we deform them to objects that look like double complexes except the differentials do not square to $0$).

First a local model: Consider the resolutions  \begin{equation}\{\dots\xrightarrow{Y_{i}}\Om(U_{i-1/2})\xrightarrow{X_{i-1}}\Om(U_{i-1/2})\}\rightarrow\Om(U_{i-1/2})/(X_{i-1})\end{equation} \begin{equation}\{\dots\xrightarrow{X_i}\Om(U_{i+1/2})\xrightarrow{Y_{i+1}}\Om(U_{i+1/2})\}\rightarrow\Om(U_{i+1/2})/(Y_{i+1})\end{equation} and deform them to ``complexes" of $\Om_{\Tt_R}$-modules given as \begin{equation}\{\dots\xrightarrow{X_{i-1}}\Om(\tU_{i-1/2})\xrightarrow{Y_{i}}\Om(\tU_{i-1/2})\xrightarrow{X_{i-1}}\Om(\tU_{i-1/2})\}\end{equation} \begin{equation}\{\dots\xrightarrow{Y_{i+1}}\Om(\tU_{i+1/2})\xrightarrow{X_{i}}\Om(\tU_{i+1/2})\xrightarrow{Y_{i+1}}\Om(\tU_{i+1/2})\}\end{equation}The ``differentials" do not square to $0$ as $X_{i-1}Y_i=q$ and $X_iY_{i+1}=q$ in the corresponding rings. 

This gives the data to deform the vertical differentials of the double complexes resolving $\mathcal{C}_!(\Om_{C_i}(-1))$ and $\mathcal{C}_!(\Om_{C_i})$. Deform the horizontal differential to  \begin{equation}j_{\tV_i,!}\Om_{\tV_i}\rightarrow j_{\tU_{i-1/2},!}\Om_{\tU_{i-1/2}}\times j_{\tU_{i+1/2},!}\Om_{\tU_{i+1/2}} \end{equation} trivially. Let $R(\Om_{C_i}(-1))_R$ and $R(\Om_{C_i})_R$ denote the totalizations of these bigraded sheaves with degree $(0,1)$ and $(1,0)$ endomorphisms. They are graded sheaves with degree $1$ endomorphisms, which squares to a degree $2$ endomorphism that is a multiple of $q\in R=\C[[q]]$.

\begin{defn} Let $\Om(\Tt_R)_{cdg}$ be the curved dg category given by \begin{itemize}
\item $Ob(\Om(\Tt_R)_{cdg})=\{\Om_{C_i}(-1),\Om_{C_i}:i\in\Z \}$
\item $hom_{\Om(\Tt_R)_{cdg}}(\scrF,\scrF'):=hom^\cdot_{\Om_{\Tt_R}}(R(\scrF)_R,R(\scrF')_R)$ for $\scrF,\scrF'\in Ob(\Om(\Tt_R)_{cdg})$. The hom-``complex" is defined in the standard way similar to actual complexes, only note its differential does not square to $0$
\item The composition is composition of homomorphisms of ``complexes''
\item The curvature term is the degree $2$ endomorphism obtained by squaring the differential of $R(\Om_{C_i}(-1))_R$ and $R(\Om_{C_i})_R$
\end{itemize}
\end{defn}
It is easy to see that this is a curved dg category over $R=\C[[q]]$. For instance, the square of the differential of $hom_{\Om(\Tt_R)_{cdg}}(\scrF,\scrF')$ is simply the difference of composition with the differentials of $R(\scrF)_R$ and $R(\scrF')_R$. It is also obvious that the specialization to $q=0$ gives $\Om(\Tt_0)_{dg}$.

We now want to elaborate on the compatibility of this formal deformation with the geometric deformation above. We show ``local" compatibility.

In general, if $\B$ is an algebra and $\B_R$ is a deformation of $\B$ over $R$, then we obtain a curved deformation of the category $\B^{mod}$ (the category of $A_\infty$-modules over $\B$). The deformation is the category of curved modules, where a curved module is defined by the same data as a semi-free $A_\infty$-module over $\B_R$ but the $A_\infty$-module equation is satisfied only up to $O(q)$. Hence, we obtain a deformation of the category of finitely generated modules as a subcategory of the deformation of $\B^{mod}$.

Assume $\B_R$ is commutative and apply this to $U=Spec(\B)$ and to $Spf(\B_R)$. This way we obtain a recipe to produce formal deformations of (generating $A_\infty$-models of) $D^b(Coh(U))$ such that $\Om_U$ deforms to $\Om_{U_R}$ Thus, the inclusion functor from the full subcategory spanned by $\Om_U$ deforms to an $A_\infty$-functor from the algebra $\Om_{U_R}$. Call such a deformation a good deformation.

Now our compatibility result is:
\begin{prop}\label{limdef} For each $i\in\Z$, there exists 
\begin{itemize}
\item A dg enhancement $\scrC oh(U_{i+1/2})$ of $D^b(Coh(U_{i+1/2}))$
\item A good deformation $\scrC oh(U_{i+1/2})_R$ of $\scrC oh(U_{i+1/2})$
\item A dg enhancement $\scrC oh(V_{i})$ of $D^b(Coh(V_{i}))$
\item $A_\infty$-functors $j_{U_{i+1/2}}^*:\Om(\Tt_R)_{cdg}\rightarrow \scrC oh(U_{i+1/2})_R $
\item $A_\infty$-functors $j_{U_{i+1/2},V_i}^*:\scrC oh(U_{i+1/2})_R\rightarrow \scrC oh(V_{i})_R$, $j_{U_{i-1/2},V_i}^*:\scrC oh(U_{i-1/2})_R\rightarrow \scrC oh(V_{i})_R$ where $\scrC oh(V_{i})_R$ is the trivial deformation of $\scrC oh(V_{i})$
\end{itemize} such that at $q=0$, $j_{U_{i+1/2}}^*$ specializes to a lift of the natural functor \begin{equation}j_{U_{i+1/2}}^*:D^b(Coh_p(\Tt_0))\rightarrow D^b(Coh(U_{i+1/2}))\end{equation}  and similarly $j_{U_{i+1/2},V_i}^*$ and $j_{U_{i-1/2},V_i}^*$. Moreover, everything can be chosen in a $\tr$-equivariant way.
\end{prop}
This proposition can be proven using constructions similar to these in Section \ref{subsec:dgmint} and it will be useful in order to write localization maps for Hochschild cohomology. These maps will be written as deformations of maps induced by restriction functors in Section \ref{sec:hoch}.
\begin{rk}
The deformations of $j_{U_{i\pm 1/2}}^*$ and $j_{U_{i\pm 1/2},V_i}^*$ in Prop \ref{limdef} can be chosen so that \begin{equation}j_{U_{i+ 1/2},V_i}^*\circ j_{U_{i+ 1/2}}^*\simeq j_{U_{i- 1/2},V_i}^*\circ j_{U_{i- 1/2}}^* \end{equation}
\end{rk}
\begin{rk}
This deformation is compatible with $\tr$ and there is an obvious strict auto-equivalence acting on $\Om(\Tt_R)_{cdg}$. This auto-equivalence deforms the translation auto-equivalance of $\Om(\Tt_0)_{dg}$. We denote it by $\tr$ as well. 
\end{rk}
\begin{rk}\label{actionrkondefoOm}The hom-sets of $\Om(\Tt_R)_{cdg}$ are graded complete vector spaces over $R$ 
and there is an action of $\G(R)=R^*$ on hom-sets deforming the action in Section \ref{subsec:action}. Moreover, the completed base change of $\Om(\Tt_0)^{eval}_{dg}$ to $R/\C$ is a non-full curved dg subcategory, inheriting the curved dg category structure from $\Om(\Tt_R)_{cdg}$. We denote it by $\Om(\Tt_R)_{cdg}^{eval}$. Its inclusion into $\Om(\Tt_R)_{cdg}$ clearly deforms the inclusion  $\Om(\Tt_0)_{dg}^{eval}\rightarrow \Om(\Tt_0)_{dg}$, which is a quasi-equivalence by Prop \ref{eval}. It is clear that for all $\scrF$, $\scrF'$ and $j$, $hom^j_{\Om(\Tt_R)_{cdg}^{eval}}(\scrF,\scrF')$ is a completed rational representation of $\G(R)=R^*$, i.e. it is the $q$-adic completion of a representation of $R^*$ with an eigenvalue decomposition.
\end{rk}
\section{The construction of the mapping torus}\label{sec:construction}
\subsection{Smash products and the construction}
In this section, we define the mapping torus category and its canonical deformation associated to a pair $(\A,\phi)$. Let us first remind the reader of smash products:
\begin{defn}\label{smash}
Let $\B$ be a dg category and $G$ be a discrete group. Assume $G$ acts on $\B$ by auto-equivalences that are bijective on $Ob(\B)$ and on hom-sets. Moreover, assume composition of the auto-equivalences associated to $g_1,g_2\in G$ 
is equal to the auto-equivalence associated to $g_1g_2$.
Define $\B\# G$ to be the dg-category such that
\begin{itemize}
\item $Ob(\B\# G):=Ob(\B)$
\item $hom_{\B\# G}(b_1,b_2):=\bigoplus_{g\in G}hom_{\B}(g(b_1),b_2)$ as a chain complex. We will denote $f\in hom_{\B}(g(b_1),b_2)$ by $f\otimes g$ when it is considered as an element of $hom_{\B\# G}(b_1,b_2)$.
\item $(f'\otimes g')\circ (f\otimes g):=(f'\circ g'(f))\otimes(g'g)$
\end{itemize}
\end{defn}
\begin{rk}
When $\B$ is taken to be an ordinary algebra, Definition \ref{smash} gives the well-known semi-direct product construction. Indeed, it is possible to recover Definition \ref{smash} by applying this construction to the total algebra of $\B$. When $G=\Z$, this construction is known as the orbit category, see \cite{kellerorbitcats}, \cite[Section 4.9]{kellerdg}. For a more general group $G$, if one considers the $G$-action as a diagram of categories, then $\B\# G$ is equivalent to the Grothendieck construction as in \cite{thomasongroth}. 
\end{rk}
\begin{rk}
Under similar assumptions, Definition \ref{smash} generalizes verbatim to curved dg algebras. 
\end{rk}
Let $(\A,\phi)$	be as in Section \ref{sec:intro}, i.e. $\A$ is a dg category satisfying \ref{C1}-\ref{C3} and $\phi$ is a strict  auto-equivalence. Note the conditions \ref{C1}-\ref{C3} are not yet necessary. The auto-equivalence $\tr\otimes \phi$ generates a $\Z$-action on $(\Om(\Tt_0)_{dg}\otimes\A)$ satisfying the assumptions of Definition \ref{smash}. 
\begin{defn}
Define $M_\phi$ to be the dg category $(\Om(\Tt_0)_{dg}\otimes\A)\#\Z$. Similarly, define $M_\phi^R$ to be the curved dg category $(\Om(\Tt_R)_{cdg}\otimes\A)\#\Z$, where the tensor product is over $\C$ (and $q$-completed) and the $\Z$-action is generated by $\tr\otimes\phi$ acting on $\Om(\Tt_R)_{cdg}\otimes\A$.
\end{defn}
\begin{rk}
The tensor product of a curved dg category with an uncurved dg category is defined in a way analogous to tensor product of dg categories. Note the curvature $\mu^0_{\scrF\times a}$ of an element $\scrF\times a\in Ob(\Om(\Tt_R)_{cdg}\otimes\A )=Ob(\Om(\Tt_R)_{cdg})\times Ob(\A )$ is $\mu^0_{\scrF}\otimes 1_{a}$, where $\mu^0_{\scrF}$ is the curvature of $\scrF$.
\end{rk}
\begin{rk}\label{actiononmt1}
The $\G(\C)$ (resp. $\G(R)$) action in Section \ref{subsec:action}(resp. Remark \ref{actionrkondefoOm}) induces an action on $(\Om(\Tt_0)_{dg}\otimes\A)$ (resp. $(\Om(\Tt_R)_{cdg}\otimes\A)$); which is compatible with $\tr\otimes \phi$ as the action on $\Om(\Tt_0)_{dg}$ (resp. $(\Om(\Tt_R)_{cdg}$) is chosen to be compatible with $\tr$. Hence, it descends to an action on $M_\phi$ (resp. $M_\phi^R$). Similarly, this action is not rational (resp. completed rational); however, we can pass to non-full quasi-equivalent (resp. quasi-equivalent at $q=0$) subcategories on which action is rational (resp. completed rational).
\end{rk}
\subsection{Bimodules over $\B\# G$ and over $M_\phi$}\label{sec:bimodgen}
Let us make some general remarks about the dg bimodules over $\B\#G$, where $(\B,G)$ is as in Definition \ref{smash}. Let $G_\Delta=\{(g,g):g\in G \}\subset G\times G$ and consider its action on $\B^e=\B\otimes\B^{op}$. One can then consider modules over $(\B^e)\#G_\Delta$. Concretely, any such module is given by 
\begin{itemize}
\item A $\B$-$\B$ bimodule $\mathfrak{M}$
\item For each $g\in G$, $b_1,b_2\in Ob(\B)$ chain isomorphisms \begin{equation}c_{b_1,b_2}(g):\mathfrak{M}(b_1,b_2)\rightarrow \mathfrak{M}(g(b_1),g(b_2))\end{equation} 
\end{itemize}such that $c_{g'(b_1),g'(b_2)}(g)\circ c_{b_1,b_2}(g')=c_{b_1,b_2}(g\circ g'),c_{b,b}(1_b)=1_{g(b)}$ and satisfying \begin{equation}g(f.m.f')=g(f).g(m).g(f') \end{equation} for any $f'\in hom_\B(b_1,b_2)$,$m\in \mathfrak{M}(b_2,b_3)$,$f\in hom_\B(b_3,b_4)$, where $g(m)$ denotes $c_{b_2,b_3}(g)(m)$.

Now construct the $\B\#G$-$\B\#G$ bimodule $\mathfrak{M}\#G$ as follows 
\begin{itemize}
\item $\mathfrak{M}\#G(b_1,b_2)=\bigoplus_{g\in G}\mathfrak{M}(g(b_1),b_2)$ as a complex. Let $m\otimes g$ denote $m\in \mathfrak{M}(g(b_1),b_2)$ when it is considered as an element of $\mathfrak{M}\#G(b_1,b_2)$
\item Given $g_1,g_2\in G$, $m\in\mathfrak{M}(g_1(b_1),b_2)$, $f\in hom_\B(g_2(b_2),b_3)$ \begin{equation}(f\otimes g_2)(m\otimes g_1)=fg_2(m)\otimes g_2g_1 \end{equation}
\item Given $g_1,g_2\in G$, $f\in hom_\B(g_1(b_1),b_2)$, $m\in\mathfrak{M}(g_2(b_2),b_3)$ \begin{equation}(m\otimes g_2)(f\otimes g_1)=mg_2(f)\otimes g_2g_1 \end{equation}
\end{itemize}
The simplest example is the diagonal bimodule of $\B$. In that case, the process clearly gives the diagonal bimodule of $\B\#G$.

This construction can be seen as a base change under the map \begin{equation}(B^e)\#G_\Delta\rightarrow (\B\# G)^e\cong (\B^e)\#(G\times G)\end{equation}  sending $(b\otimes b')\otimes g\mapsto (b\otimes g)\otimes (b'\otimes g^{-1})\in (\B\# G)^e$, which corresponds to $(b\otimes b')\otimes(g,g)\in (\B^e)\#(G\times G) $. To see this, one may simply prove this construction gives a left adjoint to the restriction map of modules under this map.

Also, note the functoriality of this construction in the dg category of dg bimodules. In particular, it sends exact triangles into exact triangles and quasi-isomorphisms into quasi-isomorphisms. 

To use this to produce bimodules over $M_\phi$, we first need to produce bimodules over $\Om(\Tt_0)_{dg}$ satisfying the above invariance condition. 
\begin{defn}
Given a complex $E$ of $\Om_{\Tt_0\times\Tt_0}$-modules we can define the corresponding $\Om(\Tt_0)_{dg}$-bimodule as \begin{equation}\label{eq:bimoddefnquotientmarks}
\mathfrak{M}_E:(\scrF,\scrF')\mapsto ``RHom_{\Om_{\Tt_0\times\Tt_0}}(q^*(\scrF),p^*(\scrF')\overset{L}{\otimes}_{\Om_{\Tt_0\times\Tt_0}}E )" \end{equation}where $q,p$ are projections to first and second factor respectively. 
\end{defn}
\begin{rk}\label{spalsrk}
To remove the quotation marks in the definition (i.e. to make it more precise), replace $\scrF$ by $R(\scrF)$ and $p^*(\scrF')\overset{L}{\otimes}_{\Om_{\Tt_0\times\Tt_0}}E$ by a (K-)injective resolution $I^\cdot_{\scrF'}$ of $p^*(R(\scrF'))\overset{}{\otimes}_{\Om_{\Tt_0\times\Tt_0}}E$ that is functorial in $\Om(\Tt_0)_{dg}$ (and replace $RHom$ by $Hom$). As we noted, we will often omit the subscripts of tensor product from the notation. To see the existence of such a resolution see $R(\scrF)\mapsto p^*(R(\scrF'))\overset{}{\otimes}_{\Om_{\Tt_0\times\Tt_0}}E$ as a dg functor from $\Om(\Tt_0)_{dg}$ to chains on the sheaves on $\Tt_0$. The latter has functorial K-injective resolutions since sheaves of $\Om_{\Tt_0}$-modules has functorial injective resolutions. See the construction in \cite{spaltenstein}.
\end{rk}
To endow it with a $\Z_\Delta$-action (i.e. with maps $c_{b_1,b_2}$ as above) fix an isomorphism
\begin{equation}\label{qiofcx}
E\simeq (\tr\times\tr)_*(E)\end{equation}and assume the injective resolution $I^\cdot_{\scrF'}$ of $p^*(R(\scrF'))\overset{L}{\otimes}E$ is carried to the injective resolution $I^\cdot_{\tr_*\scrF'}$ of $(\tr\times\tr)_*(p^*(R(\scrF'))\overset{}{\otimes}E)\simeq p^*(R(\tr_*\scrF'))\otimes (\tr\times\tr)_*E\simeq p^*(R(\tr_*\scrF'))\otimes E$ under $(\tr\times\tr)_*$. Then $(\tr\times\tr)_*$ gives us chain isomorphisms \begin{equation}hom^\cdot(q^*R(\scrF),I^\cdot_{\scrF'})\simeq hom^\cdot(q^*R(\tr_*\scrF),I^\cdot_{\tr_*\scrF'} ) \end{equation}which is the desired $\Z_\Delta$-action. In the following, the isomorphisms $E\simeq(\tr\times\tr)_*(E)$ will be obvious. 
\begin{defn}
We can produce another $\Om(\Tt_0)_{dg}$-bimodule out of the complex of $\Tt_0\times \Tt_0$ modules $E$. Namely define $\fM_E'$ by \begin{equation}\label{eq:quotationdefn1}
\mathfrak{M}'_E:(\scrF,\scrF')\mapsto ``RHom_{\Om_{\Tt_0\times\Tt_0}}(q^*(\scrF)\overset{L}{\otimes}_{\Om_{\Tt_0\times\Tt_0}}E,p^*(\scrF') )" \end{equation}
Remark \ref{spalsrk} applies in this case too. A quasi-isomorphism as in (\ref{qiofcx}) would be sufficient to endow $\fM'_E$ with a $\Z_\Delta$ equivariant structure. We will not use this fact and we skip the technical details. 
\end{defn}
Assume in addition we have a bimodule $\mathfrak{N}$ over $\A$ such that \begin{equation}\mathfrak{N}\simeq (\phi\otimes\phi)_*(\mathfrak{N}) \end{equation}strictly (via a dg-bimodule map that acts as chain isomorphisms for each pairs of objects). 
Hence, we have a $\Z_\Delta$-equivariant structure on $\mathfrak{N}$, i.e. an equivariant structure with respect to $\phi\otimes\phi$. Given $\Z_\Delta$-equivariant $\Om(\Tt_0)_{dg}$-bimodule $\fM$, we can endow the $\Om(\Tt_0)_{dg}\otimes \A$-bimodule $\fM\otimes \mathfrak{N}$ with a $\Z_\Delta$-equivariant structure (with respect to $\tr\otimes\phi$); hence obtain a bimodule over $M_\phi=(\Om(\Tt_0)_{dg}\otimes \A)\#\Z$ using the recipe above. In particular, the diagonal bimodule of $\A$ is an example of such an $\mathfrak N$.

As an application of these ideas let us prove: 
\begin{prop}\label{smoothness} $M_\phi$ is a smooth category whenever $\A$ is. \end{prop}
\begin{proof}
Consider the normalization $\pi:\mathbb{P}^1\times\Z\rightarrow \Tt_0$. Throughout this proof let $\Om_\Delta$ denote the structure sheaf of the diagonal of $\Tt_0$, and let $\tilde{\Om}_\Delta$ denote $(\pi\times\pi)_*(\Om_{\Delta_{\mathbb{P}^1\times\Z}})$, where $\Delta_{\mathbb{P}^1\times\Z}$ is the diagonal of $\mathbb{P}^1\times \Z$. We have a short exact sequence of sheaves on $\Tt_0\times \Tt_0$
\begin{equation}\label{eq:ses1}
0\rightarrow\Om_\Delta\rightarrow\tilde{\Om}_\Delta\rightarrow\bigoplus_{j\in\Z}\Om_{x_{j+1/2}}\boxtimes \Om_{x_{j+1/2}}\rightarrow 0 
\end{equation} 
Here, $x_{j+1/2}$ is the node in the chart $U_{i+1/2}$, and the map $\Om_\Delta\rightarrow\tilde{\Om}_\Delta$ comes as the push-forward of $\Om_{\Tt_0}\rightarrow\pi_*(\Om_{\mathbb{P}^1\times\Z})$ under the diagonal map. Using Beilinson's resolution of diagonal of $\mathbb{P}^1$ (at each component separately) and exactness of affine push-forward $(\pi\times\pi)_*$ we obtain a resolution 
\begin{equation}\label{eq:ses2}
0\rightarrow\bigoplus_{i\in\Z}\Om_{C_i}(-1)\boxtimes\Om_{C_i}(-1)\rightarrow \bigoplus_{i\in\Z}\Om_{C_i}\boxtimes\Om_{C_i}\rightarrow \tilde{\Om}_\Delta \rightarrow 0
\end{equation}
This implies the sheaf $\Om_\Delta$ is quasi-isomorphic to twisted complex
\begin{equation}\label{sheafres}
\bigoplus_{i\in\Z}\Om_{C_i}(-1)\boxtimes\Om_{C_i}(-1)\rightarrow \bigoplus_{i\in\Z}\Om_{C_i}\boxtimes\Om_{C_i}\rightarrow\bigoplus_{j\in\Z}\Om_{x_{j+1/2}}\boxtimes \Om_{x_{j+1/2}}
\end{equation}
We could apply $E\mapsto \fM_E$ to (\ref{sheafres}); however, $\fM_{E'\boxtimes E''}$ is not quasi-isomorphic to a Yoneda bimodule. Inspired by \cite{categoricalresolution}, we will instead apply $E\mapsto\fM'_E$ to $\Om_\Delta^\vee$(i.e. to derived dual of $\Om_\Delta$) and to dual of the resolution (\ref{sheafres}). First notice, \begin{multline}\label{blah}
\fM'_{\Om_\Delta^\vee}(\scrF,\scrF')=``RHom_{\Om_{\Tt_0\times\Tt_0}}(q^*(\scrF)\overset{L}{\otimes}_{\Om_{\Tt_0\times\Tt_0}}\Om_\Delta^\vee,p^*(\scrF') )"\simeq\\ ``RHom_{\Om_{\Tt_0\times\Tt_0}}(q^*(\scrF),R\sheafhom_{\Tt_0\times\Tt_0}(\Om_\Delta^\vee,p^*(\scrF')) )"\simeq\\ ``RHom_{\Om_{\Tt_0\times\Tt_0}}(q^*(\scrF),\Om_\Delta\otimes p^*(\scrF')) )"\simeq ``RHom_{\Tt_0}(\scrF,\scrF') "
\end{multline}
Here, the quotation marks are used to omit the resolutions that are necessary for (dg) functoriality of the corresponding expression from the notation (see Remark \ref{spalsrk}). As a result of (\ref{blah}), $\fM_{\Om_\Delta^\vee}$ is quasi-isomorphic to diagonal bimodule. The only non-trivial step is the quasi-isomorphism between second and third rows and this follows from Lemma \ref{dualitylemma} (let $X$ be $\Tt_0$ and $f$ be the diagonal embedding). Taking the derived duals, we find $\Om_\Delta^\vee$ is quasi-isomorphic to twisted complex
\begin{equation}\label{dualsheafres}
\bigoplus_{j\in\Z}\Om_{x_{j+1/2}}^\vee\boxtimes \Om_{x_{j+1/2}}^\vee \rightarrow \bigoplus_{i\in\Z}\Om_{C_i}^\vee\boxtimes\Om_{C_i}^\vee\rightarrow
\bigoplus_{i\in\Z}\Om_{C_i}(-1)^\vee\boxtimes\Om_{C_i}(-1)^\vee
\end{equation}
Notice the derived duals of coherent sheaves are quasi-isomorphic to bounded complexes of coherent sheaves, thanks to the Gorenstein property. 

Applying $E\mapsto \fM'_E$, we find $\fM'_{\Om_\Delta^\vee}$ is quasi-isomorphic to 
\begin{equation}
\bigoplus_{i\in\Z}\fM'_{\Om_{C_i}(-1)^\vee\boxtimes\Om_{C_i}(-1)^\vee}\rightarrow \bigoplus_{i\in\Z}\fM'_{\Om_{C_i}^\vee\boxtimes\Om_{C_i}^\vee}\rightarrow \bigoplus_{j\in\Z}\fM'_{\Om_{x_{j+1/2}}^\vee\boxtimes \Om_{x_{j+1/2}}^\vee}
\end{equation}
We implicitly use the fact that \begin{equation}\fM'_{\bigoplus_{m\in\Z}E_m}\simeq \bigoplus_{m\in\Z}\fM'_{E_m}\end{equation} for $\{E_m\}$ satisfying the following:
given $\scrF,\scrF'\in Coh_p(\Tt_0)$ there exists only finitely many $E_m$ whose support intersects $supp(q^*(\scrF))\cup supp(p^*(\scrF'))$.

Note also that the sheaves involved in expressions (\ref{eq:ses1}) and (\ref{eq:ses2}) can be made $(\tr\times\tr)_*$-equivariant in an obvious way so that the maps can be chosen to be compatible with these $\Z_\Delta$-equivariant structures. This does apply to their duals as well; hence, the bimodule $\fM'_{\Om_\Delta^\vee}\otimes \Delta_\A\simeq \Delta_{\Om(\Tt_0)_{dg}\otimes\A}$ is quasi-isomorphic to a twisted complex of bimodules
\begin{equation}\label{eq:diagseq}
\bigg\{
\bigoplus_{i\in\Z}\fM'_{\Om_{C_i}(-1)^\vee\boxtimes\Om_{C_i}(-1)^\vee}\rightarrow \bigoplus_{i\in\Z}\fM'_{\Om_{C_i}^\vee\boxtimes\Om_{C_i}^\vee}\rightarrow \bigoplus_{j\in\Z}\fM'_{\Om_{x_{j+1/2}}^\vee\boxtimes \Om_{x_{j+1/2}}^\vee} \bigg\}\otimes\Delta_\A
\end{equation}
compatibly with the $\Z_\Delta$-action.  

Assume $E_i=E_i'\boxtimes E_i''$, where $E_i',E_i''\in Coh_p(\Tt_0)$ satisfying $E_{i+1}'=\tr_* E_i'$ and $E_{i+1}''=\tr_* E_i''$, as in (\ref{dualsheafres}). Then $\fM'_{E_i'\boxtimes E_i''}$ is a right $\Om(\Tt_0)_{dg}\otimes \Om(\Tt_0)_{dg}^{op}$-module (i.e. a functor from $\Om(\Tt_0)_{dg}^{op}\otimes \Om(\Tt_0)_{dg}$ to chains over $\C$)
represented by \begin{equation}E_i'^\vee\times E_i''\in Ob(\Om(\Tt_0)_{dg}^{op}\otimes \Om(\Tt_0)_{dg})\end{equation} where $E_i'^\vee$ is-again- the derived dual of $E_i'$. This is essentially stating \begin{equation}RHom_{\Tt_0\times\Tt_0}(q^*\scrF\overset{L}{\otimes}(E_i'\boxtimes E_i''),p^*\scrF')\simeq RHom_{\Tt_0}(\scrF,E_i'^\vee)\otimes_\C RHom_{\Tt_0}(E_i'',\scrF') \end{equation}

Hence, $\bigoplus_{i\in\Z}\fM'_{E_i'\boxtimes E_i''}\otimes\Delta_\A$, with its obvious $(\tr_*\otimes\phi)\otimes(\tr_*\otimes \phi)$-equivariant structure, descends to \begin{equation}\bigg(\bigoplus_{i\in\Z}\fM'_{E_i'\boxtimes E_i''}\otimes\Delta_\A \bigg)\#\Z \end{equation} which is quasi-isomorphic to a twisted complex we informally denote by \begin{equation}``h_{E_0'^\vee\times E_0''}\otimes\Delta_\A"\end{equation} where $h_{E_0'^\vee\times E_0''}$ is the contravariant Yoneda functor associated to $E_0'^\vee\times E_0''$. 
To see $``h_{E_0'^\vee\times E_0''}\otimes\Delta_\A"$ is quasi-isomorphic to a twisted complex over $M_\phi^e$ one may find a twisted complex $X=(X,\delta,\pi)$ over $\A^e$ that is quasi-isomorphic to $\Delta_\A$ and apply descent to an infinite equivariant sum and obtain \begin{equation}\bigg(\bigoplus_{i\in\Z}\fM'_{E_i'\boxtimes E_i''}\otimes(\phi\otimes\phi)^i(X) \bigg)\#\Z \end{equation} which can be represented by a twisted complex of objects $``E_0'^\vee\times a'\times E_0''\times a''"$.

Hence, $\Delta_{M_\phi}$, which can be obtained by descent from $\Delta_{\Om(\Tt_0)_{dg}\otimes\A}$ can be represented by a twisted complex as the latter is $\Z_\Delta$-equivariantly quasi-isomorphic to (\ref{eq:diagseq}).
\end{proof}
\begin{lem}\label{dualitylemma}
	Let $X$,$Y$,$Z$ be (locally Noetherian) Gorenstein varieties over $\C$, $f:X\hookrightarrow Y$ be a closed embedding and $p:Y\rightarrow Z$ be a flat map. Assume $p\circ f$ is also flat. Then, for any coherent sheaf $\scrF$ on $Z$, there exists a natural isomorphism in the derived category
	\begin{equation}
	\Om_X\otimes_{\Om_Y} p^*(\scrF) \xrightarrow{\simeq}R\sheafhom_Y(\Om_X^\vee,p^*\scrF)
	\end{equation}Here $\Om_X=Rf_*\Om_X=f_*\Om_X$ and $\Om_X^\vee=R\sheafhom_Y (f_*\Om_X,\Om_Y)$.
\end{lem}
\begin{proof}
	We will drop the subscript $\Om_Y$ of the tensor product and $\otimes$ refers to derived tensor product as usual. However, notice in this case flatness of $X$ over $Z$ implies $\Om_X\otimes p^*(\scrF)\simeq \Om_X\overset{L}{\otimes} p^*(\scrF)$. In particular $\Om_X\overset{L}{\otimes} p^*(\scrF)$ is a bounded complex of coherent sheaves.
	
	We also remark that $\Om^{\vee\vee}_X\simeq\Om_X$ over $Y$, thanks to the Gorenstein property (see \cite[Section V,Theorem 9.1]{resdual}). In other words, we have an isomorphism in the derived category\begin{equation}
	\Om_X\rightarrow R\sheafhom_Y (\Om_X^\vee,\Om_Y)
	\end{equation}
	which induces $\Om_X\otimes p^*\scrF\rightarrow R\sheafhom_Y (\Om_X^\vee,\Om_Y)\otimes p^*\scrF$. 
	Our asserted quasi-isomorphism is the composition of the natural maps \begin{equation}\label{oneeq}
	\Om_X\otimes p^*\scrF\rightarrow R\sheafhom_Y (\Om_X^\vee,\Om_Y)\otimes p^*\scrF\rightarrow R\sheafhom_Y (\Om_X^\vee,\Om_Y\otimes p^*\scrF)
	\end{equation}
	Whether (\ref{oneeq}) gives a quasi-isomorphism is a local question; thus, we can assume $X,Y,Z$ to be Noetherian (and even affine). First, let us compute $\Om_X^\vee$ using Duality theorem \cite[Section VII,Theorem 3.3]{resdual}. Let $D_Y$ be a dualizing complex on $Y$ and $D_X$ be a dualizing complex on $X$.  Assume $D_Y$ and $D_X$ are related by $f$ in an appropriate sense, i.e. $f^A D_Y\simeq D_X$ in the notation of \cite{resdual}.
	One can define corresponding dualizing functors as $\mathbb{D}_Y(\mathscr E)=R\sheafhom_Y (\mathscr E,D_Y)$ and $\mathbb{D}_X(\mathscr E)=R\sheafhom_X (\mathscr E,D_X)$. Then \cite[Section VII,Theorem 3.3]{resdual} states that $Rf_*\circ\mathbb D_X\simeq \mathbb D_Y\circ Rf_*$.
	If we apply this to $\scrF=\Om_X\in Coh(X)$, we obtain \begin{equation}R\sheafhom_Y(Rf_*\Om_X,D_Y)\simeq Rf_*R\sheafhom_X(\Om_X,D_X)\simeq Rf_*D_X \end{equation}
	As $X$ and $Y$ are Gorenstein, $D_X$ and $D_Y$ are quasi-isomorphic to shifted line bundles. Hence, \begin{equation}\Om_X^\vee=R\sheafhom _Y(Rf_*\Om_X,\Om_Y)\simeq D_Y^{-1}\otimes Rf_*D_X \end{equation} 
	Moreover, a functor $f^!:D^+_{coh}(Y)\rightarrow D^+_{coh}(X)$
	satisfying 
	\begin{equation}
	Rf_*R\sheafhom_X(\mathscr E,f^!\mathscr E')\xrightarrow{\simeq}R\sheafhom_Y(Rf_*\mathscr E,\mathscr E')
	\end{equation}
	for every $\mathscr E\in D^-_{qcoh}(X),\mathscr E'\in D^+_{coh}(X)$ is constructed in the proof of \cite[Section VII, Corollary 3.4]{resdual} and it also satisfies $f^!\simeq \mathbb D_X\circ Lf^*\circ \mathbb{D}_Y$. This implies
	\begin{equation}\label{birid0}
	R\sheafhom_Y(\Om_X^\vee,p^*\scrF)\simeq R\sheafhom_Y(D_Y^{-1}\otimes Rf_*D_X,p^*\scrF) \simeq\atop R\sheafhom_Y( Rf_*D_X,D_Y\otimes p^*\scrF)\simeq Rf_*R\sheafhom_X(D_X,f^!(D_Y\otimes p^*\scrF) )
	\end{equation}
	Note, we take $\mathscr E=D_X$ and $\mathscr E'=D_Y\otimes p^*\scrF$ for the last isomorphism.
	
	Now we assert,  \begin{equation}\label{birid}
	f^!(D_Y\otimes p^*\scrF)\simeq (pf)^*\scrF \otimes D_X
	\end{equation} 
	Indeed,
	\begin{align*}
	f^!(D_Y\otimes p^*\scrF)\simeq \mathbb{D}_X Lf^* R\sheafhom_Y(D_Y\otimes p^*\scrF,D_Y) \simeq\\ \mathbb{D}_X Lf^* R\sheafhom_Y(p^*\scrF,\Om_Y)\simeq \mathbb{D}_X R\sheafhom_X((pf)^*\scrF,\Om_X)\simeq \\ (pf)^*\scrF \otimes D_X
	\end{align*}
	The last identity holds due to Gorenstein property. The identity \begin{equation}Lf^* R\sheafhom_Y(p^*\scrF,\Om_Y)\simeq R\sheafhom_X((pf)^*\scrF,\Om_X) \end{equation} can be proven using flatness of $p$ and $pf$. Namely let $E\xrightarrow{\simeq}\scrF$ be a locally free resolution. $R\sheafhom_Y(p^*\scrF,\Om_Y)\simeq p^*E^\vee\simeq Lp^*E^\vee$ is bounded below. Still \begin{equation}Lf^*Lp^*E^\vee\simeq L(pf)^*E^\vee\simeq (pf)^*E^\vee\simeq (pf)^*\scrF^\vee\end{equation}
	
	Combining (\ref{birid0}) and (\ref{birid}), we see that
	\begin{equation}
	R\sheafhom_Y(\Om_X^\vee,p^*\scrF)\simeq Rf_*((pf)^*\scrF)\simeq Rf_*(\Om_X)\otimes p^*\scrF=\Om_X\otimes p^*\scrF
	\end{equation}
	This finishes the proof.
\end{proof}
\section{Hochschild cohomology of the mapping torus categories}\label{sec:hoch}
\subsection{Hochschild cohomology of $\Om(\Tt_0)_{dg}$ and $\Om(\Tt_R)_{cdg}$}
In this section we will compute the Hochschild cohomology of the mapping torus categories. For this we first need the Hochschild cohomology of $\Om(\Tt_0)_{dg}$ and $\Om(\Tt_R)_{cdg}$. Let $\scrC oh(\Tt_0)$ be a dg enhancement for the bounded derived category of coherent sheaves on $\Tt_0$. This clearly restricts to a dg enhancement of $D^b(Coh_p(\Tt_0))$. We will denote it by $\scrC oh_p(\Tt_0)$. Similarly, let $\scrC oh(U_{i+1/2})$ and $\scrC oh(U_{i+1/2}\cap U_{j+1/2})$ be dg enhancements of corresponding derived categories. Then there are pull-back maps 
\begin{equation}\label{postzerorestr} \scrC oh(\Tt_0)\rightarrow \scrC oh(U_{i+1/2})\rightarrow  \scrC oh(U_{i+1/2}\cap U_{j+1/2}) \end{equation} which are $A_\infty$-functors but without loss of generality one can choose the enhancements so that they become dg-functors (for instance replace each of these categories by their categories of perfect right $A_\infty$-modules, and the functors by the induced functors between the module categories). 
Hence, $\scrC oh(U_{i+1/2})$ and $  \scrC oh(U_{i+1/2}\cap U_{j+1/2})$ can be considered as bimodules over $\scrC oh(\Tt_0)$, where the bimodule structure is induced by the functor from $\scrC oh(\Tt_0)$ to the respective category. Moreover, 
\begin{lem}\label{lem:cohcolimit}
The diagonal bimodule $\scrC oh(\Tt_0)$ is quasi-isomorphic to the homotopy limit 
\begin{equation}\label{holim1} 
holim\bigg(\prod_{i} \scrC oh(U_{i+1/2})\rightrightarrows \prod_{j} \scrC oh(U_{j-1/2}\cap U_{j+1/2})\bigg) 
\end{equation} 
as bimodules over $\scrC oh(\Tt_0)$ and as bimodules over its full subcategory $\scrC oh_p(\Tt_0)$.
\end{lem}
We use this notation to denote the homotopy limit of the big diagram involving $\scrC oh(U_{j+1/2})$ and $\scrC oh(U_{j-1/2}\cap U_{j+1/2})$ (as the intersections $U_{i+1/2}\cap U_{j+1/2}$ are empty for $|j-i|\geq 2$, no other $\scrC oh(U_{i+1/2}\cap U_{j+1/2})$ appears). It can be seen as the equalizer of the two arrows $\prod_{i}\scrC oh(U_{i+1/2})\rightarrow \prod_{j} \scrC oh(U_{j-1/2}\cap U_{j+1/2})$, given by the product of $\scrC oh(U_{i+1/2})\rightarrow \scrC oh(U_{i-1/2}\cap U_{i+1/2})$, and by the product of $\scrC oh(U_{i+1/2})\rightarrow \scrC oh(U_{i+1/2}\cap U_{i+3/2})$ respectively. Therefore, instead of appealing to model category structures on the category of bimodules, we define (\ref{holim1}) as the cocone (i.e. the cone shifted by $-1$) of the map 
\begin{equation}
	\prod_{i} \scrC oh(U_{i+1/2})\rightarrow \prod_{j} \scrC oh(U_{j-1/2}\cap U_{j+1/2}) 
\end{equation} given by the difference of the two arrows between them. In the following discussion, we use the analogous definition for homotopy limits of chain complexes. 

Note that Lemma \ref{lem:cohcolimit} can be seen as a Zariski descent statement for coherent sheaves on open subsets considered as bimodules, but not as categories. A version for categories can be found in \cite[Proposition 4.2.1]{indcoh}. One can presumably recover Lemma \ref{lem:cohcolimit} from this too, but we follow a different route. 
\begin{proof}[Proof of Lemma \ref{lem:cohcolimit}]
The maps (\ref{postzerorestr}) induce a bimodule homomorphism from $\scrC oh(\Tt_0)$ to (\ref{holim1}), where the latter is defined as the cocone as we have explained. One only needs to check that this homomorphism is a quasi-isomorphism. In other words, one has to show the induced map on cohomology is an isomorphism, and this follows from the Mayer-Vietoris property for $Ext$-groups. For a simple argument specific to this situation, let $U=\bigsqcup_{2|i} U_{i+1/2}$ and $U'=\bigsqcup_{2\nmid i}U_{i+1/2}$ (thus, $U\cup U'=\Tt_0$ and $U\cap U'=\bigsqcup_{j} U_{j-1/2}\cap U_{j+1/2}$). Then given $\scrF,\scrF'\in\scrC oh (\Tt_0)$, one has an exact triangle 
\begin{equation}\label{eq:mvexact}
	RHom_{\Tt_0}(\scrF,\scrF')\to RHom_U(\scrF|_U,\scrF'|_U)\times RHom_{U'}(\scrF|_{U'},\scrF'|_{U'})\to\atop  RHom_{U\cap U'}(\scrF|_{U\cap U'},\scrF'|_{U\cap U'})\xrightarrow{+1}
\end{equation}
which can be obtained by applying Mayer-Vietoris for $R\Gamma$ to local-hom $R\sheafhom_{\Tt_0}(\scrF,\scrF')$. 
Clearly, (\ref{eq:mvexact}) is the same as
\begin{equation}\label{eq:mvprodexact}
		RHom_{\Tt_0}(\scrF,\scrF')\to \prod_{i} RHom_{U_{i+1/2}}(\scrF|_{U_{i+1/2}},\scrF'|_{U_{i+1/2}})\to \atop\prod_{j} RHom_{U_{j-1/2}\cap U_{j+1/2}}(\scrF|_{U_{j-1/2}\cap U_{j+1/2}},\scrF'|_{U_{j-1/2}\cap U_{j+1/2}})\xrightarrow{+1}
\end{equation}
Exactness of (\ref{eq:mvprodexact}) implies that the bimodule homomorphism 
\begin{equation}
	\scrC oh(\Tt_0)\to  cocone\bigg(\prod_{i} \scrC oh(U_{i+1/2})\rightarrow \prod_{j} \scrC oh(U_{j-1/2}\cap U_{j+1/2})  \bigg)
\end{equation}
evaluated at any $(\scrF,\scrF')$ is a quasi-isomorphism. This concludes the proof. 
\end{proof} 
Consider the functor 
\begin{align*}
Bimod(\scrC oh_p(\Tt_0), \scrC oh_p(\Tt_0))\longrightarrow \scrC_{dg}(\C) \\ \EuScript{B}\longmapsto CC^*(\scrC oh_p(\Tt_0),\EuScript{B})
\end{align*} where $\scrC_{dg}(\C)$ is the category of chains over $\C$. This functor can be seen as a Yoneda functor and hence it preserves the limits. Hence, 
\begin{equation}\label{postzeroholim} \scrC oh(\Tt_0)\simeq holim\bigg(\prod_{i} \scrC oh(U_{i+1/2})\rightrightarrows \prod_{j} \scrC oh(U_{j-1/2}\cap U_{j+1/2})\bigg) 
\end{equation} implies 
\begin{equation}\label{postholim2}
CC^*(\scrC oh_p(\Tt_0), \scrC oh(\Tt_0))\simeq \atop holim\bigg(\prod_{i} CC^*(\scrC oh_p(\Tt_0),\scrC oh(U_{i+1/2}))\rightrightarrows \prod_{j} CC^*(\scrC oh_p(\Tt_0),\scrC oh(U_{j-1/2}\cap U_{j+1/2}) )\bigg)\end{equation}

We can easily identify the chain complexes \begin{equation}\label{postide}
CC^*(\scrC oh_p(\Tt_0) )=CC^*(\scrC oh_p(\Tt_0),\scrC oh_p(\Tt_0))\cong CC^*(\scrC oh_p(\Tt_0),\scrC oh(\Tt_0))\end{equation} Moreover we have 
\begin{lem}\label{lemloc}
Let $U\subset\Tt_0$ be a quasi-compact open subvariety. Given a dg model $\scrC oh(U)$ and restriction(pull-back) functor $\scrC oh_p(\Tt_0)\rightarrow \scrC oh(U)$, the induced chain map \begin{equation}\label{postholim3}
CC^*(\scrC oh(U),\scrC oh(U))\rightarrow CC^*(\scrC oh_p(\Tt_0),\scrC oh(U)) \end{equation} is a quasi-isomorphism.
\end{lem}
\begin{proof}
This follows from Lemma \ref{lemcoh} and Lemma \ref{lemgenloc}.
\end{proof}
\begin{lem}\label{lemcoh} Let $U\subset\Tt_0$ be an open quasi-compact subvariety. Then there exists a line bundle $\mathcal{L}$ and a section $s\in\Gamma(\mathcal{L})$ such that $U=\{s\neq 0 \}$ and for any such $(\mathcal{L},s)$ the localization of $\scrC oh_p(\Tt_0)$ at the natural transformation\begin{equation}s:1_{\scrC oh_p(\Tt_0)}\rightarrow (\cdot)\otimes\mathcal{L} \end{equation} is quasi-equivalent to $\scrC oh(U)$.
\end{lem}
\begin{proof}
See \cite{localization} for the definition of localization and the proof of a similar statement.
Note, the existence of such a pair $(\mathcal L,s)$ holds for general $U$ only because we are on a curve. But, we only need it for $U=U_{i+1/2}$ or $V_j$ in which case there are obvious pairs $(\mathcal{L},s)$.
\end{proof}
\begin{lem}\label{lemgenloc}
Let $\B$ be a dg category, $\Phi$ be an auto-equivalence and $T:1\rightarrow \Phi$ be a natural transformation. Consider the localization functor $\B\rightarrow T^{-1}\B$. Then, \begin{equation}CC^*(T^{-1}\B,T^{-1}\B)\simeq CC^*(\B,T^{-1}\B) \end{equation}	
\end{lem}

For motivation, one can consider the case where $\B$ is an ordinary commutative algebra and $T=f\in\B$. In this case, it is obvious that $RHom_{\B_f^e}(\B_f,\B_f)\cong RHom_{\B^e}(\B,\B_f)$.
\begin{proof}[Sketch of the Proof]
The general idea is the following: first one checks that $T^{-1}\B$, as defined in \cite{localization}, is the same as localization at morphisms $T_L:L\to \Phi(L)$, where the localization is defined as in \cite{GPS1}, i.e. as the quotient of $\B$ by the cones of these morphisms. This comparison can be done analogous to \cite[Lemma 3.37]{GPS1}. There is a similar notion of (left and right) localization of bimodules (see \cite{GPS1}), and it is easy to see that $T^{-1}\B$, as a $T^{-1}\B$-$\B$-bimodule can be obtained from the diagonal bimodule of $\B$ by left localization at these morphisms. Left localization is equivalent to $T^{-1}\B\otimes_\B(\cdot)$; hence, it is left adjoint to the restriction functor from $T^{-1}\B$-$\B$-bimodules to $\B$-$\B$-bimodules. Therefore, 
\begin{equation}
RHom_{\B^e}(\B,T^{-1}\B)\simeq 	RHom_{T^{-1}\B\otimes\B^{op}}(T^{-1}\B,T^{-1}\B)
\end{equation}
Moreover, $T^{-1}\B$ is right local with respect to these morphisms, i.e. the morphisms $T_L:L\to \Phi(L)$ act invertibly on the right. As a result, it is the restriction of a bimodule over $T^{-1}\B$ (indeed of the diagonal bimodule itself). By \cite{lyubamanquot}, the category of $T^{-1}\B\otimes \B^{op}$-modules on which all $1\otimes T_L$ act invertibly is equivalent to $T^{-1}\B\otimes T^{-1}\B^{op}$-modules, i.e. the category of right local $T^{-1}\B$-$\B$-bimodules is equivalent to the category of $T^{-1}\B$-$T^{-1}\B$-bimodules. Hence, 
\begin{equation}
	RHom_{T^{-1}\B\otimes\B^{op}}(T^{-1}\B,T^{-1}\B)\simeq RHom_{T^{-1}\B\otimes T^{-1}\B^{op}}(T^{-1}\B,T^{-1}\B)
\end{equation}
which implies
\begin{equation}
	RHom_{\B^e}(\B,T^{-1}\B)\simeq 	RHom_{T^{-1}\B\otimes T^{-1}\B^{op}}(T^{-1}\B,T^{-1}\B)
\end{equation}
This quasi-isomorphism can be realized as the natural restriction map from right hand side to left hand side, which implies the natural chain map $CC^*(T^{-1}\B,T^{-1}\B)\to CC^*(\B,T^{-1}\B)$ is also a quasi-isomorphism.
%
\end{proof}
We can summarize this discussion as 
\begin{equation}\label{eq:HHcolimit0}
CC^*(\scrC oh_p(\Tt_0))\simeq \atop holim\bigg(\prod CC^*(\scrC oh( U_{i+1/2})) \rightrightarrows \prod CC^*(\scrC oh(U_{j-1/2} \cap U_{j+1/2})) \bigg)  
\end{equation}			
For the moment let $\scrC oh(U_{i+1/2})_R$, $\scrC oh(V_i)_R$ denote some curved deformations compatible with the deformation of $\Tt_0$ to $\Tt_R$ (and restriction of this deformation to corresponding open subsets). Note the compatibility here is in a loose sense, see the notion of good deformation in Section \ref{subsec:defoofom} for instance. Most importantly, we need restriction functors (\ref{postzerorestr}) to deform so that the map in (\ref{postzerorestr}) deforms as well. The chain complexes \begin{equation}CC^*(\scrC oh(U_{i+1/2})_R)\text{ and }CC^*(\scrC oh(V_i)_R)\end{equation} deform the complexes $CC^*(\scrC oh(U_{i+1/2}))\text{ and }CC^*(\scrC oh(V_i))$ respectively. Similarly the complex $CC^*(\scrC oh_p(\Tt_0)_R)$ deforms $CC^*(\scrC oh_p(\Tt_0))$, where $\scrC oh_p(\Tt_0)_R$ is a curved deformation of $\scrC oh_p(\Tt_0)$ extending the one for $\Om(\Tt_0)_{dg}$. We can write the map 
\begin{equation}CC^*(\scrC oh_p(\Tt_0)_R)\xrightarrow{} \atop holim\bigg(\prod CC^*(\scrC oh( U_{i+1/2})_R) \rightrightarrows \prod CC^*(\scrC oh(U_{j-1/2} \cap U_{j+1/2})_R) \bigg)  \end{equation} following similar steps as before, for instance by deforming the maps in (\ref{postzeroholim}); and thus, (\ref{postholim2}). 
Note that the analogue of Lemma \ref{lemloc} can be shown by a semi-continuity/$q$-adic filtration argument. Namely:
\begin{lem}\label{semicont}
If one has a chain map \begin{equation}f:C^*_R\rightarrow C'^*_R \end{equation}of complexes of complete topological torsion free vector spaces over $R$ which deforms a quasi-isomorphism $C^*\xrightarrow{\simeq} C'^* $, then $f$ itself is a quasi-isomorphism.
\end{lem}
 Moreover, using the semi-continuity and deformability of the maps such as \ref{postholim2} and \ref{postholim3}, we prove
\begin{equation}\label{eq:HHcolimitR}
CC^*(\scrC oh_p(\Tt_0)_R)\simeq \atop holim\bigg(\prod CC^*(\scrC oh( U_{i+1/2})_R) \rightrightarrows \prod CC^*(\scrC oh(U_{j-1/2} \cap U_{j+1/2})_R) \bigg)
\end{equation}  
Now, let us turn to the questions about the Hochschild cohomology of $\scrC oh(U_{i+1/2})$, $\scrC oh(V_i)$ as well as their deformations. 
						
First note we can as well compute the Hochschild cohomology of perfect complexes $\scrP erf(U_{i+1/2})\subset\scrC oh(U_{i+1/2})$ and $\scrP erf(V_i)\subset \scrC oh(V_i)$ as well as their deformations. It is possible to show that the restriction maps induce isomorphisms \begin{equation}\label{eq:restr}CC^*(\scrC oh(U))\xrightarrow{\simeq} CC^*(\scrP erf(U)) \end{equation} and this implies by semi-continuity \begin{equation}CC^*(\scrC oh(U)_R)\xrightarrow{\simeq} CC^*(\scrP erf(U)_R) \end{equation} where $U$ is $U_{i+1/2}$ or $V_i$ and $\scrP erf(U)_R$ is the corresponding deformation. See \cite[Appendix F]{gaitarin} for an Ind-completed version of (\ref{eq:restr}). Alternatively one can identify Hochschild cohomologies of $\scrC oh(U)$, resp. $\scrP erf(U)$ with derived self-endomorphisms of the diagonal of $U$ in the category $D^bCoh(U\times U)$, resp. $D(QCoh(U\times U))$, which are known to match. 

As $U$ is affine, $CC^*(\scrP erf(U))\simeq CC^*(\Om(U))$ and similar for the deformations. Notice that we use the fact that we can deform the functors \begin{equation}\Om(U)\rightarrow \scrP erf(U)\rightarrow \scrC oh(U) \end{equation}which was imposed for ``good deformations''. Let $\Om(\tU)$ and $\Om(U)_R$ both denote the corresponding deformation of the algebra $\Om(U)$. More explicitly \begin{equation}\Om(\tU_{i+1/2})=\C[X_i,Y_{i+1}][[q]]/(X_iY_{i+1}-q) \end{equation}\begin{equation}\Om(\tV_i)=\C[X_i,X_i^{-1}][[q]]=\C[Y_i^{-1},Y_i][[q]] \end{equation} In summary \begin{equation}CC^*(\scrC oh(U_{i+1/2})_R)\simeq CC^*(\Om(\tU_{i+1/2})) \end{equation}\begin{equation}CC^*(\scrC oh(V_i)_R)\simeq CC^*(\Om(\tV_i)) \end{equation} where the Hochschild cohomologies are computed over $R$. Now, using \cite[Appendix, Theorem 2]{quantization} one can show:
\begin{lem}
\begin{equation}\label{postquant}
CC^*(\Om(U_{i+1/2}))\simeq \C[X_i,Y_{i+1},X_i^*,Y_{i+1}^*,\beta_{i+1/2}]/(X_iY_{i+1})  \end{equation} where the latter dga is the quotient of the free (super-commutative) graded algebra generated by the variables $X_i,Y_{i+1},X_i^*,Y_{i+1}^*,\beta_{i+1/2}$ with degrees $|X_i|=|Y_{i+1}|=0,|X_i^*|=|Y_{i+1}^*|=1,|\beta_{i+1/2}|=2$ as a graded algebra. Its differential $d$ satisfies \begin{equation}\label{eq:diffofhh1}d(X_i)=d(Y_{i+1/2})=d(\beta_{i+1/2})=0 \end{equation}\begin{equation}\label{eq:diffofhh2}d(X_i^*)=Y_{i+1}\beta_{i+1/2},d(Y_{i+1}^*)=X_{i}\beta_{i+1/2} \end{equation}
\end{lem}
Using an $R$-relative version of the same theorem, we can prove:
\begin{lem}
\begin{equation}\label{eq:defcc}CC^*(\Om(\tU_{i+1/2}))\simeq \C[X_i,Y_{i+1},X_i^*,Y_{i+1}^*,\beta_{i+1/2}][[q]]/(X_iY_{i+1}-q)  \end{equation} where the degrees of the variables are the same and the differential still satisfies (\ref{eq:diffofhh1}) and (\ref{eq:diffofhh2}). 
We note that in (\ref{eq:defcc}) the $q$-adic completion of the free graded algebra $\C[X_i,Y_{i+1},X_i^*,Y_{i+1}^*,\beta_{i+1/2}]$ is taken separately at each degree. 
\end{lem}
It is now easy to calculate the cohomology of the above dga's:
\begin{lem}
The cohomology of $CC^*(\Om(U_{i+1/2}))$ can be computed as \begin{equation}\label{eq:cohhoch}HH^*(\Om(U_{i+1/2}))=\begin{cases}
\Om(U_{i+1/2})=\C[X_i,Y_{i+1}]/(X_iY_{i+1})&*=0\\
\Om(U_{i+1/2})\langle X_iX_i^*\rangle\oplus\Om(U_{i+1/2})\langle Y_{i+1}Y_{i+1}^*\rangle&*=1\\
\C\langle\beta^k\rangle\cong \Om(U_{i+1/2})/(X_i,Y_{i+1})&*=2k,k\geq 1\\
\frac{\Om(U_{i+1/2})\langle X_iX_i^*\beta^k,Y_{i+1}Y_{i+1}^*\beta^k \rangle}{(( X_iX_i^*-Y_{i+1}Y_{i+1}^*)\beta^k)}&*=2k+1,k\geq 1
\end{cases} \end{equation}
which can be written concisely as the graded commutative algebra \begin{equation}\frac{\C[X_i,Y_{i+1},X_iX_i^*,Y_{i+1}Y_{i+1}^*,\beta_{i+1/2}]}{(X_iY_{i+1},X_i\beta_{i+1/2},Y_{i+1}\beta_{i+1/2},(X_iX_i^*-Y_{i+1}Y_{i+1}^*)\beta_{i+1/2})} \end{equation}
\end{lem}
Note, the cohomology groups (\ref{eq:cohhoch}) are not free over $\Om(U_{i+1/2})$ unless $*=0$. For instance, in the second line of (\ref{eq:cohhoch}) $Y_{i+1}(X_iX_i^*)=0$ still holds and $\Om(U_{i+1/2})\langle X_iX_i^*\rangle\oplus\Om(U_{i+1/2})\langle Y_{i+1}Y_{i+1}^*\rangle\cong \C[X_{i}]\oplus\C[Y_{i+1}]$. 
\begin{lem}
The cohomology of $CC^*(\Om(\tU_{i+1/2}))$ can be computed as \begin{equation}HH^*(\Om(\tU_{i+1/2}))=\begin{cases}
\Om(\tU_{i+1/2})=\C[X_i,Y_{i+1}][[q]]/(X_iY_{i+1}-q)&*=0\\
\Om(\tU_{i+1/2})\langle X_iX_i^*-Y_{i+1}Y_{i+1}^*\rangle&*=1\\
\C\langle\beta^k\rangle\cong \Om(\tU_{i+1/2})/(X_i,Y_{i+1})&*=2k,k\geq 1\\
0&*=2k+1,k\geq 1
\end{cases} \end{equation}
which can be written concisely as the graded commutative algebra \begin{equation}\frac{\C[X_i,Y_{i+1},X_iX_i^*-Y_{i+1}Y_{i+1}^*,\beta_{i+1/2}][[q]]}{(X_iY_{i+1}-q,X_i\beta_{i+1/2},Y_{i+1}\beta_{i+1/2},(X_iX_i^*-Y_{i+1}Y_{i+1}^*)\beta_{i+1/2})} \end{equation}where the $q$-completion is taken in each degree separately.
\end{lem}
The Hochschild cohomology of $\Om(V_i)$ and $\Om(\tV_i)$ can be computed using the same theorem or Hochschild-Kostant-Rosenberg isomorphism. We have \begin{equation}CC^*(\Om(V_i))\simeq \C[X_i,X_i^{-1},X_i^*]\text{ and } CC^*(\Om(\tV_i))\simeq \C[X_i,X_i^{-1},X_i^*][[q]] \end{equation} Here, $|X_i|=0,|X_i^*|=1$ and the differential vanishes. In the latter, the $q$-adic completion is taken separately at each degree. 
								
To compute the Hochschild cohomology of $\Om(\Tt_0)_{dg}$ and $\Om(\Tt_R)_{cdg}$ we also need the localization maps \begin{equation}\label{eq:HHloc0}HH^*(\Om(U_{i+1/2}))\rightarrow HH^*(\Om(V_i)),HH^*(\Om(U_{i+1/2}))\rightarrow HH^*(\Om(V_{i+1})) \end{equation} \begin{equation}\label{eq:HHlocR}HH^*(\Om(\tU_{i+1/2}))\rightarrow HH^*(\Om(\tV_i)),HH^*(\Om(\tU_{i+1/2}))\rightarrow HH^*(\Om(\tV_{i+1})) \end{equation}
They all vanish when $*\geq 2$ for the right hand side vanish. For the others identify\begin{equation}\Om(V_i)\cong \Om(U_{i+1/2})_{X_i},\Om(V_{i+1})\cong \Om(U_{i+1/2})_{Y_{i+1}} \end{equation}
\begin{equation}\Om(\tV_i)\cong (\Om(\tU_{i+1/2})_{X_i})[[q]],\Om(\tV_{i+1})\cong (\Om(\tU_{i+1/2})_{Y_{i+1}})[[q]] \end{equation}The identification gives the localization maps (\ref{eq:HHloc0}) and (\ref{eq:HHlocR}) for $*=0$.
								
For degree $*=1$ we have 
\begin{align*}
HH^1(\Om(U_{i+1/2}))\rightarrow HH^1(\Om(V_i))&&HH^1(\Om(U_{i+1/2}))\rightarrow HH^1(\Om(V_{i+1}))\\ X_iX_i^*\mapsto X_iX_i^*,Y_{i+1}Y_{i+1}^*\mapsto 0 && X_iX_i^*\mapsto 0,Y_{i+1}Y_{i+1}^*\mapsto Y_{i+1}Y_{i+1}^*
\end{align*}
and
\begin{align*}
HH^1(\Om(\tU_{i+1/2}))\rightarrow HH^1(\Om(\tV_i))&&HH^1(\Om(\tU_{i+1/2}))\rightarrow HH^1(\Om(\tV_{i+1}))\\ X_iX_i^*-Y_{i+1}Y_{i+1}^*\mapsto X_iX_i^* && X_iX_i^*-Y_{i+1}Y_{i+1}^*\mapsto -Y_{i+1}Y_{i+1}^*
\end{align*}    
To see this, for instance for $HH^1(\Om(\tU_{i+1/2}))\rightarrow HH^1(\Om(\tV_i))$, see $X_iX_i^*-Y_{i+1}Y_{i+1}^*$ as the derivation $X_i\partial_{X_i}-Y_{i+1}\partial_{Y_{i+1}}$ acting on $\C[X_i,Y_{i+1}][[q]]/(X_iY_{i+1}-q)$. As mentioned above, $\Om(\tV_i)=\C[X_i^{\pm}][[q]]\cong \C[X_i^{\pm},Y_{i+1}][[q]]/(X_iY_{i+1}-q)$ and the derivation acts as $X_i^m\mapsto mX_i^m$, which is exactly the action of $X_i\partial_{X_i}$ on $\C[X_i^{\pm}][[q]]$. The others follow from similar considerations. 
								
To compute the limits, we need one extra information. Namely, we identify $\Om(V_i)$ with $\C[X_i^{\pm}]$ and $\C[Y_i^{\pm}]$ and the coordinates satisfy $X_iY_i=1$. Basic calculus would tell us that the derivation corresponding to $X_iX_i^*$ acts the same as $-Y_iY_i^*$.
								
Now, we are ready to compute the Hochschild cohomology of $\Om(\Tt_0)_{dg}$ and $\Om(\Tt_R)_{cdg}$ in low degree. First, recall we can see the homotopy limit as the right derived functor of the limit functor. 
\begin{rk}\label{limitvssheaf}
For conceptual ease, we will think of above data and localization maps as defining sheaves on $\Tt_0$ and $\Tt_R$. We emphasize there is no need to pass to sheaves and one can merely work with diagram categories. However, this is the basis of many ideas we have used. Then, the desired (homotopy) limits can be thought as (right derived) global sections of these sheaves. For instance, for $\Tt_0$ consider the sheaf that assigns 
\begin{equation}
U\mapsto CC^*(\scrC oh_p(\Tt_0),\scrC oh(U) )
\end{equation}
for $U=U_{i+1/2}$ or $U_{j-1/2} \cap U_{j+1/2}$. The restriction maps are induced by the pull-back maps for the inclusions $U_{i-1/2}\cap U_{i+1/2}\to U_{i+1/2}$. By (\ref{postholim2}) and (\ref{postide}), the global sections of this sheaf compute the Hochschild cohomology of $\Om(\Tt_0)_{dg}$. See also (\ref{eq:HHcolimit0}). One can replace $CC^*(\scrC oh_p(\Tt_0),\scrC oh(U) )$ by explicit supercommutative dga as in (\ref{postquant}), but this will not be necessary since cohomology level information is sufficient to compute the cohomology of the global sections as we will see. 
\end{rk}
\begin{lem}\label{HHsheaveslemma}Cohomology of these sheaves are isomorphic to
\begin{align}
\Om_{\Tt_0}\text{ resp. }\Om_{\Tt_R}&\text{ for }*=0\\
\bigoplus_{i\in\Z}\Om_{C_i}\text{ resp. }\Om_{\Tt_R}&\text{ for }*=1\\
\bigoplus_{i\in\Z}\Om_{x_{i+1/2}}\text{ resp. }\bigoplus_{i\in\Z}\Om_{x_{i+1/2}}&\text{ for }*=2k,k\geq 1\\
\bigoplus_{i\in\Z}\Om_{x_{i+1/2}}\text{ resp. }0&\text{ for }*=2k+1,k\geq 1
\end{align}
\end{lem}
To relate the global sections (a.k.a. the limits of relevant diagrams) of these sheaves to desired homotopy limit, we can use the Grothendieck spectral sequence.
										
More precisely, let $\scrC\scrC^*$, resp. $\scrC\scrC^*_R$ denote the homotopy sheaves on $\Tt_0$, resp. $\Tt_R$ mentioned in Remark \ref{limitvssheaf}. We combine (\ref{eq:HHcolimit0}), (\ref{eq:HHcolimitR}), the invariance of Hochschild cohomology under passing to twisted complexes and Remark \ref{limitvssheaf}, and we apply Grothendieck spectral sequence to obtain two spectral sequences \begin{equation}E_2^{pq}=H^p(\mathcal{H}\mathcal{H}^q)\Rightarrow HH^{p+q}(\Om(\Tt_0)_{dg}) \end{equation} \begin{equation}E_2^{pq}=H^p(\mathcal{H}\mathcal{H}_R^q)\Rightarrow HH^{p+q}(\Om(\Tt_R)_{cdg}) \end{equation}
Here, $\mathcal{H}\mathcal{H}^q$, resp. $\mathcal{H}\mathcal{H}_R^q$, denotes the $q^{th}$ hypercohomology of $\scrC\scrC^*$, resp. $\scrC\scrC^*_R$, which are listed in Lemma \ref{HHsheaveslemma}. The spectral sequence degenerates in $E_2$ page (since $H^p=0$ unless $p=0,1$) and we can easily compute 
\begin{prop}\label{HHOmfinal}
\begin{equation}
HH^*(\Om(\Tt_0)_{dg})=\begin{cases}\C&*=0\\\prod_{i\in\Z}\C\langle Y_iY_i^*\rangle&*=1\\\prod_{i\in\Z}\C\langle \beta_{i+1/2}\rangle&*=2
\end{cases}
\end{equation}
\begin{equation}\label{eq:HHOmfinalR}
HH^*(\Om(\Tt_R)_{cdg})=\begin{cases}R&*=0\\R &*=1\\\prod_{i\in\Z}\C\langle \beta_{i+1/2}\rangle&*=2
\end{cases}\end{equation}
Moreover, $HH^1(\Om(\Tt_R)_{cdg})$ is generated by a class locally given by the derivation $Y_{i+1}Y_{i+1}^*-X_iX_i^*=Y_{i+1}\partial{Y_{i+1}}-X_i\partial{X_i}$ and $qHH^*(\Om(\Tt_R)_{cdg})=0$ for $*\geq 2$.
\end{prop}
\begin{defn}\label{gammaom}
Let $\gamma_\Om\in HH^1(\Om(\Tt_0)_{dg})$ (resp. $\gamma^R_\Om\in HH^1(\Om(\Tt_R)_{cdg})$) denote the class locally given by $Y_{i+1}Y_{i+1}^*-X_iX_i^*$. Note $\gamma_\Om$ corresponds to $(Y_iY_i^*)_{i}$, where each $Y_iY_i^*$ is considered as a vector field on $C_i\subset \Tt_0$. 
\end{defn}
As we will see $\gamma_\Om$ and $\gamma_\Om^R$ can be obtained as the ``infinitesimal action'' corresponding to $\G$-action mentioned in Section \ref{subsec:action}. See Prop \ref{postgammaomprop}, for instance.
\subsection{Hochschild cohomology of $M_\phi$}
Let us return to the main problem of computing $HH^*(M_\phi)$. The simple idea is as follows: Given two dg/$A_\infty$ categories (possibly with curvature) $\B_1$ and $\B_2$, we have a map \begin{equation}CC^*(\B_1,\B_1)\otimes CC^*(\B_2,\B_2)\rightarrow CC^*(\B_1\otimes \B_2,\B_1\otimes \B_2) \end{equation}Moreover, this is a quasi-isomorphism under certain compactness conditions on $\B_i$, for instance if both are smooth. In addition, given dg category $\B$ with a strict action of the discrete group $G$, we can compute $HH^*(\B\#G,\B\#G)$ as the derived invariants of the complex $CC^*(\B,\B\#G)$.

Let us first start with a few remarks on $HH^*(\B\#G,\B\#G)$. Let $\B$ be a dg category with a strict action of discrete group $G$. Let $\tilde{\fM}$ be a bimodule over $\B\#G$. Then we have \begin{equation}\label{eq:rhominv}RHom_{(\B\#G)^e}(\B\#G,\tilde{\fM})\cong RHom_{\B^e\#G_\Delta}(\B,\tilde{\fM}) \end{equation}This is true since $\B\#G$, as a bimodule over $\B\#G$ can be obtained as a base change under \begin{equation}\B^e\#G_\Delta\rightarrow (\B\#G)^e \end{equation}i.e. it is isomorphic to the induced representation $Ind^{G\times G}_{G_\Delta}(\B)$. 
%
Hence, 
\begin{equation}\label{eq:rhominv2}
RHom_{(\B\#G)^e}(\B\#G,\tilde{\fM})\cong RHom_G(\C,RHom_{\B^e}(\B,\tilde{\fM}))
\end{equation} 
Here, $RHom_G(\C,\cdot)$ is the derived invariants functor on $D(Rep(G))$. Let $G=\Z$ and $C^*$ be a representation of $G$, where the generator $1\in \Z$ acts by $\eta\curvearrowright C^*$. Then, we can construct a chain model for the derived invariants as \begin{equation}\label{eq:invcocone}cocone(C^*\xrightarrow{\eta-1_{C^*}} C^*)=cone(C^*\xrightarrow{\eta-1_{C^*}} C^*)[-1] \end{equation}
Assume $G=\Z$ and the generator $1\in\Z$ acts on $\B$ by the strict auto-equivalence $\psi$. Let $\psi_*$ denote the auto-equivalence induced on $CC^*(\B,\tilde{\fM})$. Note the action on $\tilde{\fM}$ is by $t\otimes t^{-1}\in (\B\#\Z)^e$ where $t\in\B\#\Z$ denotes the generator of $\Z$ in $\B\#\Z$. We have 
\begin{equation}\label{eq:oldequivariance}
\xymatrix{CC^*(\B\#\Z,\tilde{\fM})\ar[d]^{i^*}\\CC^*(\B,\tilde{\fM})\ar[r]^{\psi_*-1} &CC^*(\B,\tilde{\fM}) } 
\end{equation}
where $i^*$ is induced by $i:\B\rightarrow\B\#\Z$ and where the composition is $0$ in cohomology. The composition claim follows from the identification (\ref{eq:rhominv2}), and the quasi-isomorphisms of Hochschild cochains with homomorphisms in the respective bimodule categories from the diagonal bimodules of $\B$, resp. $\B\# G$. 
To define a chain map
\begin{equation}\label{eq:earlyequivariancecocone}
	CC^*(\B\#\Z,\tilde{\fM})\rightarrow cocone( CC^*(\B,\tilde{\fM})\xrightarrow{\psi_*-1}CC^*(\B,\tilde{\fM}))
\end{equation}
one needs an $h$
\begin{equation}\label{eq:earlyequivariancetriangle}
\xymatrix{CC^*(\B\#\Z,\tilde{\fM})\ar[d]^{i^*}\ar@{-->}[rd]^h\\CC^*(\B,\tilde{\fM})\ar[r]^{\psi_*-1} &CC^*(\B,\tilde{\fM}) } 
\end{equation}
satisfying $d(h)=(\psi_*-1)\circ i^*$. Presumably, one can define such an $h$ explicitly. However, instead of appealing to this, we remark that (\ref{eq:oldequivariance}) can be completed to a natural strictly commutative square
\begin{equation}\label{eq:equivariancetriangle}
\xymatrix{CC^*(\B\#\Z,\tilde{\fM})\ar[d]^{i^*}&C(\B,\tilde{\fM})\ar[d]^{0}\ar[l]_{\;\;\;\;\simeq} \\CC^*(\B,\tilde{\fM})\ar[r]^{\psi_*-1} &CC^*(\B,\tilde{\fM}) } 
\end{equation}
where the upper vertical arrow is a quasi-isomorphism. This amounts to writing the map (\ref{eq:earlyequivariancecocone}) in the derived category. Such a complex $C(\B,\tilde{\fM})$ can be obtained using the quasi-isomorphisms between Hochschild complexes and hom-complexes in categories of bimodules over $\B$ and $\B\# G$. More precisely, one initially obtains a zigzag of quasi-isomorphisms to $CC^*(\B\#\Z,\tilde{\fM})$; however, it is standard to replace the zigzag by an actual quasi-isomorphism from a complex $C(\B,\tilde{\fM})$ to $CC^*(\B\#\Z,\tilde{\fM})$.

As a result, we have a natural map (in the derived category over the base ring, which is $R$ or $\C$)
\begin{equation}\label{eq:equivariancecocone}
CC^*(\B\#\Z,\tilde{\fM})\rightarrow cocone( CC^*(\B,\tilde{\fM})\xrightarrow{\psi_*-1}CC^*(\B,\tilde{\fM}))
\end{equation}
(\ref{eq:equivariancecocone}) is a quasi-isomorphism by the previous remarks (i.e. by (\ref{eq:rhominv}), (\ref{eq:rhominv2}), that (\ref{eq:invcocone}) computes the derived invariants, and by the identification of Hochschild complexes with $hom$-complexes in bimodule categories). Moreover, (\ref{eq:equivariancetriangle}) and (\ref{eq:equivariancecocone}) generalize to the curved case as well and (\ref{eq:equivariancecocone}) is still a quasi-isomorphism by Lemma \ref{semicont}. We prefer to notationally pretend that the quasi-isomorphism (\ref{eq:equivariancecocone}) is a chain map.

Using the remarks above we can prove:
\begin{prop}\label{CCtorus}
Let $\A$ be a dg category that satisfies Conditions \ref{C1}-\ref{C3}. Then \begin{equation}CC^*(M_\phi,M_\phi)\simeq cocone\big(CC^*(\Om(\Tt_0)_{dg},\Om(\Tt_0)_{dg} )\otimes CC^*(\A,\A)\atop\xrightarrow{\tr_*\otimes \phi_*-1} CC^*(\Om(\Tt_0)_{dg},\Om(\Tt_0)_{dg} )\otimes CC^*(\A,\A) \big) \end{equation}i.e. $CC^*(M_\phi,M_\phi)$ is given by the derived invariants of the $\Z$-action on \begin{equation}CC^*(\Om(\Tt_0)_{dg},\Om(\Tt_0)_{dg} )\otimes CC^*(\A,\A)\end{equation}
\end{prop}
\begin{proof}
We noted (\ref{eq:equivariancecocone}) is a quasi-isomorphism. As a special case, we obtain the quasi-isomorphism \begin{equation}CC^*(M_\phi,M_\phi)\simeq cocone\big(CC^*(\Om(\Tt_0)_{dg}\otimes\A,(\Om(\Tt_0)_{dg}\otimes\A)\#\Z) \atop\xrightarrow{\tr_*\otimes \phi_*-1} CC^*(\Om(\Tt_0)_{dg}\otimes\A,(\Om(\Tt_0)_{dg}\otimes\A)\#\Z) \big) \end{equation}
We can write \begin{equation}(\Om(\Tt_0)_{dg}\otimes\A)\#\Z=\bigoplus_{n\in\Z} (\Om(\Tt_0)_{dg}\otimes\A)_{(\tr\otimes \phi)^n}\end{equation} as a bimodule over $\Om(\Tt_0)_{dg}\otimes\A$. Here, $(\Om(\Tt_0)_{dg}\otimes\A)_{(\tr\otimes \phi)^n}$ denotes the diagonal bimodule of $\Om(\Tt_0)_{dg}\otimes\A$ twisted by $(\tr\otimes\phi)^n$ on the right (i.e. $\Om(\Tt_0)_{dg}\otimes\A$ acts on the right by the composition of $(\tr\otimes\phi)^n$ and the right action on the diagonal bimodule). If $n\neq 0$ \begin{equation}\label{eq:bort}CC^*(\Om(\Tt_0)_{dg}\otimes\A,(\Om(\Tt_0)_{dg}\otimes\A)_{(\tr\otimes \phi)^n})\simeq 0\end{equation}which follows from \begin{equation}\label{eq:zort}
RHom_{\Om(\Tt_0)^e}(\Om(\Tt_0)_{dg},(\Om(\Tt_0)_{dg})_{\tr^n})=0
\end{equation} unless $n=0$. That (\ref{eq:zort}) implies (\ref{eq:bort}) follows from a calculation very similar to the calculation below. 

We will not prove (\ref{eq:zort}) here but simply mention that its proof is based on showing \begin{equation}\label{eq:fmvsshv}
RHom_{\Om(\Tt_0)^e}(\Om(\Tt_0)_{dg},(\Om(\Tt_0)_{dg})_{\tr^n})\simeq RHom_{\Tt_0\times\Tt_0}(\Om_{graph(\tr^n)}^\vee,\Om_{\Delta_{\Tt_0}}^\vee) \end{equation}which is $0$ as the graph of $\tr^n$ and the diagonal are disjointly supported. For the equivalence one does not need to fully develop Fourier-Mukai theory for compactly supported coherent sheaves on $\Tt_0$. Instead, we can write resolutions of $\Om_{\Delta_{\Tt_0}}$ and $\Om_{graph(\tr^n)}$ by infinite direct sums of exterior products of compactly supported sheaves (such as $\Om_{C_i}\boxtimes\Om_{C_i}$) such that direct sums satisfy some finiteness property (as in (\ref{sheafres}) and (\ref{dualsheafres})). We can make the comparison in (\ref{eq:fmvsshv}) (i.e. compare the homomorphisms of coherent sheaves and induced bimodules) for these exterior tensor products first, and then use this to deduce (\ref{eq:fmvsshv}). 

In summary\begin{equation}CC^*(M_\phi,M_\phi)\simeq cocone\big(CC^*(\Om(\Tt_0)_{dg}\otimes\A,\Om(\Tt_0)_{dg}\otimes\A) \atop\xrightarrow{\tr_*\otimes \phi_*-1} CC^*(\Om(\Tt_0)_{dg}\otimes\A,\Om(\Tt_0)_{dg}\otimes\A) \big) \end{equation}
Now, consider the natural map \begin{equation}CC^*(\Om(\Tt_0)_{dg},\Om(\Tt_0)_{dg} )\otimes CC^*(\A,\A)\rightarrow CC^*(\Om(\Tt_0)_{dg}\otimes\A,\Om(\Tt_0)_{dg}\otimes\A)\end{equation}
We would like to show this gives a quasi-isomorphism. Notice
\begin{align*}
CC^*(\Om(\Tt_0)_{dg}\otimes\A,\Om(\Tt_0)_{dg}\otimes\A)\simeq\\RHom_{(\Om(\Tt_0)_{dg}\otimes\A)^e}(\Om(\Tt_0)_{dg}\otimes\A,\Om(\Tt_0)_{dg}\otimes\A)\simeq\\ RHom_{\Om(\Tt_0)_{dg}^e}(\Om(\Tt_0)_{dg},RHom_{\A^e}(\A,\Om(\Tt_0)_{dg}\otimes\A)\simeq\\
RHom_{\Om(\Tt_0)_{dg}^e}(\Om(\Tt_0)_{dg},CC^*(\A,\A)\otimes\Om(\Tt_0)_{dg})
\end{align*}
The last quasi-isomorphism is due to smoothness of $\A$. The K\"unneth map \begin{equation}RHom_{\Om(\Tt_0)_{dg}^e}(\Om(\Tt_0)_{dg},\Om(\Tt_0)_{dg})\otimes CC^*(\A,\A)\rightarrow\atop RHom_{\Om(\Tt_0)_{dg}^e}(\Om(\Tt_0)_{dg},CC^*(\A,\A)\otimes\Om(\Tt_0)_{dg}) \end{equation}is obvious. Clearly, this map  strictly commutes with $\Z$-actions; hence, it induces a map between derived $\Z$-invariants of left and right hand sides. We want to show this map is a quasi-isomorphism. The conditions \ref{C1},\ref{C2} imply that $CC^*(\A,\A)$ has finite dimensional cohomology in each degree. Moreover, \begin{equation}RHom_{\Om(\Tt_0)_{dg}^e}(\Om(\Tt_0)_{dg},\Om(\Tt_0)_{dg})\end{equation} has bounded below cohomology. This is sufficient to show that the map above induces a quasi-isomorphism. 
This finishes the proof. 
\end{proof}
\begin{cor}\label{HHtoruscor}
For $\A$ satisfying the conditions \ref{C1}-\ref{C3}, we have isomorphisms \begin{equation}HH^0(M_\phi)\cong \C,HH^1(M_\phi)\cong HH^1(\T_0)\oplus HH^1(\A)^\phi\cong \C^2\oplus HH^1(\A)^\phi  \end{equation} as vector spaces.
\end{cor}
\begin{proof}
This follows from Prop \ref{HHOmfinal} and Prop \ref{CCtorus}.
\end{proof}
Recall $\T_0$ denotes the nodal elliptic curve over $\C$. 
\begin{cor}\label{HHtoruscor2}
If $HH^1(\A)=HH^2(\A)=0$, then $HH^1(M_\phi)\cong\C^2$ and $HH^2(M_\phi)\cong\C$. 
\end{cor}
\begin{rk}
The analogue of Prop \ref{CCtorus} holds for $M_\phi^R$ as well. In other words, \begin{equation}CC^*(M_\phi^R,M_\phi^R)\simeq cocone\big(CC^*(\Om(\Tt_R)_{cdg},\Om(\Tt_R)_{cdg} )\otimes CC^*(\A,\A)\atop\xrightarrow{\tr_*\otimes \phi_*-1} CC^*(\Om(\Tt_R)_{cdg},\Om(\Tt_R)_{cdg} )\otimes CC^*(\A,\A) \big) \end{equation}
where $\otimes$ denotes the $q$-adic completion of the tensor product over $\C$. The proof works similarly. One can alternatively use the semi-continuity (Lemma \ref{semicont}) since the K\"unneth map and the map in (\ref{eq:equivariancecocone}) admit natural deformations over $R$.
\end{rk}
\begin{defn}
Let $\gamma_\phi\in HH^1(M_\phi)$ (resp. $\gamma^R_\phi\in HH^1(M_\phi^R)$) denote the class obtained by ``descent'' of $\gamma_\Om\otimes 1$ (resp. $\gamma^R_\Om\otimes 1$).
\end{defn}
Similar to $\gamma_\Om$ and $\gamma_\Om^R$, these classes come as the infinitesimal action of $\G$. This will be shown in Corollary \ref{postgammacor} for $\gamma_\Om^R$. 
\section{A family of endo-functors of $M_\phi$}\label{sec:family}
\subsection{Introduction}
In this section, we will use $\cG_R\subset\Tt_R\times\Tt_R\times Spf(A_R)$, resp. $\cG:=\cG_R|_{q=0}\subset \Tt_0\times\Tt_0\times Spec(A)$ to define explicit modules over $M_\phi^R\otimes M_\phi^{R,op}\otimes A_R$, resp. $M_\phi\otimes M_\phi^{op}\otimes A$, i.e.``families of bimodules parametrized by $A_R$, resp. $A$". We can see them as bimodules over $M_\phi^R$, resp. $M_\phi$, taking values in $A_R$-modules (resp. $A$-modules). 

First, define a bimodule over $\Om(\Tt_R)_{cdg}$ with values in $A_R$-modules by the formula 
\begin{equation}\label{eq:bigbimodR}
(\scrF,\scrF')\mapsto hom^\cdot_{\Om_{\Tt_R\times\Tt_R}}(q^*(R(\scrF)_R),p^*(R(\scrF')_R)\otimes_{\Om_{\Tt_R\times\Tt_R}}\Om_{\cG_R})
\end{equation}
Here, as before $q$ and $p$ are projections onto first and second factor respectively. Recall, the $R(\scrF)_R$ and $R(\scrF')_R$ are ``pseudo-complexes" of sheaves, i.e. graded sheaves whose $d^2$ is divisible by $q\in R$. See Definition \ref{defn:psdcx} and Section \ref{subsec:defoofom}. Tensor product is taken in each degree separately and hom-complex is as in ordinary complexes. Homomorphisms are over $\Om_{\Tt_R\times\Tt_R}$; hence, we obtain an $A_R$-module, which is flat by Lemma \ref{flatsub}. Denote the $A_R$-semi-flat bimodule defined by (\ref{eq:bigbimodR}) by $\scrG_{R}^{pre}$ (A pseudo-complex is called $A_R$-semi-flat if it is flat over $A_R$ in each degree. Similarly, $A_R$-semi-flatness of a bimodule $\fM$ means each $\fM(L,L')$ is an $A_R$-semi-flat pseudo-complex, and the bimodule maps are $A_R$ linear.). 
The only subtlety with semi-flatness of \ref{eq:bigbimodR} is that it involves infinite products of flat $A_R$-modules. However, this does not cause a problem as the flatness of these infinite products can be shown explicitly, or alternatively one can use \cite[Theorem 2.1]{chase}.

Similarly define a bimodule over $\Om(\Tt_0)_{dg}$ with values in $A$-modules by 
\begin{equation}\label{eq:bigbimod0}
(\scrF,\scrF')\mapsto hom^\cdot_{\Om_{\Tt_0\times\Tt_0}}(q^*R(\scrF),p^*R(\scrF')\otimes_{\Om_{\Tt_0\times\Tt_0}}\Om_{\cG})
\end{equation}
This bimodule is the restriction of the bimodule $\scrG_{R}^{pre}$ defined by (\ref{eq:bigbimodR}) to $q=0$. It is again $A$-semi-flat. Denote it by $\scrG^{pre}$.

Both $\cG_R$ and $\cG$ are invariant under the action of $\tr\times\tr\times 1$. This implies $\scrG_{R}^{pre}$ and $\scrG^{pre}$ satisfy the invariance condition (i.e. carry a  $\Z_\Delta$-equivariant structure) in Section \ref{sec:bimodgen}. Therefore, so does the $A_R$ (resp. $A$)-valued $\Om(\Tt_R)_{cdg}\otimes\A$ (resp. $\Om(\Tt_0)_{dg}\otimes \A$)-bimodule $\scrG_{R}^{pre}\otimes_\C\Delta_\A$ (resp. $\scrG^{pre}\otimes_\C\Delta_\A$). Recall that $\Z_\Delta$ is the diagonal action corresponding to action generated by $\tr\times\phi$ on $\Om(\Tt_R)_{cdg}\otimes\A$ (resp. $\Om(\Tt_0)_{dg}\otimes \A$).
\begin{defn}
Let $\scrG_R$ denote the $A_R$-valued $M_\phi^R$-bimodule obtained by descent of $\scrG_{R}^{pre}\otimes_\C\Delta_\A$ as in Section \ref{sec:bimodgen}. Similarly, let $\scrG$ denote the $A$-valued $M_\phi$-bimodule obtained by descent of $\scrG^{pre}\otimes_\C\Delta_\A$. 
\end{defn}
\begin{rk} Clearly, $\scrG=\scrG_R|_{q=0}$. Also, $\scrG_R$ and $\scrG$ are semi-flat (over $A_R$, resp. $A$) as well, i.e. $\scrG_R(L,L')$ is flat over $A_R$ in each degree (and same for $\scrG(L,L')$). 
\end{rk}
\subsection{Review of generalities on families of objects and their infinitesimal change}\label{subsec:familyreview}
In this subsection, we will recall how to make notions such as families of (bi)modules and their infinitesimal change precise. We will mostly follow the first section of \cite{flux}. We will write the definitions for curved algebras over $R$; however, it works for curved categories over other pro-finite local rings as well (hence for uncurved categories). Contrary to most of the rest of the paper we will work with $A_\infty$-algebras/categories and modules, instead of dg algebras/categories. These can be considered as a special case of $A_\infty$-algebras. The only major difference is in the homomorphisms between them; for instance, homomorphisms of $A_\infty$-modules are automatically derived. We used the notation $RHom$ to remove any ambiguity before, but below the hom-complexes are complexes of $A_\infty$-morphisms.

First a preliminary definition:
\begin{defn}\label{defn:psdcx}A pseudo-complex over the local ring $R=\C[[q]]$ is a graded(and complete in each degree) $R$-module $C^*$ and a degree $1$ endomorphism, ``the differential'', $d$ such that $d^2$ is a multiple of $q\in R$. Pseudo-complexes form a curved dg category, where the homomorphisms of a given degree are given by graded module homomorphisms and the curvature element is $d^2$. We denote this category by $\mathcal{C}_{cdg}(R)$. \end{defn}
\begin{defn}
Similarly, we can form a curved category of pseudo-complexes over $A_R$, which we denote by $\mathcal{C}_{cdg}(A_R)$. Let $\mathcal C_{cdg}^{sf}(A_R)$ denote the full (curved) subcategory of $\mathcal{C}_{cdg}(A_R)$ spanned by pseudo-complexes that are $q$-adically complete and topologically free (i.e. $q$-adic completion of a free $A_R$-module) in each degree and whose restrictions to $q=0$ give K-projective complexes of $A$-modules (recall that an unbounded complex $C^\cdot$ of $A$-modules is called $K$-projective if $hom^\cdot_A(C^\cdot,\cdot)$ sends acyclic complexes to acyclic complexes).
\end{defn}
\begin{defn}\label{curvedfamily}
Let $\B$ be a curved $A_\infty$-algebra over $R=\C[[q]]$. A family of (right) modules parametrized by $Spf(A_R)$ is an $A_\infty$-homomorphism $\B^{op}\rightarrow \mathcal{C}_{cdg}^{sf}(A_R)$. In other words it is a module $\fM$ over $\B$ such that each $\fM(b)^i$ is a topologically free $A_R$-module, $\fM(b)|_{q=0}$ is K-projective over $A$ and the structure maps are $A_R$-linear and continuous. Families of left modules and bimodules are defined similarly. \end{defn}
\begin{defn}\label{uncurvedfamily}
If $\B_0$ is an uncurved category over $\C$, then a family over it is defined similarly as a functor from $\B^{op}_0$ to K-projective complexes of $A$-modules.
\end{defn}
\begin{rk}
$\scrG_R$ fails to be ``semi-K-projective''(i.e. each $\fM(b)|_{q=0}$ is K-projective) but we will pass to a semi-free replacement of it satisfying K-projectivity condition. The phrase ``$A_R$-valued'' bimodule/module refers to such a bimodule/module with complete $A_R$-linear structure as above, where freeness/K-projectivity conditions are dropped. In other words, a given module $\fM$ is $A_R$-valued if $\fM(b)$ is an $A_R$-module in each degree for every object $b$, and the structure maps of the module are linear over $A_R$.
\end{rk}
Now, let us make the condition \ref{G1} precise:
\begin{defn}\label{compactfamdef}
Given an $A_\infty$-category $\B_0$ over $\C$, define a coherent twi-family of $\B_0$-modules parametrized by $A$ to be a twisted complex of objects $b\otimes M\in ob(\B_0\otimes \mathcal C_{dg}(A) )$, where $b\in ob(\B_0)$ and $M$ is a K-projective complex of flat $A$-modules whose cohomology is bounded and coherent (finitely generated) over $A$. In other words, this is the category spanned by such objects $b\otimes M$, $b'\otimes M'$ whose hom-complex is given by $\B_0(b,b')\otimes hom_A^\cdot(M,M')$. There exists a Yoneda functor from the category of twi-families to $(\B_0)_A^{mod}$, the category of families of $\B_0$-modules parametrized by $Spec(A)$ with $A$-linear morphisms, see Definition \ref{arfamilycat}. This functor can be shown to be cohomologically fully faithful. A family that is quasi-isomorphic to an object in the image of the idempotent completion of coherent twi-families is called a coherent family.
\end{defn}
\begin{rk}
Coherent twi-families are analogous to families of twisted complexes defined in \cite{flux}, except we allow (K-projective replacements of) coherent sheaves $M$ that are not just vector bundles over the base curve. See \cite[Section 1f]{flux}. Yoneda lemma- for this Yoneda functor- can be shown in a way similar to well known Yoneda lemma for $A_\infty$-(bi)modules.  
When we write $b\otimes M$ for a finitely generated module $M$ over $A$, we will mean a K-projective replacement of $M$. We will elaborate more on this later (to replace in a way compatible with $\tr$ and action). We denote the replacements by $\scrG^{pre,sf},\scrG^{pre,sf}_R,\scrG^{sf},\scrG^{sf}_R$, etc. 
\end{rk}
\begin{lem}
$K$-projective replacements of families exist, and they are unique up to quasi-isomorphism of families.
\end{lem}
\begin{proof}
The existence follows from the existence of functorial $K$-projective replacement functors on $\scrC_{dg}(A)$ that extend to $\scrC_{cdg}(A_R)\to \scrC_{cdg}^{sf}(A_R)\subset \scrC_{cdg}(A_R)$. See the construction in \cite{spaltenstein}. Their uniqueness follow from a length filtration argument similar to \cite[Lemma 1.10]{flux}. More precisely, one only needs to show that the homomorphisms from $\fM^{sf}$ to $\fM'$ is acyclic when $\fM^{sf}$ is $K$-projective at $q=0$, and $\fM'|_{q=0}$ is acyclic. As $hom(\fM,\fM')$ deforms $hom(\fM|_{q=0},\fM'|_{q=0} )$, and acyclicity of the latter implies that of the former, we can focus on $hom(\fM|_{q=0},\fM'|_{q=0} )$. The length filtration argument, and $K$-projectivity of $\fM|_{q=0}$ implies the result.
\end{proof}
\begin{rk}
Presumably, one may modify the definition of a family as a functor from $\B_0$ to Ind-coherent sheaves over $A$ and realize coherent families as compact objects of category of such. However we do not need this.
\end{rk}
Now, we state a lemma:
\begin{lem}\label{cptfamilylem}
If $\fM$ and $\fM'$ are coherent families and $\B_0$ satisfies conditions \ref{C1}-\ref{C2} (see Section \ref{subsec:introcatcon}), then $(\B_0)_{A}^{mod}(\fM,\fM')$ is cohomologically bounded below and finitely generated over $A$.
\end{lem}
\begin{proof}
This follows from the analogous statement for coherent twi-families and Yoneda lemma.
\end{proof}
\begin{cor}\label{cptfamilycor}
Let $\fM$ and $\fM'$ are families over a curved category $\B$ over $R$ such that their restrictions to $q=0$ are coherent, and $\B|_{q=0}$ satisfies \ref{C1}-\ref{C2}. Then $\B_{A_R}^{mod}(\fM,\fM')$ is cohomologically bounded below and cohomologically finitely generated over $A_R$.
\end{cor}
\begin{proof}
This follows from Lemma \ref{cptfamilylem} and \ref{semicont2}.
\end{proof}
\begin{rk}\label{perfseirk}
Definition \ref{curvedfamily} and \ref{uncurvedfamily} are obvious generalizations of definition of families of modules over smooth curves in \cite{flux}. Moreover, one can define pushforward of families along $Spec(\C[t])\rightarrow Spec(A)$ etc. and the pushforward of (a direct summand of) a family of modules coming from a family of twisted complexes is obviously coherent. For instance, when $\B_0=\C$, a vector bundle over $Spec(\C[t])$ gives such a family that pushes forward to a coherent family.
\end{rk}
\begin{figure}
	\centering
	\includegraphics[height=4cm]{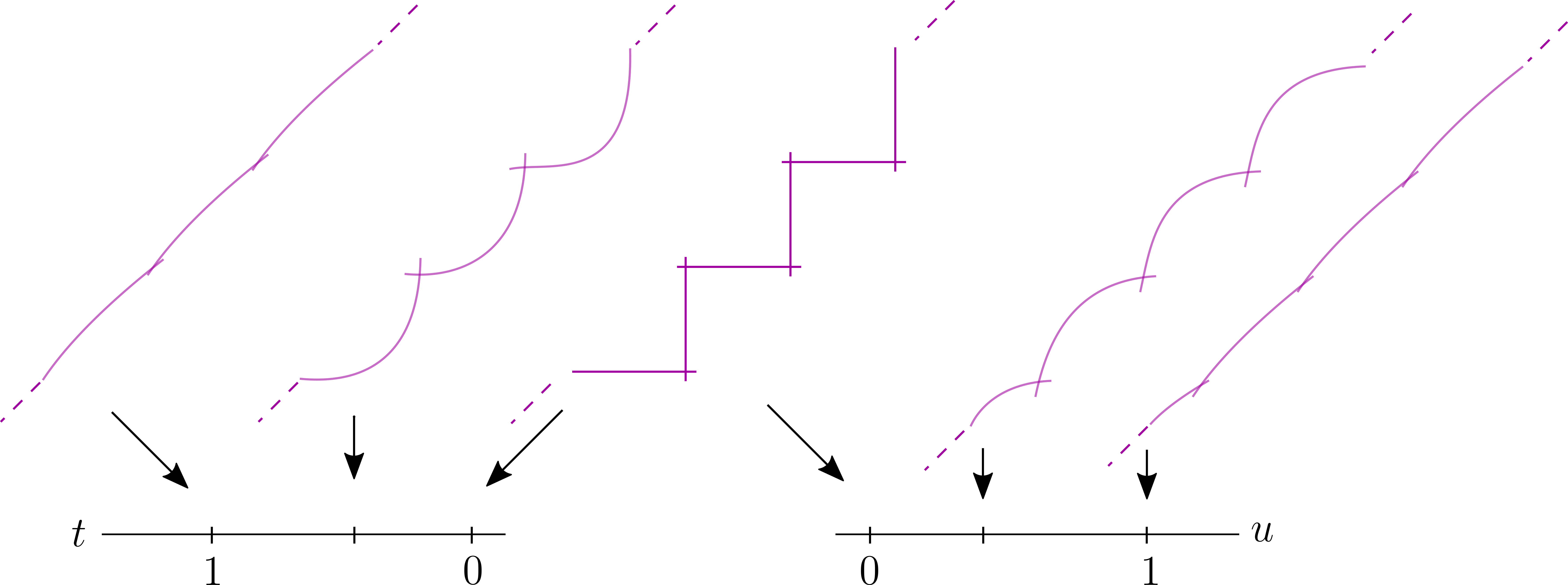}
	\caption{The graph $\cG$ shown separately over $t$ and $u$ axes of $Spec(\C[t,u]/(tu))$}
	\label{figure:graph}
\end{figure}
\begin{figure}
	\centering
	\includegraphics[height=4cm]{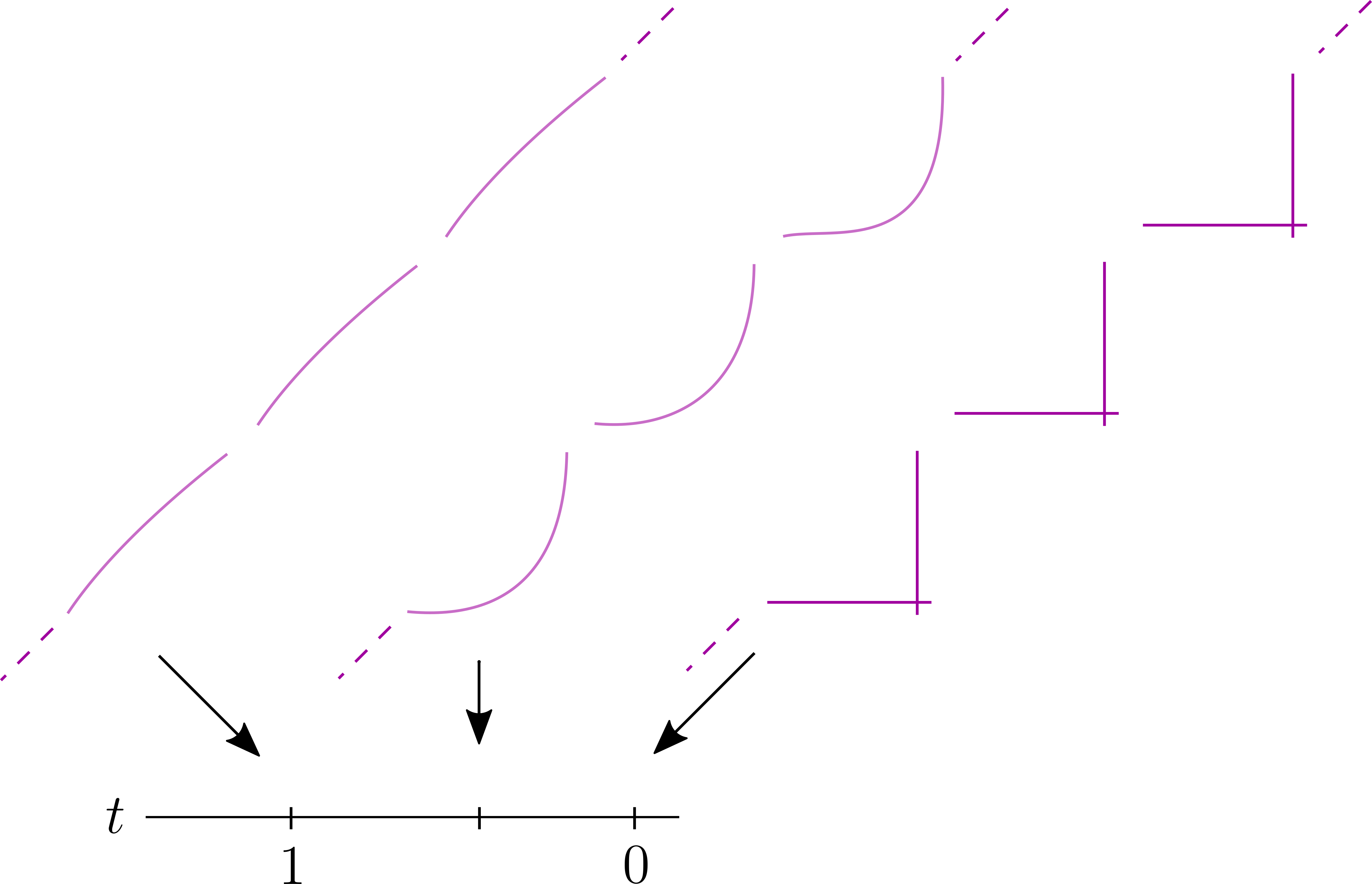}
	\caption{The relative partial normalization $\tilde{\cG_t}$ which can also be seen as a degeneration of $\G$ action on $\mathbb P^1\times \Z$}
	\label{figure:normalizedgraph}
\end{figure}
\begin{prop}\label{compactfamily}
The family $\scrG_R^{sf}|_{q=0}\simeq \scrG^{sf}$ is coherent. 	
\end{prop}
\begin{proof}
The proof is similar to proof of Prop \ref{smoothness}; thus, we skip some details. 

First, apply Lemma \ref{dualitylemma} when $X=\cG$, $Y=\Tt_0\times\Tt_0\times Spec(A)$, $Z=\Tt_0$ and $p$ is the second projection. It implies the family $\scrG^{pre,sf}$ is ($\Z_\Delta$-equivariantly) quasi-isomorphic to \begin{equation}\label{birequ}
(\scrF,\scrF')\mapsto ``RHom_{\Tt_0\times \Tt_0}(q^*\scrF\otimes \Om_{\cG}^\vee,p^*\scrF')"
\end{equation}
As before (e.g. (\ref{eq:bimoddefnquotientmarks}), (\ref{eq:quotationdefn1}) and (\ref{blah})), we put quotation marks since we use a suitable enhancement of (\ref{birequ}) to a dg functor (see also Remark \ref{spalsrk}). Also note that  $RHom_{\Tt_0\times \Tt_0}(\cdot,\cdot)$ in (\ref{birequ}) is an abbreviation for 
\begin{equation}
RHom_{\Tt_0\times \Tt_0\times Spec(A)/Spec(A)}(\cdot,\cdot)
\end{equation}
We can use the notation of Section \ref{sec:construction} and denote this family by $\fM_{\Om_\scrG^{\vee}}'$.
Tensoring $\Om_{\cG}$ with the (pull-back of) short exact sequence $0\rightarrow A\rightarrow A/u\oplus A/t\rightarrow A/(u,t)\rightarrow 0$, we obtain a quasi-isomorphism
\begin{equation}
\Om_{\cG}^\vee\simeq cocone (\Om_{\cG}^\vee|_{t=0}\oplus \Om_{\cG}^\vee|_{u=0}\rightarrow \Om_{\cG}^\vee|_{t=u=0} )
\end{equation}
where $\Om_{\cG}^\vee|_{t=0}$ refers to push-forward of derived restriction of $\Om_{\cG}^\vee$ along $Spec(\C[u])\rightarrow Spec(A)$ (similar for others). This quasi-isomorphism is compatible with natural $\Z_\Delta$-actions and it implies 
\begin{equation}
\fM'_{\Om_\cG^\vee}\simeq cone( \fM'_{\Om_{\cG}^\vee|_{t=u=0}} \rightarrow \fM'_{\Om_{\cG}^\vee|_{t=0}}\oplus \fM'_{\Om_{\cG}^\vee|_{u=0}})
\end{equation}
$\Z_\Delta$-equivariantly. Hence, it is sufficient to prove $\fM'_{\Om_{\cG}^\vee|_{t=0}}$, $\fM'_{\Om_{\cG}^\vee|_{u=0}}$ and $ \fM'_{\Om_{\cG}^\vee|_{t=u=0}}$ descend to coherent families over $M_\phi$, after tensoring with $\A$. That $ \fM'_{\Om_{\cG}^\vee|_{t=u=0}}$ is coherent follows from the others. Also, the proof for $\fM'_{\Om_{\cG}^\vee|_{t=0}}$ is almost the same as the proof for $\fM'_{\Om_{\cG}^\vee|_{u=0}}$; hence, we prove coherence only for the latter. $\fM'_{\Om_{\cG}^\vee|_{u=0}}$- as a family over $Spec(\C[u,t]/(ut) )$- can be seen as the push-forward of the family $\fM'_{\Om_{\cG|_{u=0}}^\vee}$ considered as a family of bimodules over $Spec(\C[t])$.

Consider the subscheme $\cG|_{u=0}\subset \Tt_0\times\Tt_0\times Spec(\C[t])$, where we identify $\C[t]$ with $A/(u)$. We proceed similar to Prop \ref{smoothness}. The subscheme $\cG_t:=\cG|_{u=0}$ is given by 
\begin{equation}
tY_{i+1}=Y_{i+1}',tX_i'=X_i,Y_{i+1}X_i'=0\text{ on } U_{i+1/2}\times U_{i+1/2}\times Spec(\C[t])
\end{equation}
\begin{equation}
Y_{i+1}=X_{i-1}'=0,Y_{i}'X_i=t\text{ on }U_{i+1/2}\times U_{i-1/2}\times Spec(\C[t]) 
\end{equation}
It is flat over $\C[t]$ by Lemma \ref{flatsub} and can be seen as a flat degeneration of the graph of $Spec(\C[t,t^{-1}])$-action on $\Tt_0$. Consider the normalization $\pi:\mathbb{P}^1\times \Z\rightarrow \Tt_0$, where \begin{equation}\pi|_{\mathbb{P}^1\times\{n\}}:\mathbb{P}^1\times\{n\}\rightarrow C_n\end{equation} is an isomorphism. We can see $X_i,Y_i$ as coordinates of $\mathbb{P}^1\times\{i\}$ satisfying $X_iY_i=1$. Let $\tilde{\cG_t}\subset \mathbb{P}^1\times\Z\times\mathbb{P}^1\times\Z\times Spec(\C[t])$ denote a natural flat degeneration of the graph of $\G=Spec(\C[t,t^{-1}])$-action on $\mathbb{P}^1\times\Z$ given by \begin{equation}Y_{i}\mapsto tY_i,X_i\mapsto t^{-1}X_i \end{equation}
More precisely, $\tilde{\cG_t}$ is given by
\begin{equation}Y_i'=tY_i\text{ on }Spec(\C[Y_i,Y_i',t]) \cong\{X_i\neq 0 \}\times \{X'_i\neq 0 \}\times Spec(\C[t]) \end{equation}
\begin{equation}X_i=tX_i'\text{ on }Spec(\C[X_i,X_i',t]) \cong\{Y_i\neq 0 \}\times \{Y'_i\neq 0 \}\times Spec(\C[t]) \end{equation}
\begin{equation}X_iY_i'=t\text{ on }Spec(\C[X_i,Y_i',t]) \cong\{Y_i\neq 0 \}\times \{X'_i\neq 0 \}\times Spec(\C[t]) \end{equation}
The domains on the right side are considered as subsets of \begin{equation}\mathbb{P}^1\times\{i\}\times\mathbb{P}^1\times\{i\}\times Spec(\C[t])\end{equation}
See Figure \ref{figure:normalizedgraph} for a picture of $\tilde{\cG_t}$, and Figure \ref{figure:graph} for a picture of $\cG$, where $\cG_t=\cG|_{u=0}$ and $\cG|_{t=0}$ are drawn separately.

It is easy to check that $\pi\times \pi\times 1$ restricts to a morphism $\tilde{\pi}:\tilde{\cG_t}\rightarrow\cG_t$. It is an isomorphism over the part of $\cG_t$ in $U_{i+1/2}\times U_{i-1/2}\times Spec(\C[t])$. The part of $\cG_t$ in $U_{i+1/2}\times U_{i+1/2}\times Spec(\C[t])$ has coordinate ring
\begin{equation}\C[X_i,Y_{i+1},X_i',Y_{i+1}',t]/(X_iY_{i+1},X'_iY'_{i+1},Y_{i+1}X_i',Y_{i+1}'-tY_{i+1},X_i-tX'_i )\atop\cong \C[Y_{i+1},X_i',t]/(Y_{i+1}X_i') \end{equation} 
and the part of $\tilde{\cG_t}$ over it has coordinate ring 
\begin{equation}\C[X_i,X_i',t]/(X_i-tX_i')\times \C[Y_{i+1},Y_{i+1}',t]/(Y_{i+1}'-tY_{i+1})\cong\C[X_i',t]\times \C[Y_{i+1},t] \end{equation}
The map induced on coordinate rings from the former to the latter is given by \begin{equation}X_i'\mapsto (X_i',0),Y_{i+1}\mapsto (0,Y_{i+1}),t\mapsto (t,t)\end{equation}

Hence, it is isomorphic to normalization map of an affine nodal curve relative to $\C[t]$. This description implies the map  $\Om_{\cG_t}\rightarrow\tilde{\pi}_*\Om_{\tilde{\cG_t}}$ corresponding to $\tilde{\cG_t}\rightarrow\cG_t$ is injective with cokernel $\bigoplus_{i\in\Z}\Om_{(x_{i+1/2},x_{i+1/2})}\boxtimes \C[t]$, where $x_{i+1/2}$ still denotes the nodal point in $U_{i+1/2}$. In other words, we have a short exact sequence 
\begin{equation}\label{eq:ses3}
0\rightarrow \Om_{\cG_t}\rightarrow\tilde{\pi}_*\Om_{\tilde{\cG_t}}\rightarrow \bigoplus_{i\in\Z}(\Om_{x_{i+1/2}}\boxtimes\Om_{x_{i+1/2}})\boxtimes \C[t]\rightarrow 0
\end{equation}This is the analogue of (\ref{eq:ses1}) in Prop \ref{smoothness}. Moreover, using smoothness of $\mathbb{P}^1\times\{i \}$, we can resolve $\Om_{\tilde{\cG_t}}|_{\mathbb{P}^1\times\{i \}\times\mathbb{P}^1\times\{i \}\times Spec(\C[t])}$ by sheaves of type $E\boxtimes E'$, where $E$ and $E'$ are coherent. More precisely, it is quasi-isomorphic to a complex of sheaves of type $E\boxtimes E'$(we do not need to consider its direct summands as $D^b(\mathbb{P}^1\times \mathbb{P}^1)$ is generated by exterior products, but it would not affect us.) Concretely, one can use \begin{equation}\label{eq:respr}
\Om(-1)\boxtimes\Om(-1)\boxtimes  \C[t]\xrightarrow{X\boxtimes Y'\boxtimes 1-Y\boxtimes X'\boxtimes t}\Om\boxtimes\Om\boxtimes \C[t]\end{equation} (\ref{eq:ses3}) and resolution maps from (\ref{eq:respr}) can be made invariant under $\tr\times \tr\times 1$ and hence $\tilde{\pi}_*\Om_{\tilde{\cG}_t}$ is quasi-isomorphic to a finite complex of sheaves of type $\bigoplus_{i\in\Z}({\tr^i\tilde{E'}\boxtimes \tr^i\tilde{E'}})\boxtimes\C[t]$, where $\tilde{E}'$ and $\tilde{E}''$ are push-forwards of compactly supported coherent sheaves on the normalization $\mathbb{P}^1\times \Z$ (hence, isomorphic to twisted complexes over $\Om(\Tt_0)_{dg}$). This complex is the analogue of (\ref{eq:ses2}) in Prop \ref{smoothness}. The same holds for $\Om_{\cG_t}$ by (\ref{eq:ses3}) and for $\Om_{\cG_t}^\vee$ by taking duals. 

Let $\mathbf{E}=\bigoplus_{i\in\Z}({\tr^i\tilde{E'}^\vee\boxtimes \tr^i\tilde{E'}^\vee})\boxtimes\C[t]$. A $\C[t]$-relative version of the idea in the proof of Prop \ref{smoothness} shows that  $\fM'_\mathbf{E}\otimes \A$ , as a family over $Spec(\C[t])$ descends to a family of bimodules over $M_\phi$ that is representable by a (family of) twisted complexes (see \cite{flux}). Hence, its push-forward along $Spec(\C[t])\rightarrow Spec(A)$ is coherent by Remark \ref{perfseirk} and the same holds for $\fM'_{\Om_{\cG} ^\vee|_{u=0}}\simeq \fM'_{\Om_{\cG_t} ^\vee}$, which finishes the proof.
\end{proof}
\begin{rk}
One can presumably show that $\scrG$ is representable by a family of twisted complexes and $\scrG_R$ is representable by an (unobstructed) twisted complex over $M_\phi^R\otimes M_\phi^{R,op}\otimes A_R$, as it is a deformation of $\scrG$. However, we do not need this.	
\end{rk}
Before making infinitesimal change precise, we need a few more definitions:
\begin{defn}\label{arfamilycat}
Let $\B_{A_R}^{pre}$ be the category whose objects are families of modules over $\B$ parametrized by $Spf(A_R)$ and morphisms $\fM\rightarrow \fM'$ are pre-module homomorphisms $\fM\rightarrow \fM'$ over $\B$ in the sense of \cite{seidelbook}. This is a dg category. Let $\B_{A_R}^{mod}$ (or simply $\B_{A_R}$ abusing the notation) denote the subcategory where the pre-module homomorphisms are the $A_R$-linear ones. One can define such categories for left modules and bimodules similarly.
\end{defn}
\begin{rk}
The superscripts ``pre'' in Definition \ref{arfamilycat} and in $\scrG^{pre}_R$ are unrelated. 
\end{rk}
\begin{rk}\label{arfamilycatrk}
More explicitly, a morphism of $\B^{pre}_{A_R}$
can be defined as a sequence of $R$-linear maps 
\begin{flalign*}
f^1:&\fM\rightarrow\fM\\
f^2:&\fM\otimes \B\rightarrow \fM [-1]\\ \dots
\end{flalign*}
One obtains a morphism of $\B_{A_R}^{mod}$ if $A_R$-linearity is imposed.
\end{rk}
\begin{rk}
Notice the hom-sets of $\B_{A_R}^{pre}$ have the structure of an $A_R\otimes A_R^{op}$-module, which comes from the algebra maps \begin{equation}A_R\rightarrow \B_{A_R}^{pre}(\fM,\fM) \end{equation} for each $\fM$. The algebra map sends $a\in A_R$ to $f=(f^1=a,0,0,\dots)$, i.e. to the multiplication by $a$. 
\end{rk}
Now, we will make ``the infinitesimal change of the family'' precise, still following \cite{flux}. For simplicity let us confine ourselves to $A_\infty$-algebras, keeping in mind that one can pass to them from $A_\infty$-categories via constructions similar to total algebra.

Let $\B$ be a curved $A_\infty$-algebra over $R=\C[[q]]$. Let $B(\B)$ denote the graded $q$-adic completion of $T(\B[1])=\bigoplus_{n\geq 0}\B[1]^{\otimes n}$. Recall $B(\B)$ is a (co-unital) coalgebra (in the category of $q$-adically complete, graded $R$-modules) and one can see the $A_\infty$-structure as a coderivation $\mu$ of degree $1$ satisfying $\mu\circ\mu=\frac{1}{2}[\mu,\mu]=0$ and such that $\mu\circ\epsilon$ is a multiple of $q\in R$, where $\epsilon:R\rightarrow B(\B)$ is the natural coaugmentation given by inclusion of $\B^{\otimes 0} =R$. 

A right $A_\infty$-module structure on graded complete $R$-module $\fM$ is given by a degree one endomorphism of the comodule $\fM\otimes B(\B)$ satisfying co-Leibniz rule with $\mu$. In other words, it is a dg comodule over the dg coalgebra $(B(\B),\mu)$. See \cite{koso} for details. Note again the tensor product denotes the completed tensor product over $R$. In this language, the morphisms of $\B_{A_R}^{pre}$ are comodule homomorphisms of  $\fM\otimes B(\B)$ and the differential on the hom-set is given by the commutator with the endomorphism corresponding to the $A_\infty$-structure.
\begin{rk}
If the module $\fM$ is $A_R$-valued- e.g. if it is a family of modules- then the comodule $\fM\otimes B(\B)$ has an $A_R$-linear structure commuting with the comodule structure maps. 
\end{rk}
To define the infinitesimal change let us introduce an auxiliary notion:
\begin{defn}\label{precondefn}
Let $\fM$ be a family of modules over $\B$(or more generally an $A_R$-valued $\B$-module). A pre-connection on $\fM$ along the derivation $D_{A_R}=t\partial_t-u\partial_u$ (see Appendix \ref{sec:modules}) is an element $\precon\in (\B_{A_R}^{pre})^0(\fM,\fM)$ such that $[\precon,a]=D_{A_R}(a).1_{\fM}$ for every $a\in A_R$ considered as an element of $(\B_{A_R}^{pre})^0(\fM,\fM)$. 
\end{defn}
In other words, a pre-connection is a comodule endomorphism of $\fM\otimes B(\B)$ that satisfies Leibniz rule (with respect to natural $A_R$-linear structure on $\fM\otimes B(\B)$). Families of modules (over $A_\infty$-categories over $\C$) parametrized by $Spec(A)$ and pre-connections on them along $D_A$ can be defined analogously. 
\begin{rk}
The endomorphisms of the comodule $\fM\otimes B(\B)$ can be shown to be in one to one correspondence with the pre-module endomorphisms in the sense of \cite{seidelbook}. Hence, $\precon$ can be defined as a sequence of maps 
\begin{flalign*}
\precon^1:&\fM\rightarrow\fM\\
\precon^2:&\fM\otimes \B\rightarrow \fM [-1]\\ \dots
\end{flalign*}such that $\precon^1$ is a connection of the graded module $\fM$(up to sign) and $\precon^i$, $i\geq 2$ are $A_R$-linear and there is no further constraint. This is how they are defined in \cite{flux} and this approach shows that $A_R$-semi-freeness implies existence of pre-connections.
\end{rk}
\begin{defn}
Given a pre-connection $\precon$ on $\fM$ its deformation class is $def(\precon):=d(\precon)\in (\B_{A_R}^{pre})^1(\fM,\fM)$, i.e. the differential of $\precon$ in the dg category $\B_{A_R}^{pre}$. It is $A_R$-linear due to commutation relation in Definition \ref{precondefn}. Its class in $H^1(\B_{A_R}(\fM,\fM))$ will be denoted by $Def(\precon)$.
\end{defn}
\begin{rk}
Two pre-connections on $\fM$ differ by an element of $(\B_{A_R})^0(\fM,\fM)$; hence, the classes of their differentials are the same in the category $\B_{A_R}$. We denote it by $Def(\fM)$ as well.
\end{rk}
Let us show the naturality of this class:
\begin{lem}
Let $\fM$ and $\fM'$ be two families of modules with pre-connections $\precon$ and $\precon'$ resp. Let $f:\fM\rightarrow \fM'$ be a closed morphism in $\B_{A_R}^{0}$. Then the images of $Def(\fM)$ and $Def(\fM')$ coincide under the natural maps \begin{equation}H^1(\B_{A_R}(\fM,\fM))\xrightarrow{f\circ(\cdot)}H^1(\B_{A_R}(\fM,\fM'))\xleftarrow{(\cdot)\circ f }H^1(\B_{A_R}(\fM',\fM')) \end{equation}
\end{lem}
\begin{proof}
Consider the pre-module homomorphism \begin{equation}\precon'\circ f-f\circ \precon:\fM\rightarrow\fM' \end{equation}It is $A_R$-linear, as $f$ is $A_R$-linear; hence, it falls into category $\B_{A_R}$. Its differential is equal to \begin{equation}d(\precon')\circ f-f\circ d(\precon)=def(\precon')\circ f-f\circ def(\precon) \end{equation}Hence, the $\B_{A_R}$-classes of $def(\precon')\circ f$ and $f\circ def(\precon)$ are the same. 
\end{proof}
\begin{cor}Let $f:\fM\rightarrow\fM'$ be a quasi-isomorphism of families with pre-connections. Then $Def(\fM)$ corresponds to $Def(\fM')$ under the natural isomorphism \begin{equation}H^1(\B_{A_R}(\fM,\fM) )\cong H^1(\B_{A_R}(\fM',\fM') ) \end{equation}
\end{cor}
We also want to show naturality of deformation classes under Morita equivalences. Let $\B'$ be another curved $A_\infty$-algebra and let $X$ be a $\B$-$\B'$-bimodule. For a definition of $A_\infty$-bimodules see \cite{categoricaldynamics}.
One can also see a bimodule as a graded complete module over $R$ such that the bicomodule $B(\B)\otimes X\otimes B(\B')$(again tensor product is over $R$ and completed) has a differential compatible with the coderivations of $B(\B)$ and $B(\B')$. Such a bimodule $X$ induces a dg functor
\begin{equation}\label{eq:conv1}
\B^{mod}\rightarrow\B'^{mod} 
\end{equation} 
between the categories of right modules as well as dg functors \begin{equation}\label{eq:conv2}
\B_{A_R}^{pre}\rightarrow\B'^{pre}_{A_R}
\end{equation}
\begin{equation}\label{eq:conv3}
\B_{A_R}^{mod}\rightarrow\B'^{mod}_{A_R} 
\end{equation}
between the categories of families. It is given by $(\cdot)\otimes_\B X$. See \cite{generation} for a definition. Note also, (\ref{eq:conv2}) (resp. (\ref{eq:conv3}) is $A_R\otimes A_R$ (resp. $A_R$)-linear.
\begin{defn}\label{moritadefn}
$X$ is called a Morita equivalence if there exists a $\B'$-$\B$ bimodule $Y$ such that
\begin{equation}Y\otimes_\B X\simeq \B'\text{ in the dg category of bimodules over }\B' \end{equation}
\begin{equation}X\otimes_{\B'} Y\simeq \B\text{ in the dg category of bimodules over }\B \end{equation}In this case, $\B$ and $\B'$ are called Morita equivalent.
\end{defn}
If $X$ is a Morita equivalence, then the induced functors (\ref{eq:conv1}),(\ref{eq:conv2}) and (\ref{eq:conv3}) are quasi-equivalences.
For a definition of $X\otimes_{\B'}Y$ and $Y\otimes_\B X$ see \cite{generation}. One can also define them as cotensor products of corresponding dg bicomodules, clarifying the module structure. As the transformation (\ref{eq:conv2}) is $A_R\otimes A_R$-linear and strictly unital, it sends an endomorphism $\precon$ satisfying the commutation rule $[\precon,a]=D(a).1$ to such an endomorphism of the image. In other words, it produces a pre-connection of the image and clearly $def(\precon)$ is sent to the deformation class of the image. Hence, we have proved
\begin{cor}\label{defmorita}
Let $X$ be a $\B$-$\B'$-bimodule admitting an ``inverse'' $Y$ as above and thus inducing an equivalence $\Phi_X:\B^{mod}_{A_R}\rightarrow \B'^{mod}_{A_R}$. Then, for a given family (with connection) $\fM\in Ob(\B^{mod}_{A_R})$, the deformation class $Def(\fM)$ is sent to $Def(\Phi_X(\fM) )$ under \begin{equation}\Phi_X:H^1(\B_{A_R}(\fM,\fM))\rightarrow H^1(\B'_{A_R}(\Phi(\fM),\Phi(\fM) )) \end{equation}
\end{cor}
Now, let us make the meaning of infinitesimal change precise following \cite{flux}. First, recall there exists a natural map \begin{equation}\label{eq:restright}
CC^*(\B,\B)\rightarrow CC^*(\B^{mod},\B^{mod}) \end{equation} inducing a chain map \begin{equation}\label{eq:restrrighttomod}CC^*(\B,\B)\rightarrow \B^{mod}(\mathfrak{N},\mathfrak{N}) \end{equation} for every $\B$-module $\mathfrak{N}$. Seen as a map $T\B[1]\otimes \mathfrak{N}\rightarrow \mathfrak{N}$, the latter is given by explicit formulas 
\begin{equation}\label{eq:CCmaphom}
-\sum\mu_\mathfrak{N}^i(1_\mathfrak{N}\otimes 1_\B^{\otimes r}\otimes g^j\otimes 1_\B^{\otimes s})
\end{equation}
where $g^j$ denotes the components of a cochain $g\in CC^*(\B,\B)$ (note again the Kozsul signs or see (1.19) in \cite{flux}, up to possible differences in signs). 
Using the same formula, we have \begin{equation}CC^*(\B\otimes A_R,\B\otimes A_R)\rightarrow \B_{A_R}(\fM,\fM) \end{equation} for every family of $\B$-modules. Moreover, any cochain $\gamma\in CC^*(\B,\B)$ induces a cochain in $CC^*(\B\otimes A_R,\B\otimes A_R)$; hence, we have a chain map 
\begin{equation}\label{eq:restrightfamily}
\Gamma_\fM:CC^*(\B,\B)\rightarrow \B_{A_R}(\fM,\fM)
\end{equation}
given by the formula (\ref{eq:CCmaphom}). Indeed, we can simply treat $\fM$ as a $\B$-module to compute the class. Then, the induced $A_\infty$-module endomorphism on $\fM$ is $A_R$-linear. 
\begin{rk}
There are analogues of 	(\ref{eq:restright}) and (\ref{eq:restrrighttomod}) for left modules and bimodules, which also generalizes to families as in (\ref{eq:restrightfamily}). For instance, given a $\B$-$\B$-bimodule (or more generally $\B''$-$\B$-bimodule) $\fM$, we have the map
\begin{equation}
CC^*(\B,\B)\to hom^*_{Bimod(\B,\B)}(\fM,\fM) 
\end{equation} that maps $g\in CC^*(\B,\B)$ to $-\sum \mu^{i'|1|i}_\fM(1_{\B}^{\otimes i'}|1_\fM|1_\B^{\otimes r}\otimes g^j\otimes 1_\B^{\otimes s} )$ with similar sign conventions. The image of $g$ will be denoted by $\Gamma_\fM(1\otimes g)$. 
\end{rk}
\begin{defn}
Let $\gamma\in CC^1(\B,\B)$	be a closed cochain and $\fM$ be a family of right $\B$ modules (resp. $\B$-$\B$-bimodules) admitting a pre-connection. We will say $\fM$ follows $[\gamma]$ (resp. $1\otimes \gamma$) if $Def(\fM)=[\Gamma_\fM(\gamma)]\in H^1(\B_{A_R}(\fM,\fM) )$ (resp. $Def(\fM)=[\Gamma_\fM(1\otimes \gamma)]$). 
\end{defn}
\begin{rk}\label{rkmorita}
$CC^*(\B,\B)$ is quasi-isomorphic to endomorphisms of the diagonal bimodule. $\mathfrak N\otimes_\B \B\simeq \mathfrak N$, for any right module $\mathfrak N$; thus, we have a natural map \begin{equation}CC^*(\B,\B)\simeq hom^*_{\B^e}(\B,\B)\rightarrow hom^*_\B(\mathfrak N\otimes_\B \B,\mathfrak N\otimes_\B \B)\simeq hom^*_\B(\mathfrak N,\mathfrak N) \end{equation}It is possible to show this map is $\pm \Gamma_{\mathfrak N}$ in cohomology (the notation $\Gamma_{\mathfrak N}$ is used for $\B$-modules in general not only families over $A_R$). 
It follows in the setting of Corollary \ref{defmorita} that if a family $\fM$ follows $\gamma$ then $\Phi(\fM)$ follows the class corresponding to $\gamma$ under the quasi-isomorphism $CC^*(\B,\B)\simeq CC^*(\B',\B')$. The same holds for families of left modules and bimodules.
\end{rk}
Now, we want to prove the cohomology groups of hom-complexes between families following the same class admit connections along $D_{A_R}$ (see Definition \ref{connectionalong} for the notion of connections along $A_R$-modules). First some preliminaries:
\begin{defn}
Let $E=(E^\cdot,d)$ be a chain complex of $A_R$-modules (resp. $A$-modules). A pre-connection on $(E^\cdot,d)$ is a choice of connections $D_{E^i}:E^i\rightarrow E^i$ along $D_{A_R}$ (resp. $D_A$) for each $i$. We will denote a pre-connection by $\precon_{E}$, or simply by $\precon$.
\end{defn}
We will mostly drop ``resp. $D_A$'' keeping in mind that the definitions and proofs would go through analogously.
\begin{defn}Define the Atiyah class $at(\precon):E^\cdot\rightarrow E^\cdot[1]$ of a pre-connection to be the differential of $\precon:E^\cdot\rightarrow E^\cdot$ considered as an $R$-linear map. More precisely, \begin{equation}at(\precon):=d\circ \precon-\precon\circ d \end{equation}It is a chain map over $A_R$. 
\end{defn}
\begin{rk}The cohomology class $[at(\precon)]\in hom^\cdot_{A_R}(E^\cdot,E^\cdot)$ is independent of the choice of pre-connection. Denote it by $At(E^\cdot)$. 
\end{rk}
We also include the following, which is proven in \cite{flux}.
\begin{lem}\label{preconcomplexlemma}
$At(E^\cdot)=0$ if and only if one can find a pre-connection that is a chain map over $R$. 
\end{lem}
\begin{proof}
The only if part is clear. Let us prove the if part. If $At(E^\cdot)=0$, that means there exists a pre-connection such that $at(\precon)=d(c)$ for some $c\in hom^0_{A_R}(E^\cdot,E^\cdot)$ such that $at(\precon)=d(c)$. Thus, $\precon-c$ is a degree $0$ chain map, which is still a pre-connection since $c$ is linear over $A_R$. 	
\end{proof}
We will call such a connection a homotopy connection. Let us note a general lemma that will be used later:
\begin{lem}\label{fgcomplex}
Let $C^*$ be a chain complex of complete $A_R$-modules whose cohomology groups are finitely generated over $A_R$ in each degree. Assume $t-1\in A_R$ acts injectively on $C^*$(e.g. when it is flat over $A_R$). Also, assume $C^*$ carries a homotopy connection, i.e. a pre-connection that is also a chain map. Then \begin{equation}H^i(C^*/(t-1)C^*)\cong H^i(C^*)/(t-1)H^i(C^*)\end{equation}
\end{lem}
\begin{proof}
First, note $t-1$ acts injectively on any finitely generated $A_R$-module $N$ that carries a connection along $D_{A_R}$. To see this consider \begin{equation}N_0=\{x\in N:(t-1)^nx=0\text{ for some }n\geq 0 \}\subset N\end{equation}It is invariant under the connection $D_N$ on $N$. As it is still finitely generated, there exists $n_0\geq 0$ such that $(t-1)^{n_0}N_0=0$. Given $x\in N_0$
\begin{equation}
0=D_N( (t-1)^{n_0}x  )=n_0(t-1)^{n_0-1}tx+(t-1)^{n_0}D_N(x)=n_0(t-1)^{n_0-1}x
\end{equation}
as $(t-1)^{n_0-1}tx=(t-1)^{n_0}x+(t-1)^{n_0-1}x=(t-1)^{n_0-1}x$ and $D_N(x)\in N_0$. Thus, $(t-1)^{n_0-1}N_0$. By induction \begin{equation}(t-1)^{n_0-1}N_0=(t-1)^{n_0-2}N_0=\dots =N_0=0\end{equation}

Now consider the long exact sequence associated to short exact sequence \begin{equation}0\rightarrow C^*\xrightarrow{t-1}C^*\rightarrow C^*/(t-1)C^*\rightarrow 0\end{equation}
Putting $N=H^i(C^*)$, we see $H^i(C^*)\xrightarrow{t-1}H^i(C^*)$ is injective; hence, the induced map\begin{equation}H^i(C^*)/(t-1)H^i(C^*)\rightarrow H^i(C^*/(t-1)C^*) \end{equation}
is an isomorphism.
\end{proof}
Now let us prove a crucial result, again following \cite{flux}:
\begin{prop}\label{followsame}
Let $\fM$ and $\fM'$ be two families of $\B$-modules with pre-connections. Assume there exists a class $[\gamma]\in HH^1(\B,\B)$ such that both $\fM$ and $\fM'$ follow $[\gamma]$. Then, $\B_{A_R}(\fM,\fM')$ has Atiyah class $0$, and it admits a homotopy connection. Moreover, the homotopy connection can be chosen to be compatible with the composition (possibly up to homotopy).
\end{prop}
\begin{proof}Recall $\B_{A_R}(\fM,\fM')$ can be thought as the comodule homomorphisms and the map
$\precon_{\fM'}\circ(\cdot)-(\cdot)\circ \precon_{\fM}$ gives a pre-connection on this complex. Its differential is given by 
\begin{equation}\label{precatat}
d(\precon_{\fM'})\circ(\cdot)-(\cdot)\circ d(\precon_{\fM})=def(\precon_{\fM'})\circ(\cdot)-(\cdot)\circ def(\precon_{\fM})
\end{equation}
which is cohomologous to
\begin{equation}\label{precatgam}
\Gamma_{\fM'}(\gamma)\circ(\cdot)-(\cdot)\circ \Gamma_{\fM}(\gamma)
\end{equation}
However, as $\gamma$ is a closed class it induces a natural transformation and this cocycle is null-homotopic. Indeed, $\Gamma_\fM(\gamma)$ is the degree $0$ part of restriction of a Hochschild cocycle $\Gamma(\gamma)$ to $\fM$ where \begin{equation}\Gamma:CC^*(\B,\B)\rightarrow CC^*(\B^{mod},\B^{mod}) \end{equation} is a chain map, and components of $\Gamma(\gamma)$ can be non-vanishing only if the input length is $0$ or $1$. It is what Seidel denotes by $\gamma^{mod}$ in \cite{flux} (up to sign). As $\gamma$ is closed, $\Gamma(\gamma)$ too is closed. The vanishing of the differential implies \begin{equation}\Gamma_{\fM'}(\gamma)\circ(\cdot)-(\cdot)\circ \Gamma_{\fM}(\gamma) \end{equation} is equal to differential of degree $1$ part $\pm\Gamma(\gamma)^1$. Hence, the Atiyah class vanishes and by Lemma \ref{preconcomplexlemma}, the complex admits a homotopy connection. 

For the compatibility with the composition, first correct the pre-connections $\precon_\fM$ and $\precon_{\fM'}$ by $A_R$-linear cochains bounding $def(\precon_{\fM})-\Gamma_{\fM}(\gamma)$ and $def(\precon_{\fM'})-\Gamma_{\fM'}(\gamma)$ so that (\ref{precatat}) and (\ref{precatgam}) would actually be equal. Second, note the pre-connection $\precon_{\fM'}\circ(\cdot)-(\cdot)\circ\precon_{\fM}$ is automatically compatible with the composition. To make it into a homotopy connection as in Lemma \ref{preconcomplexlemma}, we have to correct it by a cochain bounding the Atiyah class and $\Gamma(\gamma)^1$ is a natural choice. The closedness of $\Gamma(\gamma)$ implies \begin{equation}\pm\Gamma(\gamma)^1(\cdot)\circ (\cdot)\pm(\cdot)\circ\Gamma(\gamma)^1(\cdot)\pm\Gamma(\gamma)^1(\cdot\circ\cdot)=\atop \pm\mu^2(\Gamma(\gamma)^1,\cdot)\pm\mu^2(\cdot,\Gamma(\gamma)^1)\pm\Gamma(\gamma)^1(\mu^2(\cdot,\cdot))=0\end{equation}
Hence, the pre-connection $\precon_{\fM'}\circ(\cdot)-(\cdot)\circ\precon_{\fM}$ corrected by $\Gamma(\gamma)^1$ is still compatible with composition $\mu^2$ and is a homotopy connection. 
\end{proof}
\begin{cor}The cohomology groups $H^i(\B_{A_R}(\fM,\fM'))$, considered as $A_R$-modules, admit connections along $D_{A_R}.$
\end{cor}
As we mentioned, the notion of following a class $[\gamma]$ measures infinitesimal change on a family. Now, we will give a recipe to compute the class which a family follows by using $\G$-actions. 
\begin{rk}
The heuristic is as follows: let $M$ be a manifold and $G$ be a Lie group acting smoothly on $M$. Then, to any $X\in Lie(G)$, one associates a vector field $X^\#$ on $M$ obtained as $X^\#_m=\frac{d(exp(tX).m )}{dt}|_{t=0}$. Smooth equivariant maps relate infinitesimal action on both sides.
\end{rk}
We start with some generalities. First, let us define a class of categories and equivariant modules on which one can make sense of the infinitesimal action. As we will follow a formal approach, it will not make a big difference to work over $\C$ or over $R$.
\begin{defn}\label{defn:prorat}
Let $\B_0$ be an $A_\infty$-category over $\C$ with a strict action of $\G(\C)$ (i.e. it acts on hom-sets and differentials and compositions are equivariant). Call this action pro-rational if one can choose a product decomposition for each $hom^i_{\B_0}(b,b')$ into countably many rational representations of $\G$ such that the decompositions together satisfy the following: if we restrict the differential or one factor of the composition into a finite subproduct it factors through a finite subproduct of the target. Similarly, call a strict action of $\G(R)$ on a curved $A_\infty$-category $\B$ over $R$ pro-rational if each $hom^i_{\B}(b,b')$ admits a product decomposition into $R$-free completed rational representations (i.e. $q$-adic completion of rational representations of $\G(R)$). Same condition for differential and compositions is imposed (and no extra condition on curvature is needed).
\end{defn}
\begin{exmp}
The guiding example is the following: consider the abelian category $Rat(\G)$ of rational representations of $\G=\mathbb G_{m,\C}$. Let $\B_0$ be the dg category of unbounded complexes over $Rat(\G)$. Then given such complexes $C^\cdot,D^\cdot$ \begin{equation}
hom_{\B_0}^i(C^\cdot,D^\cdot)=\prod_{n\in\Z} hom_{Rat(\G)}(C^n,D^{n+i})
\end{equation}
Clearly, this decomposition satisfies desired property. One can give analogous example for pseudo-complexes.
\end{exmp}
The following can be thought as a special case of this example:
\begin{lem}\label{lem:prorat}
$\Om(\Tt_0)_{dg}$ and $\Om(\Tt_R)_{cdg}$ are pro-rational.
\end{lem}
\begin{proof}
$\Om(\Tt_0)_{dg}$ is a subcategory of chain complexes of sheaves on $\Tt_0$. Each degree of these complexes is equipped with a $\G$-equivariant structure, and hom's of these complexes are given by products of hom's of these $\G$-equivariant sheaves. A closer examination shows each factor in these products is a rational representation, and the product decomposition is compatible with differentials and compositions. The proof is the same for $\Om(\Tt_R)_{cdg}$.
\end{proof}
\begin{defn}
A strictly equivariant module $\fM$ over a (uncurved/curved) category with strict pro-rational action of $\G(\C)$/$\G(R)$ is called pro-rational if each $\fM^i(b)$ admits a product decomposition such that the module maps satisfy a similar local finiteness as above. Similarly, a family of modules $\fM\in\B_{A_R}=\B_{A_R}^{mod}$ is called strictly pro-rational if it admits a strict pro-rational $\G(R)$-equivariant structure as a $\B\otimes A_R$-module. Here $z\in \G(R)$ acts on $A_R$ by $t\mapsto zt,u\mapsto z^{-1}u$. In other words it is pro-rational as a $\B$-module and each $\fM^i(b)$ is equivariant as an $A_R$-module.

Similar definitions make sense for bimodules and tri-modules and so on. 
\end{defn}
By Section \ref{sec:flowg}, we have the following (which is similar to Lemma \ref{lem:prorat}).
\begin{lem}$A_R$, resp. $A$-valued bimodules $\scrG_R^{pre}$, resp. $\scrG^{pre}$ are pro-rational.
\end{lem} 
\begin{rk}\label{rk:diffac}
Putting a pro-rational action on a category is essentially a special case of enriching the category in the complexes/pseudo-complexes in pro-completion of the category of rational (or completed rational) representations. For a rational representation $W$, we can formally differentiate the $\G$-action and obtain an operator $(z\partial_z)^\#$ associated to $z\partial_z\in Lie(\G)$. If $\G$-action on $v\in W$ is by $z\mapsto z^m$, the $(z\partial_z)^\#$ action is by $m$. The pro-completion process remembers the cofiltration of representations and hence we can formally define $(z\partial_z)^\#$-operator on the vector spaces underlying pro-objects. As the morphisms between pro-objects are compatible, they intertwine with $(z\partial_z)^\#$. Similar statements hold for ind-completion; hence, we can formally define this operator on direct sums of pro-objects of rational representations.
\end{rk}
\begin{rk}\label{rk:indprorat}
$M_\phi$ and $M_\phi^R$ are not pro-rational in the sense above, neither are $\scrG_R$ and $\scrG$ for their definition includes direct sums. However, the complexes involved are direct sums of pro-rational representations; hence, we can define infinitesimal action of $z\partial_z$ at each component separately, and we will use this. It is straightforward to define the notion of ``ind-pro rational'' as direct sums of pro-rational representations. More generally, such direct sums would be included (and infinitesimal action would be built-in), if we used sub-representations of products of rational representations that are invariant under $(z\partial_z)^\#$ (which is already defined on the product). Note, the morphisms of the latter category are assumed to be not only $\G$-equivariant, but also compatible with $(z\partial_z)^\#$. It is closed under products, direct sums and tensor products. 
\end{rk}
\begin{lem}\label{generalinfinitesimalalgebra}
Given a strictly pro-rational curved $A_\infty$ category $\B$ over $R$, \begin{equation}(z\partial_z)^\#:hom^i_{\B}(b,b')\rightarrow hom^i_{\B}(b,b')\end{equation} defines a Hochschild cocycle of degree $1$. Same holds for uncurved categories over $\C$.
\end{lem}
\begin{proof}
One only needs to check the closedness of the class, which can be proven using explicit computation. Alternatively, one can turn the $\G(R)$-action on $\B$ into an action on the bar construction $B(\B)$. The $A_\infty$ structure $\mu$ can be considered as a coderivation, which is clearly equivariant; hence the action is by dg coalgebra automorphisms. The bar construction is not a product of rational representations, but one can still differentiate the action to obtain a meaningful coderivation $(z\partial_z)^\#$.
Let $z$ act by $\rho_z\curvearrowright B(\B)$. Differentiate the relation $\rho_z\circ\mu=\mu\circ\rho_z$ to obtain $[\mu,(z\partial_z)^\#]=0$.
\end{proof}
\begin{lem}\label{followinfi}
Let $\fM$ be an $A_R$-semi-free strictly pro-rational family of $\B$-modules. Then, $\fM$ admits a natural pre-connection and its deformation class is the image of $[(z\partial_z)^\#]\in HH^1(\B,\B)$ under the natural map $HH^1(\B,\B)\rightarrow H^1(\B_{A_R}(\fM,\fM))$. In other words, $\fM$ follows $(z\partial_z)^\#$.
\end{lem}
\begin{proof}
$z\in \G(R)$ acts as an operator $\eta_z\in \fM(b,b')\rightarrow \fM(b,b')$ and it is possible to differentiate it by the remarks above. Moreover, $\eta_z(ax)=(z.a)\eta_z(x)$ for any $a\in A_R,x\in \fM(b,b')$ and differentiating this relation, we obtain \begin{equation}(z\partial_z)^\#_\fM(ax)=a(z\partial_z)^\#_\fM(x)+(z\partial_z)^\#_{A_R}(a)x \end{equation} where $(z\partial_z)^\#_\fM(x)=``(\frac{d\eta_z(x)}{dz}|_{z=1})"$ and $(z\partial_z)^\#_{A_R}(a)=``(\frac{d(z.a)}{dz}|_{z=1})"$ denote the corresponding infinitesimal actions. As it will be remarked in Appendix \ref{sec:modules}, \begin{equation}(z\partial_z)^\#_{A_R}(a)=D_{A_R}(a)\end{equation} hence, $(z\partial_z)^\#_\fM(x)$ is a pre-connection with no higher maps. We denote it by $\precon_\fM$.

Now similarly differentiate $\G(R)$-action on $\fM\otimes B(\B)$. We obtain a coderivation acting on $\fM\otimes B(\B)$. Let us project it to $\fM$ and extend as a comodule homomorphism. This way we obtain the pre-connection $\precon_{\fM}$ seen as a comodule endomorphism of $\fM\otimes B(\B)$. 

Now, its differential: first note the differential on dg comodule $\fM\otimes B(\B)$ can be written as the sum $\delta_\fM+\delta_\B\curvearrowright  \fM\otimes B(\B)$. Here, $\delta_\fM$ is the extension of the structure maps of the module as a comodule endomorphism. More explicitly, it is given by \begin{equation}\delta_\fM=\sum \mu_\fM^i\otimes 1_\B\otimes 1_\B\otimes \dots\otimes 1_\B\end{equation} $\delta_\B$ denotes the remaining terms, i.e. \begin{equation}\delta_\B=\sum 1_\fM\otimes1_\B\otimes\dots \otimes\mu^j_\B\otimes\dots\otimes1_\B \end{equation}
 
The extension of an operator $\precon_\fM:\fM\rightarrow\fM$ to $\fM\otimes B(\B)$ is via the formula $\precon_\fM\otimes 1_\B\otimes 1_\B\otimes \dots 1_\B$ on the each summand $\fM\otimes \B^{\otimes i}$. Let $\overline{\precon_\fM}$ denote this operator (only for the rest of this proof). The differential of it is given by the commutator \begin{equation}(\delta_\fM+\delta_\B)\circ \overline{\precon_\fM}-\overline{\precon_\fM}\circ(\delta_\fM+\delta_\B)=[\delta_\fM,\overline{\precon_\fM}]+[\delta_\B,\overline{\precon_\fM}] \end{equation}
By the formulas above, $[\delta_\B,\overline{\precon_\fM}]=0$. Hence, \begin{equation}def(\precon_\fM)=\delta_\fM\circ \overline{\precon_\fM}-\overline{\precon_\fM}\circ \delta_\fM\end{equation}which is a comodule homomorphism whose composition with the projection $\fM\otimes B(\B)\rightarrow\fM$ is given by \begin{equation} \mu^i_\fM(\precon_\fM\otimes 1_\B\dots\otimes 1_\B)-\precon_\fM\circ \mu^i_\fM  \end{equation}on the summand $\fM\otimes \B[1]^{\otimes i-1}$. 
For a fixed $i$, equivariance of $\mu^i_\fM$ and the fact that $\precon_\fM$ is obtained as the infinitesimal action of $z\partial_z\in Lie(\G)$ implies \begin{equation}\mu^i_\fM\circ(\precon_\fM\otimes 1_\B^{\otimes i-1})+\sum_j\mu^i_\fM\circ(1_\fM\otimes 1_\B^{\otimes j}\otimes(z\partial_z)^\#\otimes 1_\B^{\otimes i-j-2} )=\precon_\fM\circ\mu_\fM^i \end{equation}
Hence, the deformation class(projected to $\fM$) is given by \begin{equation}-\sum_j\mu^i_\fM\circ(1_\fM\otimes 1_\B^{\otimes j}\otimes(z\partial_z)^\#\otimes 1_\B^{\otimes i-j-2} )\end{equation}The Lemma follows from the definition of $\Gamma_\fM(\gamma)$ given by formula (\ref{eq:CCmaphom}).
\end{proof}
Most of the results/definitions above follow similarly for left modules and bimodules. Indeed, we will use the results for a semi-free, semi K-projective replacement of $\scrG_R$, which is a family of bimodules.
\subsection{The deformation class of $\scrG_R$}\label{sec:flowg}
To apply the results above, first we should clarify the $\G(R)$-action on $\scrG_R$. First, the graph $\cG_R\subset \Tt_R\times\Tt_R\times Spf(A_R)$ (resp. $\cG\subset \Tt_0\times\Tt_0\times Spec(A)$) is invariant under the action of $\G(R)$ ( resp. $\G(\C)$) which is trivial on the first factor, as in Remark \ref{actionremark} on the second factor, and by $z:t\mapsto zt,u\mapsto z^{-1}u$ on the third. This is clear from the defining equations (\ref{eq:graph1}) and (\ref{eq:graph2}). 

Hence, there is an action of $\G(R)$ (resp. $\G(\C)$) on $\scrG_R^{pre}(\scrF,\scrF')$ (resp. $\scrG^{pre}(\scrF,\scrF')$) compatible with the action of the same group on $\Om(\Tt_R)_{cdg}\otimes \Om(\Tt_R)_{cdg}^{op}\otimes A_R$ (resp. $\Om(\Tt_0)_{dg}\otimes \Om(\Tt_0)_{dg}^{op}\otimes A$) that is trivial on the first factor, as in Remark \ref{actionrkondefoOm} (resp. Section \ref{subsec:action}) in the second factor, and by $z:t\mapsto zt,u\mapsto z^{-1}u$ on the third. A similar action exists for $A_R$ (resp. $A$)-valued $M_\phi^R$ (resp. $M_\phi$)-bimodule $\scrG_R$ (resp. $\scrG$). The strict pro-rationality is obvious in the former case. In the latter, we do not have pro-rationality. However, by Remark \ref{rk:indprorat}, it is still sensible to formally differentiate the action and the results above are valid since we have direct sums of pro-rational representations. 

However, as mentioned the bimodules above are not semi-free over $A_R$ (resp. $A$) and they do not satisfy K-projectivity condition. This is easy to resolve by passing to equivariant semi-free replacements that are K-projective over $A$ at $q=0$. More generally, consider a curved $A_\infty$-category $\B$ with strict and pro-rational action of $\G(R)$. Let $\fM$ be a bimodule over $\B$ with a pro-rational action. We may weaken pro-rationality to ``differentiability" of the action (see Remarks \ref{rk:diffac} and \ref{rk:indprorat}). Then the bar construction gives us a bimodule. To its objects, it assigns a complex which in each degree is given by a product of expressions such as  \begin{equation} \B^{i_k}(b_k,\cdot)\otimes\B^{i_{k-1}}(b_{k-1},b_k)\otimes\dots\otimes\fM^{i_l}(b_{l-1},b_l)\otimes\dots\otimes \B^{i_1}(b_1,b_2)\otimes\B^{i_0}(\cdot,b_1) \end{equation} 
Its differential and structure maps can also be given by explicit expressions involving the structure maps of $\fM$ and $A_\infty$-products of $\B$ (alternating sums of expressions such as $1_\B\otimes 1_\B\otimes \dots \otimes \mu_\fM^{p|1|q}\otimes\dots \otimes 1_\B$) and it also defines a strictly equivariant $\B$-bimodule, where the infinitesimal action can be formally defined again (i.e. by differentiating each component of the tensor product and applying the Leibniz rule, see also Remark \ref{rk:indprorat}). We gave it for bimodules just as an illustration, and this can be done for left modules, right modules, trimodules, and so on. In particular, we can use this construction to replace $\scrG_R^{pre}$ and $\scrG_R$ with $A_R$-semi-free, semi-K-projective families of bimodules over $\Om(\Tt_R)_{cdg}$ and $M_\phi^R$ in a canonical way. It is compatible with $\tr$ in the former case. Let us denote these replacements by $\scrG_R^{pre,sf}$ and $\scrG_R^{sf}$.

To find Hochschild cohomology classes that are followed by $\scrG_R^{pre,sf}$ and $\scrG_R^{sf}$, we can apply Lemma \ref{followinfi}. The following proposition relates the infinitesimal action cocycle $(z\partial_z)^\#$ to previously defined Hochschild classes.
\begin{prop}\label{postgammaomprop}
The infinitesimal action cocycle $(z\partial_z)^\#$ defined in Lemma \ref{generalinfinitesimalalgebra} of $\G(R)$-action on $\Om(\Tt_R)_{cdg}$ has the class $\gamma_\Om^R$.
\end{prop}
\begin{proof}
This follows from local to global techniques of Section \ref{sec:hoch}. For instance, consider the isomorphism (\ref{eq:HHcolimitR}). It is based on maps \begin{equation} 
\xymatrix{CC^*(\Om(\Tt_R)_{cdg},\Om(\Tt_R)_{cdg})\ar[r]&CC^*(\Om(\Tt_R)_{cdg},\scrC oh(U)_R)\\ &CC^*(\scrC oh(U)_R,\scrC oh(U)_R)\ar[u]_\simeq }\end{equation}where $U=U_{i+1/2}$ or $V_i$ and $\scrC oh (U)_R$ is a curved deformation of $\scrC oh (U)$. We can replace $\scrC oh (U)$ by the image of the restriction functor from $\Om(\Tt_R)_{cdg}$ and use a strictly $\G(R)$-equivariant model such that this functor would be strictly equivariant too.
Moreover, it is easy to see that the images of cocycles $(z\partial_z)^\#$ in lower right and upper left complexes correspond in cohomology. Similar statements hold for the restriction maps \begin{equation}CC^*(\scrC oh(U_{i+1/2}) )\rightarrow CC^*(\C[X_i,Y_{i+1}][[q]]/(X_iY_{i+1}-q) ) \end{equation} and so on.
Hence, it is enough to prove the infinitesimal action cocycle $(z\partial_z)^\#$ is the same as local building blocks of $\gamma_\Om^R$. In other words, we wish to show $Y_{i+1}Y_{i+1}^*-X_iX_i^*$ corresponds to cocycle $(z\partial_z)^\#$ for the action \begin{equation}z:Y_{i+1}\mapsto zY_{i+1},X_i\mapsto z^{-1}X_i \end{equation}
under the Hochschild-Kostant-Rosenberg isomorphism of \cite[Appendix, Theorem 2]{quantization}. Examining this isomorphism, we can see $Y_{i+1}Y_{i+1}^*-X_iX_i^*$ corresponds to derivation $Y_{i+1}\partial_{Y_{i+1}}-X_i\partial_{X_i}$; i.e., to $(z\partial_z)^\#$ for the given action. 
\end{proof}
\begin{cor}\label{postgammacor}Consider the $\G(R)$-action on $M_\phi^R$. The infinitesimal action cocycle $(z\partial_z)^\#$ has the same class as $\gamma_\phi^R$.
\end{cor}
\begin{proof}
This follows from a similar examination of the isomorphism in Prop \ref{CCtorus}.
\end{proof}
\begin{cor}\label{gfollows}
The family $\scrG_R^{pre,sf}$ of $\Om(\Tt_R)_{cdg}$-bimodules follows $1\otimes \gamma^\Om_R\in HH^1(\Om(\Tt_R)^e_{cdg})$. Similarly, the family of $M_\phi^R$-bimodules $\scrG_R^{sf}$ follows $1\otimes \gamma^\phi_R\in HH^1((M_\phi^R)^e)$.
\end{cor}
\begin{proof}
This follows from the remarks at Section \ref{sec:flowg} about the $\G(R)$-action on $\scrG_R^{pre},\scrG_R$ as well as on families $\scrG_R^{pre,sf},\scrG_R^{sf}$ by applying Lemma \ref{followinfi}.
\end{proof}
\begin{rk}\label{grk}
Results similar to Corollary \ref{gfollows} hold for $\scrG_R^{pre,sf}|_{q=0}$, which is a semi-free replacement of $\scrG^{pre}$ and for $\scrG_R^{sf}|_{q=0}$, which is a semi-free replacement of $\scrG$.
\end{rk}

\section{A rank $2$ lattice inside $HH^1(M_\phi^R)$}\label{sec:rank2}
In this section, we will find a subgroup of $HH^1(M_\phi^R)$ that contains $\gamma_\phi^R$, that is isomorphic to $\Z^2$ and that is preserved under Morita equivalences. The basic idea is as follows:

Given an $A_\infty$-category $\B$, one can define the derived Picard group as the functor from commutative rings to groups sending \begin{equation}DPic:T \mapsto\{T\text{-semifree, invertible }\B^e\otimes T\text{-modules} \}/\text{quasi-isomorphism} \end{equation}Here we call a $\B^e\otimes T$-module $\fM$ invertible if there exists another $\B^e\otimes T$-module $\mathfrak{N}$ such that \begin{equation}\fM\otimes_\B\mathfrak{N}\simeq\mathfrak{N}\otimes_\B\fM\simeq \B\otimes T \end{equation}In other words it is a ``family'' of invertible $\B$-bimodules parametrized by $T$. The group structure is given by $(\fM,\fM')\mapsto \fM\otimes_\B\fM'$ See \cite{dpic} for a infinitesimal and derived version of it. In \cite{dpic}, the author also argues to show that the Lie algebra of this group is isomorphic to $HH^1(\B)$, with the Gerstenhaber bracket as the Lie bracket. This group functor is obviously Morita invariant. Hence, it is natural to look at its coroots, i.e. maps $\G\rightarrow DPic$ and the induced image of $(z\partial_z)\in Lie(\G)$. This subset will be a lattice in our case. However, we will not formally refer to this group functor again. Instead, we will simply use of group like families of bimodules, whose definition is close to definition above.

In the case of mapping tori, notice another $\G(\C)$, resp. $\G(R)$-action on $M_\phi$, resp. $M_\phi^R$, this time rational, resp. completed rational (which we will informally refer as another $\G$/$\widehat{\G}$-action). By definition $\B\#\Z$ is automatically equipped with an extra $\Z$-grading. Recall the morphisms from $b_1$ to $b_2$ are \begin{equation}\bigoplus_{g\in \Z}hom_{\B}(g(b_1),b_2)\end{equation} and we declare $hom_{\B}(g(b_1),b_2)$ to be the degree $g$ part in this extra grading. In particular, $M_\phi$ and $M_\phi^R$ carry this extra grading and we let $z\in\G$ act by $z^g$ on degree $g$ part. 
\begin{rk}
When $\A=\C$ and $\phi=1_\C$, this new action corresponds to twist by line bundles in $Pic^0(\T_R/R)$.
\end{rk}

It is easy to see that the new and old $\G$-actions (resp. $\widehat{\G}$-actions) commute strictly. Hence, we have an action of $\G\times\G$ (resp. $\widehat{\G}\times_{Spf{R}}\widehat{\G}$) on $M_\phi$ (resp. $M_\phi^R$). We want to organize them in group-like families of bimodules. First, let us give meaning to this more general notion of families, mimicking \cite{flux} and \cite{dpic}. We will work over $R$, as everything is similar over $\C$.
\begin{defn}Let $\B$ be a curved $A_\infty$-category over $R$. Let $T$ be a topologically finitely generated complete commutative $R$-algebra (examples are given by formal completions of affine subschemes of affine varieties). A family of $\B$-modules/bimodules parametrized by $T$ is an $A_\infty$-module/bimodule $\fM$ over $\B$ such that each $\fM(b)$ resp. $\fM(b,b')$ is a free $T$-module (in each degree), $\fM(b)/q\fM(b)$ resp. $\fM(b,b')/q\fM(b,b')$ is K-projective over $T/qT$ and the module maps are $T$-linear. The morphisms between such families can be defined in a analogous way to Definition \ref{arfamilycat} and Remark \ref{arfamilycatrk}. We denote the category of families (with $T$-linear morphisms) by $\B_T^{mod}$ or simply by $\B_T$.
\end{defn}
Contrary to previous section, we will assume $T$ is the ring of functions on a formal affine scheme that is smooth over $R$. We can define pre-connections in a similar way:
\begin{defn}
A pre-connection $\quabla$ on the family $\fM$ of right $\B$-modules parametrized by $T$ is a $B(\B)$-comodule map \begin{equation}\fM\otimes B(\B)\rightarrow  \Omega^1_{T/R}\otimes\fM\otimes B(\B) \end{equation} that satisfies the Leibniz rule with respect to $T$.
\end{defn}
This can again be seen as a collection of maps 
\begin{flalign*}
\quabla^1:\fM(b_0)\rightarrow \Omega^1_{T/R}\otimes\fM(b_0)\\ \quabla^2:\fM(b_1)\otimes \B(b_0,b_1)\rightarrow \Omega^1_{T/R}\otimes\fM(b_0)\\  \quabla^3:\fM(b_2)\otimes \B(b_1,b_2)\otimes \B(b_0,b_1)\rightarrow \Omega^1_{T/R}\otimes\fM(b_0)\\ \dots
\end{flalign*} such that $\quabla^1$ satisfies the Leibniz rule and $\quabla^i$ are $T$-linear for $i>1$. See \cite{flux}. The deformation class \begin{equation}def(\quabla)\in (\B_T^{mod})^1(\fM,\Omega^1_{T/R}\otimes \fM)\end{equation} has a similar definition, i.e. as the differential of $\quabla$. The class of $def(\quabla)$ is independent of the choice of pre-connection and will again be denoted by $Def(\fM)$. Everything works for bimodules in a similar way.

Now we define group-like families. The moral of the definition is simple: For a (formal) affine group $G$ the transformations ``$G\rightarrow DPic(\B)$'' are the families parametrized by $G$ and group-like families are those corresponding to group homomorphisms ``$G\rightarrow DPic(\B)$''. More explicitly, let $T=\Om(G_0)[[q]]=\Om(G)$, where $G_0$ is an affine algebraic group over $\C$ and $G=G_0\times_\C Spf(R)$. $T$ is a Hopf algebra in the category of complete $R$-modules. Let \begin{equation}m,f_1,f_2:G\times G\rightarrow G \end{equation} denote multiplication, first and second projection respectively. 
\begin{defn}
A group-like family of invertible bimodules parametrized by $T$ is a family $\fM$ such that $f_1^*\fM\otimes_\B f_2^*\fM\simeq m^*\fM $.
\end{defn}
This condition can be phrased in terms of the Hopf algebra structure of $T$. In particular, it is easy to see that a (strict) $G$-action on $\B$ gives a group-like family.
\begin{exmp}\label{2dfamily}
We mentioned an action of $\widehat{\G\times\G}:=\widehat{\G}\times_{Spf(R)}\widehat{\G}$ on $M_\phi^R$. This was a pointwise action; however, it is easy to see it as a group-like family parametrized by $\C[z_1^\pm,z_2^\pm][[q]]$. Denote it by $\rho_{uni}$. Similarly, there is group-like family of bimodules over $M_\phi$ parametrized by $\G\times\G$.
\end{exmp}
\begin{lem}\label{uniinjective}
The restriction of the family parametrized by $\widehat{\G\times\G}$ to two different $R$-points are different. In other words, ``$\widehat{\G\times\G}(R)\rightarrow DPic_{M_\phi^R}(R)$'' is injective. 
\end{lem}
\begin{proof}
Assume the family is trivial over an element of $\widehat{\G\times\G}(R)=R^*\times R^*$. This implies there exists $z_1,z_2\in R^*$ such that the bimodule corresponding to $z_1$ for the action in Remark \ref{actiononmt1} is quasi-isomorphic to bimodule corresponding to $z_2$ for the action coming from extra grading. 

First, let us show $z_1=1$. Pick a smooth $R$-point $p\in\Tt_R$ such that the restriction to $q=0$ is on $C_0\subset\Tt_0$. We can represent $\Om_p$ as a deformation of a mapping cone $Cone(\Om_{C_0}(-1)\rightarrow \Om_{C_0})$; hence, by a twisted complex of this form, which we also denote by $\Om_p$ (it is an unobstructed object over $\Om(\Tt_R)_{cdg}$, the differential of this twisted complex may include other terms that are $O(q)$). Hence, we obtain a subcategory $\Om_p\otimes\A=\{\Om_p\otimes a:a\in ob(\A) \}\subset tw^\pi(M_\phi^R)$. The $z_1$-action moves this subcategory to $\Om_{z_1.p}\otimes \A$. Moreover, the morphism $\Om_{C_0}(-1)\rightarrow \Om_{C_0}$ is of degree $0$, in the extra grading; hence, $z_2$ fixes $\Om_{p}\otimes \A$. It is not hard to prove $\Om_{p}\otimes \A$ is orthogonal to $\Om_{z_1.p}\otimes \A$ unless $z_1.p=p$. Thus, $z_1=1$.

Hence, by assumption the bimodule $_{z_2}(M_\phi^R)_1$ corresponding to $z_2$-action is quasi-isomorphic to diagonal bimodule. The bimodule $_{z_2}(M_\phi^R)_1$ has the same underlying pseudo-complexes but the $M_\phi^R$ action is twisted by $z_2$ on the right (i.e. $f.x.f'=z_2(f)xf'$). This bimodule can be obtained from $\Om(\Tt_R)_{cdg}\otimes \A$ as in Section \ref{sec:bimodgen} and the $z_2$ twist amounts to changing the $\Z_\Delta $-equivariant structure, where the new $\Z_\Delta $-equivariant structure is given by $g.m=z_2^{g}g(m)$. The complex $CC^*(M_\phi^R,\,_{z_2}(M_\phi^R)_1)$ can be computed to be the derived invariants of \begin{equation}CC^*(\Om(\Tt_R)_{cdg},\Om(\Tt_R)_{cdg})\otimes CC^*(\A,\A) \end{equation} as in Proposition \ref{CCtorus}, but again the $\Z$-action on $CC^*(\Om(\Tt_R)_{cdg},\Om(\Tt_R)_{cdg})$ is different from the one in Proposition \ref{CCtorus} by $z_2$ (i.e. it acts by $z_2(tr_*)$ in the first component).

This complex has no negative degree cohomology and its cohomology in degree $0$ is isomorphic to $R$, fixed by the previous action. Hence, the derived invariants of it in degree $0$ is $0$ with respect to new action, unless $z_2=1$. Therefore, $HH^0(M_\phi^R,\,_{z_2}(M_\phi^R)_1)=0$ unless $z_2=1$; thus, $z_2=1$.
\end{proof}
\begin{rk}\label{defact}
For a group-like family $\fM$, the deformation class \begin{equation}Def(\fM)=[def(\quabla)]\in  H^1(\B_T^{e,mod})(\fM,\Omega^1_{T/R}\otimes \fM)\simeq HH^*(\B,\Omega^1_{T/R}\otimes \B) \end{equation} induces a linear map \begin{equation}\mathfrak g\rightarrow HH^1(\B,\B) \end{equation}where $\mathfrak g=Lie(G/R)=Lie(G_0)[[q]]$.
\end{rk}
\begin{lem}\label{lieinjective}
The map $R^2=Lie(\widehat{\G\times\G})\rightarrow HH^1(M_\phi^R,M_\phi^R)$ induced as in Remark \ref{defact} is an isomorphism. 	
\end{lem}
\begin{proof}
We know both sides are isomorphic to $R^2$. It is enough to show the restriction $\C^2\rightarrow HH^1(M_\phi,M_\phi)$ is an isomorphism. This map is given by the deformation class of a group-like family parametrized by $\G\times\G$. The restriction of the family to $\G\times\{1 \}$ corresponds to the action in Remark \ref{actiononmt1}. This restricted family carries a $\G$-equivariant structure; where $\G$ action on $M_\phi\otimes M_\phi^{op}\otimes \Om(\G)$ is trivial on the first factor, by $z:t\mapsto zt$ on $\Om(\G)$ and as in Remark \ref{actiononmt1} on the second. By a version of Lemma \ref{followinfi}, we can show the restricted family follows $1\otimes\gamma_\phi\in HH^1(M_\phi^e,M_\phi^e)$. Its restriction to diagonal bimodule gives $\pm\gamma_\phi\in HH^1(M_\phi,M_\phi)$; hence, the image of $(z\partial_z,0)\in Lie(\G\times\G)$ is $\pm\gamma_\phi$. The sign depends on the identification of $RHom^*_{M_\phi^e}(M_\phi^e,M_\phi^e)$ with $HH^*(M_\phi,M_\phi)$.

Similarly, consider the restriction to $\{1 \}\otimes \G$. The infinitesimal action $\gamma_2:=(0,z\partial_z)^\#$ is a 1-cocycle on $M_\phi^R$ that acts by $n\in \Z$ on degree $n$ morphisms with respect to extra grading($\gamma_2^1(f)=nf$ for $ |f|=n$, $\gamma_2^i=0$ for $i\neq 1$) and as before the family restricted to $\{1\}\times\G$ follows $\pm\gamma_2$ (it follows $1\otimes \gamma_2$ to be precise, and the restriction to diagonal gives $\pm\gamma_2$). 

We want to show $\gamma_\phi$ and $\gamma_2$ are independent. By the proof of Proposition \ref{CCtorus}, we have a long exact sequence \begin{equation}\dots\rightarrow HH^0(\Om(\Tt_0)_{dg}\otimes \A)\xrightarrow{\tr_*\otimes\phi_*-1=0}HH^0(\Om(\Tt_0)_{dg}\otimes \A)\rightarrow \atop HH^1(M_\phi)\rightarrow HH^1(\Om(\Tt_0)_{dg}\otimes\A)\rightarrow\dots \end{equation}
$\gamma_\phi\in HH^1(M_\phi)$ maps to $\gamma_\Om\otimes 1$ and the image is non-zero. On the other hand, the map on the lower line is the equivalent to \begin{equation}HH^1((\Om(\Tt_0)_{dg}\otimes\A)\#\Z,(\Om(\Tt_0)_{dg}\otimes\A)\#\Z )\rightarrow HH^1(\Om(\Tt_0)_{dg}\otimes\A,(\Om(\Tt_0)_{dg}\otimes\A)\#\Z ) \end{equation}and the restriction of $\gamma_2$ to $\Om(\Tt_0)_{dg}\otimes \A$ is zero since the functor $\Om(\Tt_0)_{dg}\otimes \A\rightarrow (\Om(\Tt_0)_{dg}\otimes \A)\#\Z$ maps to degree $0$ part (with respect to extra grading). Thus, $\gamma_\phi$ and $\gamma_2$ are independent and this finishes the proof.
\end{proof}
\begin{lem}\label{equalizer}
Let $\B$ and $G=G_0\times_{\C}Spf(R)$ be as above and let $\rho$ and $\rho'$ be two group-like families of bimodules parametrized by $G$. Further assume $\B|_{q=0}$ is smooth and proper in each degree and $HH^0(\B)=R$. Assume $\rho$ and $\rho'$ can be represented as objects of $tw^\pi(\B^e\otimes \Om(G))$, where the tensor product is over $R$ and completed as usual (more precisely, $\rho$ and $\rho'$ are direct summands of families of twisted complexes in the sense of \cite{flux}). Then there is a natural formal subgroup scheme $S\subset G$ whose $R$-points are given by $\{x\in G(R):\rho_x\simeq \rho'_x \}$. Moreover, $S$ is closed.
\end{lem}
\begin{proof}
By choosing a minimal model for $\B|_{q=0}$ and considering the corresponding deformation, we may assume $\B(b,b')$ is a bounded below complex of finite rank free $R$-modules. Hence, choosing representatives for $\rho$ and $\rho'$ by twisted complexes, we can assume \begin{equation}hom^\cdot_{(\B^e)^{mod}_{\Om(G)}}(\rho,\rho')=:hom^\cdot(\rho,\rho')\end{equation}(and other hom-complexes among $\rho$ and $\rho'$) is a bounded below complex of finitely generated $\Om(G)$-modules.
Here,  $(\B^e)^{mod}_{\Om(G)}$ is the category of families of $\B$-bimodules parametrized by $\Om(G)$. 

Given a bounded below complex $C^*$ of finitely generated, free $\Om(G)$-modules, we can define ``a closed locus of points such that $rkH^0(C^*_x )\geq m$" as follows: consider \begin{equation}\dots\rightarrow C^{-1}\xrightarrow{d^{-1}}C^0\xrightarrow{d^0}C^1\rightarrow\dots \end{equation}Let $C^0\cong \Om(G)^r$. We wish to define the locus of points $x$ where \begin{equation}rk(d^{-1}_x)+rk(d^0_x)\leq r-m\end{equation} where $rk$ is the matrix rank. This is the same as \begin{equation}\bigcup_{a+b=r-m+2} \{rk(d^{-1}_x)<a\text{ and }rk(d^{0}_x)<b \}\end{equation} and hence can be defined by the ideal \begin{equation}\label{ideal}\bigcap_{a+b=r-m+2} (I_a(d^{-1})+I_b(d^0))\end{equation} where $I_k(d^i)$ denotes the ideal generated by $k\times k$ minors of $d^i$, where $d^i$ is considered as an $\Om(G)$-valued matrix
(or alternatively one can realize it as \begin{equation}\bigcap_{\alpha+\beta=r-m+1 }\{rk(d^{-1}_x)<\alpha \text{ or } rk(d^{0}_x)<\beta \}\end{equation} as in \cite{greenlazarsfeld}).
Borrowing the terminology of \cite{greenlazarsfeld}, denote the (formal) subscheme defined by the ideal (\ref{ideal}) by $S_m(C^*)$. Let \begin{equation}S_1=S_1(hom^\cdot(\rho,\rho) )\cap S_1(hom^\cdot(\rho',\rho) )\cap S_1(hom^\cdot(\rho,\rho') )\cap S_1(hom^\cdot(\rho',\rho') )\setminus\atop \big(S_2(hom^\cdot(\rho,\rho) )\cup S_2(hom^\cdot(\rho',\rho) )\cup S_2(hom^\cdot(\rho,\rho') )\cup S_2(hom^\cdot(\rho',\rho') ) \big)\end{equation}  which deforms the locus of points $x\in G_0=G|_{q=0}$ where $H^0(hom^\cdot(\rho_x,\rho'_x) )$ etc. are of rank $1$. Note as we are working with formal schemes, this is not immediately a defining condition. Note also we do not assert $R$-flatness of $S_1$. $S_1$ is an (formal) open subscheme of the closed intersection of $S_1(hom^\cdot(\rho,\rho'))$ and so on. 

Now, let us define $S$ as a formal open subscheme of $S_1$. Roughly, we want to define a subscheme whose $K$-points satisfy the property that the composition \begin{equation}H^0(hom^\cdot(\rho'_x,\rho_x) )\otimes_{\Om(G)} H^0(hom^\cdot(\rho_x,\rho'_x) )\rightarrow H^0(hom^\cdot(\rho_x,\rho_x) ) \end{equation} is surjective, where $R/q\subset K$. Let $x\in S_1(K)$ be a $K$-point, where $R/q=\C\rightarrow K$ is a field extension. Let $v\in hom^0(\rho_x,\rho'_x)$ be a closed element generating $H^0(hom^\cdot(\rho_x,\rho'_x))=K$. We can extend $v$ to $\tilde v$ defined on a neighborhood of $x\in S_1$ such that $d^0(\tilde v)=0$. The existence of such a $\tilde v$ follows from Lemma \ref{balik}.
Pick $\tilde v'$, a closed section of $hom^0(\rho',\rho)$ satisfying the same property. The composition $\tilde v'\circ\tilde v$ generates the cohomology $H^0(hom^\cdot(\rho_x,\rho_x) )$ and by the proof of Lemma \ref{balik} it is invertible in a neighborhood of $x$ in $S_1$. Same holds for the composition $\tilde v\circ \tilde v'$ as well. Hence, there exists a neighborhood $\Omega $ of $x$ in $S_1$ such that $\rho|_{\Omega}\simeq \rho'|_\Omega$. Define $S$ to be the union of open subsets $\Omega\subset S_1$ such that $\rho|_{\Omega}\simeq \rho'|_\Omega$.

We want to characterize $S$ functorially. A formal scheme $\mathfrak X$ over $R=\C[[q]]$ be seen as an ind-scheme, presented as the colimit of $\mathfrak X\times_R \C[[q]]/(q^{n+1})$. See \cite[\href{http://stacks.math.columbia.edu/tag/0AIT}{Tag 0AIT}]{stacks-project}, \cite{33079} or \cite{236351}. More explicitly, we can realize it as a functor over $Alg_\C$, algebras over $\C$ or even better over $Alg_{R,f}$ algebras over $R$ such that $q$ maps to a nilpotent element (such as $R/(q^m)$). Given $T\in Alg_{R,f}$, the $R$-linear maps $Spec(T)\rightarrow S\subset G$ are maps $Spec(T)\rightarrow G$ such that $\rho_T$ and $\rho'_T$ are locally isomorphic over $T$. Indeed, if we have such a map  $f:Spec(T)\rightarrow S$ and a point $x\in Spec(T)$ we can choose a neighborhood $f(x)\in\Omega \subset S$ such that $\rho_\Omega\simeq\rho'_\Omega$ and thus $\rho_{f^{-1}(\Omega)}\simeq \rho'_{f^{-1}(\Omega)}$; hence they are locally isomorphic. On the other hand, if we assume $\rho_T$ and $\rho_T'$  are locally isomorphic, we can easily check the rank conditions so that $f$ factors through $S_1\subset G$. Local surjectivity is also easy to check; thus, it is actually a map into $S$.

From this functorial description, it is easy to see that $S$ is a subgroup functor of $G$ over $R$ (hence, it has a flat unit section over $R$). Clearly, it is locally closed. But locally closed subgroup schemes (such as $S|_{q=0}$) are actually closed. Hence, we are done.
\end{proof}
\begin{lem}\label{balik}
Let $U,V,W$ be finite dimensional vector spaces over $\C$ and \begin{equation}U\xrightarrow{ A(q,s)} V\xrightarrow{ B(q,s)} W\end{equation}	be a family of matrices parametrized by a formal scheme $\mathfrak X'$ over $R$ (we do not assume $R$-flatness). Assume $B\circ A=0$ identically and restrict to locus \begin{equation}``\mathfrak X:=\bigcup_{a+b=r+1}\{rkA< a,rk B<b \}\setminus\bigcup_{a+b=r}\{rkA< a,rk B<b \} "\end{equation}where $r=dim V$. Then given a point $x=(q=0,s=s_0)$ of $X=\mathfrak X|_{q=0}$, there exists a neighborhood $\Omega$ of $x$ inside $\mathfrak X$ ($X$ and $\mathfrak X$ has the same underlying topological space) and a section $v(q,s)$ defined over $\Omega$ (i.e. a family of vectors in $V$ parametrized by $\Omega$) which restricts to a generator of $ker (B(0,s))/Im (A(0,s))$ for all $(0,s)\in \Omega$.
\end{lem}
\begin{proof}
The locus $\mathfrak X$ is essentially the locus of points at which the cohomology $kerB/ImA$ is exactly of rank $1$. Let $x=(0,s_0)$ satisfy $rkA(x)<a,rkB(x)<b$ for $a+b=r+1$. When we perturb $x$, $nullity(B)$ may only decrease and $rk(A)$ may only increase, but if this happens the rank of cohomology decreases too. Thus, $rk A$ and $rk B$ are constant in a neighborhood of $x$ inside $\mathfrak X$. 
In other words, it has a neighborhood $\Omega\subset\mathfrak X$ on which $a\times a$ minors of $A(q,s)$ and $b\times b$ minors of $B(q,s)$ all vanish and there exists an $b-1\times b-1$ minor of $B$ which is invertible on $\Omega$. By row and column operations we can assume this minor is the upper-left principal minor and the upper-left $b-1\times b-1$ square submatrix of $B$ is the identity matrix. By more row and column operations
we can assume the rest of the entries of $B$ are $0$. Hence, $ker B$ has the simple description as the column vectors with vanishing first $a-1$-entries. This implies there exists $v(q,s),(q,s)\in\Omega$ such that $B(q,s)v(q,s)=0$ and $v(0,s_0)\not\in Im(A(0,s))$. $Im(A(0,s_0))$ is generated by columns of $A(0,s_0)$ and the condition $v(0,s_0)\not\in Im(A(0,s_0))$ can be phrased as the columns of $[A(0,s_0),v(0,s_0)]$ generate the subspace of vectors with vanishing first $a-1$ entries. Hence, it is an open condition and by further shrinking $\Omega $ we can ensure $v(0,s)$ generates the cohomology at $(0,s)\in\Omega$.
\end{proof}
\begin{rk}
Note the statement of the Lemma \ref{equalizer} does not immediately imply flatness of $S$ over $R$; however, we believe this to be true.	
\end{rk}
\begin{rk}
There is a possibility that it is unnecessary to assume the representability by objects $tw^\pi(\B^e\otimes\Om(G))$ as it may be a corollary of smoothness of $\B|_{q=0}$.
\end{rk}
\begin{prop}\label{liftingcoroot}
Let $\rho$ be a group-like family of invertible bimodules over $M_\phi^R$ parametrized by $\widehat{\G}$. There exists a homomorphism $\eta:\widehat{\G}\rightarrow \widehat{\G\times\G}$ of formal group schemes over $R$ such that $\rho$ is the pull-back of the family $\rho_{uni}$ (see Example \ref{2dfamily}) under $\eta$.
\end{prop}
\begin{proof}
Let $G=\widehat{\G\times \G}$ and recall $\rho_{uni}$ denote the family in Example \ref{2dfamily}. Pulling back $\rho$ resp. $\rho_{uni}$ under projections $G\times\widehat{\G}\rightarrow\widehat{\G}$ resp. $G\times\widehat{\G}\rightarrow G$ we obtain two group-like families on $G\times\widehat{\G}$, which we denote by $\rho'$ resp. $\rho''$. Apply Lemma \ref{equalizer} to $\rho'$ and $\rho''$ to obtain a formal subgroup scheme $S$ of $G\times\widehat{\G}$, ``the locus of points such that $\rho_y'\simeq\rho_y''$''. Hence, \begin{equation}S(R)=\{(x,x')\in G(R)\times\widehat{\G}(R):\rho_{uni,x}\simeq \rho_{x'} \} \end{equation}
By Lemma \ref{uniinjective}, the map $S(R)\rightarrow\widehat{\G}(R)$ is injective. It is easy to prove a version of Lemma \ref{uniinjective} for the special fiber $q=0$ by using the same idea; thus, $S|_{q=0}(\C)=S(\C)\rightarrow\G(\C)$ is injective as well.

By the functorial description of $S$, the Lie algebra (i.e. $R[\epsilon]/(\epsilon^2)$-points that specialize to identity at $\epsilon=0$) of $S$ has a description as $Lie(G)\times_{HH^1(M^R_\phi)}Lie(\widehat{\G})$ and by Lemma \ref{lieinjective} the map $Lie(S)\rightarrow Lie(\widehat{\G})$ is an isomorphism. Similarly, $Lie(S|_{q=0})\rightarrow Lie(\widehat{\G}|_{q=0})$ is an isomorphism. 

Combined with the injectivity statement above, this shows $S|_{q=0}\rightarrow\G$ is an isomorphism. Being a formal closed subscheme of a formal affine scheme, $S=Spf(B)$ for some quotient $B$ of $\Om(G\times \widehat{\G})$ and the map $S\rightarrow \widehat{\G}$ corresponds to an algebra map $\C[z^{\pm}][[q]]\rightarrow B$ inducing an isomorphism $\C[z^{\pm}]\rightarrow B/qB$. Thus, the map $\C[z^{\pm}][[q]]\rightarrow B$ is surjective and $S$ can be seen as a formal affine subscheme of $\widehat{\G}$. One can prove surjectivity of $\C[z^{\pm}][[q]]\rightarrow B$ by lifting an element $b\in B$ step by step. This uses the fact that $\bigcap_{n}q^nB=0$, which follows from $q$-adic completeness of $B$.

Thus, let $B=\C[z^\pm][[q]]/I$. The identity morphism $Spf(R)\rightarrow \widehat{\G}$ factors through $S$; hence, $I\subset (z-1)$. Moreover, the $R[\epsilon]/\epsilon^2$-points of $S$ that specialize to identity at $\epsilon=0$ are in correspondence with such points of $\widehat{\G}$. Thus, $I\subset((z-1)^2)$.
$S$ is a subgroup of $\widehat{\G}$ over $Spf(R)$, thus the comultiplication \begin{equation}\Delta:\C[z^\pm][[q]]\rightarrow\C[z^\pm][[q]]\otimes\C[z^\pm][[q]]\cong\C[z_1^\pm,z_2^\pm][[q]],z\mapsto z\otimes z\cong z_1z_2 \end{equation} should induce a map $B\rightarrow B\otimes B$. Let $f(z)=g(z)(z-1)^2\in I$. \begin{equation}\Delta(f(z))=\Delta(g(z))(z_1z_2-1)^2\in I\otimes\C[z^\pm][[q]]+\C[z^\pm][[q]]\otimes I\subset\atop ((z_1-1)^2,(z_2-1)^2)\subset\C[z_1^\pm,z_2^\pm][[q]]\end{equation}
As $(z_1z_2-1)^2=(z_1z_2-z_2+z_2-1)^2\equiv 2(z_1-1)(z_2-1)(mod((z_1-1)^2,(z_2-1)^2)  )$, \begin{equation}\Delta(g(z))(z_1-1)(z_2-1)\in ((z_1-1)^2,(z_2-1)^2) \end{equation}Thus, $\Delta(g(z))\in (z_1-1,z_2-1)$. Thus, $g(z)\in (z-1)$ and $f(z)\in((z-1)^3)$; hence, $I\subset ((z-1)^3)$. Inductively, $I\subset (z-1)^n$ for all $n$. This shows $I=0$ and the embedding $S\rightarrow\widehat{\G}$ is an isomorphism.

Hence, we have a diagram $G\leftarrow S\xrightarrow{\cong}\widehat{\G}$ of groups and inverting the isomorphism a group homomorphism $\eta:\widehat{\G}\rightarrow G$. The pull-back of $\rho_{uni}$ along this map is the quasi-isomorphic to pull-back of $\rho''$ along $\eta\times 1:\widehat{\G}\rightarrow G\times\widehat{\G}$, which factors through $S$; hence, it is (locally) quasi-isomorphic to pull-back of $\rho'$ and thus to the family $\rho$ on $\widehat{\G}$(there is no non-trivial line bundle on $\widehat{\G}$ so local isomorphism is the same as isomorphism). This completes the proof.
\end{proof}
\begin{defn}
For a curved $A_\infty$-category $\B$ over $R=\C[[q]]$, define $L(\B)\subset HH^1(\B,\B)$ to be the set of           $<z\partial_z,Def(\rho)>$ for all group-like families $\rho$ parametrized by $\widehat{\G}$. In other words, it is the set of images of $(z\partial_z)^\#$ under the map $Lie(\widehat{\G})\rightarrow HH^1(\B,\B)$ induced by the deformation class of a group-like family $\rho$.
\end{defn}
\begin{cor}
$L(M_\phi^R)\subset HH^1(M_\phi^R,M_\phi^R)\cong R^2$ is a subgroup isomorphic to $\Z^2$ spanned by a basis of the free module $R^2$.	
\end{cor}
\begin{proof}
By Prop \ref{liftingcoroot}, the deformation class of a group-like family parametrized by $\widehat{\G}$ can be computed as the pull-back of $Def(\rho_{uni})$ under a group homomorphism $\widehat{\G}\rightarrow G\cong\widehat{\G}\times\widehat{\G}$. It is easy to classify such maps as the rank $2$ coroot lattice inside $R^2$ (the proof is exactly the same as $\mathbb G_{m,\C}$) and this observation together with Lemma \ref{lieinjective} concludes the proof.
\end{proof}
\section{Two relative spherical twists of the trivial mapping torus}\label{sec:symmetries}
To prove the theorem, we need to modify the Morita equivalence $M_\phi\simeq M_{1_\A}$ such that the induced isomorphism $HH^1(M_\phi,M_\phi)\cong HH^1( M_{1_\A}, M_{1_\A})$ carries $\gamma_\phi$ to $\gamma_{1_\A}$. It is easy to show $M_{1_\A}\simeq \A\otimes (\Om(\Tt_0)_{dg}\#\Z)$, and $\Om(\Tt_0)_{dg}\#\Z$ is Morita equivalent to wrapped Fukaya category of a punctured torus by \cite{lekpol}; thus, it must have a $SL_2(\Z)$-symmetry, which we would expect to act transitively on the primitive lattice points of $HH^1(M_{1_\A},M_{1_\A})$. We will not use action coming through this Morita equivalence, and instead we will write spherical twists that act in the desired way. The first twist is more general, but the second one only exists for the trivial mapping torus. 
\subsection{The twist of $M_{\phi}^R$ along a smooth point}\label{subsec:firsttwist}
We first find a self-Morita equivalence of the category $M_{\phi}^R$ that sends $\gamma_\phi^R$ to $\gamma_\phi^R\pm \gamma_2^R$ and that fixes $\gamma_2^R$. Here $\gamma_2^R\in HH^1(M_\phi^R,M_\phi^R)$ denotes the $R$-relative version of the Hochschild cocycle $\gamma_2$ defined in Lemma \ref{lieinjective} of Section \ref{sec:rank2}. Namely, $\gamma_2^R$ is the infinitesimal action of second $\widehat{\G}(R)$-action (which has weight equal to extra grading) and it forms a basis of $HH^1(M_\phi^R,M_\phi^R)\cong R^2$ together with $\gamma_\phi^R$. 

In this subsection, we will not officially refer to twists by spherical functors as defined for instance in \cite{logvinenko}. However, for those interested we remark that what we construct is equivalent to using twist by the functor \begin{equation}\A[[q]]\rightarrow M^R_\phi\atop ``a\mapsto a\otimes\Om_x" \end{equation}
where $\Om_x$ is the structure sheaf of a smooth $R$-section of $\T_R=\Tt_R/\Z$. After showing existence of right and left Morita adjoints, one can write the spherical twist and conclude that it is an equivalence at $q=0$ using \cite{logvinenko}. Then it is easy to show that invertible bimodules over $M_\phi$ deform only to invertible bimodules over $M_\phi^R$.
The spherical twist/cotwist by the structure sheaf of a smooth point $p$ on a curve $C$ is simply $(\cdot)\otimes\Om_C(p)$, resp. $(\cdot)\otimes\Om_C(-p)$, so we will use this directly. 

Let $p=p_0$ be a smooth $R$-point of $\Tt_R$ supported on $C_0$. Let $p_i=\tr^i(p)$ and consider the line bundle $\mathcal L_p=\Om\big(\sum_{i\in\Z}p_i \big)$. Define a bimodule over $\Om(\Tt_R)_{cdg}$ by the rule \begin{equation}\tilde\Lambda_p :(\scrF,\scrF')\mapsto hom^\cdot_{\Om_{\Tt_R}}(R(\scrF)_R,R(\scrF')_R\otimes \mathcal L_p ) \end{equation}Recall $R(\scrF)_R$,$R(\scrF')_R\otimes \mathcal L_p $ are pseudo-complexes of $\Om_{\Tt_R}$-modules and we are taking their hom pseudo-complexes as usual. This defines an (unobstructed for tautological reasons) $\Om(\Tt_R)_{cdg}$-bimodule. It is the bimodule corresponding to (non-existent) functor $(\cdot)\otimes \mathcal L_p$ and we will pretend as if it is this functor. One can make its restriction to $q=0$ into an actual functor by extending $\Om(\Tt_0)_{dg}$ with similar resolutions $R(\Om_{C_i}(a))$ of $\Om_{C_i}(a)$, for all $a\in\Z$. Call this bigger category $\Om(\Tt_0)_{dg}^{super}$. The line bundle $\mathcal L_p$ is invariant under $\tr_*$ and thus the $\Om(\Tt_R)_{cdg}\otimes \A$ bimodule $\tilde{\Lambda}_p\otimes \A$ has an obvious $\Z_\Delta$-equivariant structure. We can thus descent it to a bimodule \begin{equation}\Lambda_p=(\tilde{\Lambda}_p\otimes \A)\#\Z\end{equation}over $M_\phi^R$. Its restriction $\Lambda_p|_{q=0}$ still does not induce an $A_\infty$-functor; however, we can construct $M_\phi^{super}:=(\Om(\Tt_0)_{dg}^{super}\otimes\A)\#\Z\supset M_\phi$ (which is equivalent to a full subcategory of $tw^\pi(M_\phi)$) on which $\Lambda_p|_{q=0}$ acts as an $A_\infty$-functor ($``(\cdot)\otimes \Om(p)"$). It is easy to see this functor is a quasi-equivalence with a quasi inverse defined by the similar formula $``(\cdot)\otimes \Om(-p)"$ (more precisely by using $\mathcal{L}_p^{-1}$ in place of $\mathcal L_p$ in the definition of $\tilde{\Lambda}_p$).
This implies $\Lambda_p|_{q=0}$; thus, $\Lambda_p$ is invertible. 

Fix $p\in\Tt_R(R)$ as above. Let $\rho$ temporarily denote the group like family $\rho_{uni}|_{\G\times\{1\}}$, which corresponds to $\widehat{\G}(R)$-action on $M_\phi^R$ defined earlier in Remark \ref{actiononmt1}. Using $\Lambda_p$ we can define a new group like family \begin{equation}``\Lambda_p\circ \rho\circ \Lambda_p^{-1}":\simeq z\mapsto \Lambda_p\otimes_{M_\phi^R}\rho_z\otimes_{M_\phi^R}\Lambda_p^{-1}  \end{equation}The reason composition is in quotation marks is again that we have ``quasi-functors" instead of actual functors. However, we will abuse the notation and simply use composition symbol. By Proposition, \ref{liftingcoroot} this family can be seen as the restriction of $\rho_{uni}$ along a cocharacter of $\widehat{\G\times\G}$. We wish to compute this cocharacter. 

Consider instead the group-like family \begin{equation}\Lambda_p\circ \rho\circ \Lambda_p^{-1}\circ\rho^{-1}:z\mapsto \Lambda_p\otimes_{M_\phi^R} \rho_z\otimes_{M_\phi^R} \Lambda_p^{-1}\otimes_{M_\phi^R}\rho_z^{-1}\end{equation}
It can be seen as the composition of $\Lambda_p$ and $\rho\circ\Lambda_p^{-1}\circ \rho^{-1}$. Given $z\in\widehat{\G}$, we can compute $\rho_z\circ\Lambda_p^{-1}\circ\rho^{-1}_z\simeq \Lambda_{z.p}^{-1}$.
Hence, \begin{equation}\Lambda_p\circ \rho_z\circ \Lambda_p^{-1}\circ\rho^{-1}_z\simeq\Lambda_p\circ\Lambda_{z.p}^{-1} \end{equation}
\begin{lem}\label{linebundlebimod}
$\Lambda_p\circ\Lambda_{z.p}^{-1}$ is the bimodule obtained by replacing $\mathcal L_p$ by $\mathcal L_p\otimes\mathcal L_{z.p}^{-1}$ at the beginning. In other words, first consider \begin{equation}(\scrF,\scrF')\mapsto hom^\cdot_{\Om_{\Tt_R}}(R(\scrF)_R,R(\scrF')_R\otimes \mathcal L_p\otimes\mathcal L_{z.p}^{-1} ) \end{equation} then take its exterior product with $\A$ and descent to $M_\phi^R$.	
\end{lem}
\begin{proof}(Sketch)
Instead of showing this leads to $\Lambda_p\circ\Lambda_{z.p}^{-1}$ for individual $z$, one can specialize to $q=0$ and compare bimodules over $M_\phi$. Extending $M_\phi$ as above, both  $\Lambda_p\circ\Lambda_{z.p}^{-1}$ and the bimodule corresponding to $\mathcal L_p\otimes\mathcal L_{z.p}^{-1}$ can be realized as actual $A_\infty $-functors and both are enhancements of $``(\cdot)\otimes \mathcal L_p\otimes\mathcal L_{z.p}^{-1}"$. Hence, the specializations to $q=0$ are the same.
As families of bimodules over deformation $M_\phi^R$, they are group-like which correspond to coroots of $\widehat{\G\times\G}$ speacialing to same coroots of $\G\times\G$. Hence, by discreteness of the coroots we conclude they are the same. 
\end{proof}
Now let us turn our attention to the line bundle \begin{equation}\mathcal L_p\otimes \mathcal L^{-1}_{z.p}=\Om_{\Tt_R}\Bigg(\sum_{i\in\Z}(p_i-z.p_i) \Bigg)\end{equation}The equivariant structure of the induced bimodule comes from the obvious isomorphism \begin{equation}\tr_*\Om_{\Tt_R}\Bigg(\sum_{i\in\Z}(p_i-z.p_i) \Bigg)\cong \Om_{\Tt_R}\Bigg(\sum_{i\in\Z}(p_i-z.p_i) \Bigg)\end{equation}
\begin{lem}
$\mathcal L_p\otimes \mathcal L^{-1}_{z.p}$ admits a trivialization $G\in \Gamma(\mathcal L_p\otimes \mathcal L^{-1}_{z.p})$ such that under the isomorphism $\Om_{\Tt_R}\xrightarrow{\cong}\mathcal L_p\otimes \mathcal L^{-1}_{z.p}$ the equivariant structure $\mathcal L_p\otimes \mathcal L^{-1}_{z.p}\rightarrow \tr_*(\mathcal L_p\otimes \mathcal L^{-1}_{z.p})$ identifies with $\Om_{\Tt_R}\xrightarrow{z^{-1}}\tr_*\Om_{\Tt_R}=\Om_{\Tt_R}$.
\end{lem}
\begin{proof}
A section of $\mathcal L_p\otimes \mathcal L^{-1}_{z.p}$ is a rational function on $\Tt_R$ with simple poles at $p_i$ and zeroes at $z.p_i$, for all $i\in\Z$. Denote the $Y_0$-coordinate of the smooth point $p=p_0\in \tU_{1/2}$ by $y_0\in R^*$. We can find a section using convergent infinite products. Namely, consider the chart $\tU_{i+1/2}=Spf(\C[X_i,Y_{i+1}][[q]]/(X_iY_{i+1}-q))$. Define the rational function $\tilde G_{i+1/2}$ on $\tU_{i+1/2}$ by the formula \begin{equation}\tilde G_{i+1/2}:=\prod_{j=0}^{\infty}\frac{1-q^jzy_0X_i}{1-q^jy_0X_i} \prod_{j=0}^{\infty}\frac{1-q^jz^{-1}y_0^{-1}Y_{i+1}}{1-q^jy_0^{-1}Y_{i+1}} \end{equation} 
Its convergence is obvious by $q$-adic completeness. On $\tU_{i+1/2}$ it is \begin{equation}\frac{1-zy_0X_i}{1-y_0X_i}\frac{1-z^{-1}y_0^{-1}Y_{i+1}}{1-y_0^{-1}Y_{i+1}}\end{equation} up to invertible functions. Using the relations $X_iY_i=1$, $X_{i-1}=qX_i$ and $Y_{i+1}=qY_i$ we can compare \begin{equation}\frac{\tilde G_{i+1/2}}{\tilde G_{i-1/2}}=\frac{(1-zy_0X_i)/(1-y_0X_i)}{(1-z^{-1}y_0^{-1}Y_i)/(1-y_0^{-1}Y_i)}=z \end{equation} on $\tU_{i-1/2}\cap \tU_{i+1/2}$. Now define $G=G(z)\in \Gamma (\mathcal L_p\otimes \mathcal L^{-1}_{z.p})$ locally by the formula \begin{equation}G_{i+1/2}:=z^{-i}\tilde G_{i+1/2}\end{equation}
This gives a trivialization of $\mathcal L_p\otimes \mathcal L^{-1}_{z.p}$ and $G\circ\tr=z^{-1}G$ as a rational function.
Hence, the equivariant structure turns into $\Om_{\Tt_R}\xrightarrow{z^{-1}}\tr_*\Om_{\Tt_R}=\Om_{\Tt_R}$ under the identification $\mathcal L_p\otimes\mathcal L_{z^{-1}.p}^{-1}\cong \Om_{\Tt_R}$.
\end{proof}
Hence, the bimodule \begin{equation}(\scrF,\scrF')\mapsto hom^\cdot_{\Om_{\Tt_R}}(R(\scrF)_R,R(\scrF')_R\otimes \mathcal L_p\otimes\mathcal L_{z.p}^{-1} )\end{equation} identifies with the diagonal bimodule of $\Om(\Tt_R)_{cdg}$. Moreover, its $\Z_\Delta$-equivariant structure is \begin{equation}\Om(\Tt_R)_{cdg}(\scrF,\scrF')\xrightarrow{z.\tr} \Om(\Tt_R)_{cdg}(\tr\scrF,\tr\scrF') \end{equation}(while $\Om(\Tt_R)_{cdg}(\scrF,\scrF')\xrightarrow{\tr} \Om(\Tt_R)_{cdg}(\tr\scrF,\tr\scrF') $ is the $\Z_\Delta$-equivariant structure descending to diagonal).
If we descent $\Om(\Tt_R)_{cdg}\otimes\A $ with respect to $\Z_\Delta$ equivariant structure $z.\tr\otimes 1_\A$, we obtain the bimodule \begin{equation}_z(M_\phi^R)_1\cong _1(M_\phi^R)_{z^{-1}}\end{equation} It is the bimodule with same underlying pseudo-complexes and right action as $M_\phi^R$ but with left action twisted by action of $z$ by extra grading (i.e. $(f\otimes h)(m\otimes g)=z^hfh(m)\otimes hg$). Hence, depending on the convention this is the bimodule $\rho_{uni,(1,z)}$ or $\rho_{uni,(1,z^{-1})}$. Everything above can be done relative to $z\in \widehat{\G}$ and we conclude 
\begin{cor}\label{somecommutation}
$\Lambda_p\circ \rho\circ\Lambda_p^{-1}\circ \rho^{-1}$ is quasi-isomorphic to $\rho_{uni}|_{\{1\}\times\widehat\G }$ or its composition with the antipode $z\mapsto z^{-1}:\{1\}\times\widehat\G\rightarrow\{1\}\times\widehat\G$.
\end{cor}
Denote $\rho_{uni}|_{\{1\}\times\widehat \G}$ temporarily by $\rho_2$
\begin{cor}
$\Lambda_p\circ \rho\circ\Lambda_p^{-1}\simeq \rho\circ\rho_2^{\pm}$.	
\end{cor}
Taking their deformation classes we find 
\begin{cor}
Under the automorphism of $HH^1(M_\phi^R)$ induced by $\Lambda_p$, $\gamma_\phi^R$ corresponds to $\gamma_\phi^R\pm \gamma_2^R$.
\end{cor}
Now we want to show $\Lambda_p$ fixes $\gamma_2^R$ and $\gamma_2$. For this we will again examine $\rho_2\circ\Lambda_p\circ\rho_2^{-1}$. A systematic approach would be first proving $\Lambda_p$ is the same as the twist by  \begin{equation}\A[[q]]\rightarrow M^R_\phi\atop ``a\mapsto a\otimes\Om_p" \end{equation} as mentioned above and then showing its conjugate by $\rho_{2,z}$ is the same as the twist by the composition of the spherical functor with conjugation by $\rho_{2,z}$. For instance, in the case $\A=\C$, this is given as the twist by $``\rho_{2,z}(\Om_p)=\Om_p"$; hence, it is the same. However, we take a simpler approach. 
\begin{lem}$\rho_{2,z}\circ \Lambda_p\circ\rho_{2,z}^{-1}\simeq \Lambda_p$
\end{lem}
\begin{proof}
By Corollary \ref{somecommutation} \begin{equation}\Lambda_p\circ\rho_z\circ\Lambda^{-1}_p\circ\rho_z^{-1}\simeq \rho_{2,z}^\pm\end{equation}
Hence, it is sufficient to show \begin{equation}\Lambda_p^{-1}\circ\Lambda_p\circ\rho_z\circ\Lambda^{-1}_p\circ\rho_z^{-1}\circ\Lambda_p\simeq \Lambda_p\circ\rho_z\circ\Lambda^{-1}_p\circ\rho_z^{-1} \end{equation}
Clearly, the former is quasi-isomorphic to $\rho_z\circ\Lambda^{-1}_p\circ\rho_z^{-1}\circ\Lambda_p$, which is simply
$\Lambda_{z.p}^{-1}\circ \Lambda_p$. By the proof of Lemma \ref{linebundlebimod}, $\Lambda_{z.p}^{-1}\circ \Lambda_p$ is given by the descent of the bimodule corresponding to the same line bundle namely $\mathcal L_p\otimes \mathcal L_{z.p}^{-1}$, with the same $\Z_\Delta$-equivariant structure. Hence \begin{equation}\Lambda_{z.p}^{-1}\circ \Lambda_p\simeq \Lambda_p\circ\Lambda_{z.p}^{-1} \end{equation}We have shown before that $\Lambda_p\circ\Lambda_{z.p}^{-1}\simeq \Lambda_p\circ\rho_z\circ\Lambda_p^{-1}\circ\rho_z^{-1}$. This completes the proof. 
\end{proof}
\begin{cor}
The induced action of $\Lambda_p$ on $HH^1(M_\phi^R)$ fixes $\gamma_2^R$.	
\end{cor}
\begin{proof}
Compare the deformation classes of $\Lambda_p\circ \rho_2\circ \Lambda_p^{-1}\simeq \rho_2$.
\end{proof}
\begin{rk}
Notice we can twist the family $\scrG_R^{sf}$ by \begin{equation}\scrG_{R}':=``\Lambda_p\circ\scrG_R^{sf}\circ\Lambda_p^{-1}" \end{equation}to obtain a family that follows $1\times(\gamma_\phi^R\pm \gamma_2^R)$. It is easy to see that the family satisfies Properties \ref{G1}-\ref{G3} with $\gamma=\gamma_\phi^R\pm \gamma_2^R$. One can attempt to use ``convolutions of the families of bimodules relative to $Spf(A_R)$'' to produce families following other classes in $L(M_\phi^R)\subset HH^1(M_\phi^R,M_\phi^R)$. However, we do not know how to show property \ref{G1} for the new family.
\end{rk}
\subsection{The twist of $M_{1_\C}$ along the ``structure sheaf''}
The second twist is more restrictive. We can still work the the curved algebra $M_{1_\C}^R$; however, we will not do this. 
We find a self Morita equivalence of $M_{1_\A}$ that fixes $\gamma_{1_\A}$ and that carries $\gamma_2$ to $\gamma_2\pm \gamma_{1_\A}$. It is sufficient to do this for $\A=\C$. In the following $\gamma_1$ will denote $\gamma_{1_\C}$.

We can work as in the previous subsection. However, we find it conceptually relieving to relate $M_{1_\C}$ to algebraic geometry. Hence, we wish to start by sketching a proof of a weaker version of the claim in Example \ref{exmpalggeo}. Namely:
\begin{lem}\label{categorycomparisonlemma}
$tw^\pi(M_{1_\C})$ is a dg enhancement of $D^b(Coh(\T_0))$, where $\T_0$ is the nodal elliptic curve. 
\end{lem}
\begin{proof}
$\T_0$ can be realized as $``\Tt_0/(x\sim\tr(x))"$ and we have a projection map $\pi:\Tt_0\rightarrow \T_0$ (denoted by $\pi$ only throughout this proof). Choose dg-models $\scrC oh_p(\Tt_0)$ and $\scrC oh(\T_0)$ for $D^b(Coh_p(\Tt_0) )$ and $D^b(Coh(\T_0))$ such that 
\begin{itemize}
\item There exists a dg functor $\pi_*:\scrC oh_p(\Tt_0)\rightarrow\scrC oh(\T_0)$ enhancing the push-forward by $\pi$
\item $\tr$ induces a strict action $\tr_*$ on $\scrC oh_p(\Tt_0)$
\item $\pi_*\circ\tr_*=\pi_*$ (strictly)
\end{itemize}
We can further assume $\scrC oh_p(\Tt_0)$ has the objects $\{\Om_{C_i}(-1),\Om_{C_i}:i\in\Z \}$ and denote $\tr_*$ by $\tr$ following the previous convention. We can also assume there exists a zigzag of strictly $\Z$-equivariant dg quasi-equivalences relating $\scrC oh_p(\Tt_0)$ and $\Om(\Tt_0)_{dg}$. Hence, $\scrC oh_p(\Tt_0)\#\Z\simeq \Om(\Tt_0)_{dg}\#\Z$. The relation $\pi_*\circ\tr=\pi_*$ implies $\pi_*$ descends to \begin{equation}\scrC oh_p(\Tt_0)\#\Z\rightarrow \scrC oh(\T_0)\atop \scrF\longmapsto\pi_*(\scrF) \end{equation}Let $f\in \scrC oh_p(\Tt_0)(\tr^g\scrF,\scrF')$ and consider $f$ as an element of $(\scrC oh_p(\Tt_0)\#\Z)(\scrF,\scrF')$(recall we denoted it by $f\otimes g$). It is sent to \begin{equation}\pi_*(f)\in \scrC oh(\T_0)(\pi_*(\tr^g\scrF),\pi_*(\scrF'))=\scrC oh(\T_0)(\pi_*(\scrF),\pi_*(\scrF'))\end{equation} under the new functor.
Denote the new functor by $\pi_*$ as well. 

The induced functor between homotopy categories of twisted envelopes is essentially surjective. This follows from the fact that the push-forward of $\Om_{\mathbb P^1}(-1)$ and $\Om_{\mathbb P^1}$ under the normalization map $\mathbb{P}^1\rightarrow \T_0$ generates $D^b(Coh(\T_0))$. See \cite{lekpol}. 

To conclude the proof, we need to check (cohomological) fully faithfulness of the functor \begin{equation}\scrC oh_p(\Tt_0)\#\Z\rightarrow \scrC oh(\T_0)\atop \scrF\longmapsto\pi_*(\scrF) \end{equation}
We do this only for $\scrF=\scrF'=\Om_{C_0}$ as the others are similar. First notice \begin{equation}(\scrC oh_p(\Tt_0)\#\Z)(\Om_{C_0},\Om_{C_0})=\atop \scrC oh_p(\Tt_0)(\Om_{C_{-1}},\Om_{C_0})\oplus\scrC oh_p(\Tt_0)(\Om_{C_{0}},\Om_{C_0})\oplus \scrC oh_p(\Tt_0)(\Om_{C_{1}},\Om_{C_0}) \end{equation}
and its cohomology is \begin{equation}RHom_{\Om_{\Tt_0}}(\Om_{C_{-1}},\Om_{C_0})\oplus RHom_{\Om_{\Tt_0}}(\Om_{C_{0}},\Om_{C_0})\oplus RHom_{\Om_{\Tt_0}}(\Om_{C_{1}},\Om_{C_0}) \end{equation}
A simple local computation (i.e. calculating local hom's and their global sections) reveals $RHom_{\Om_{\Tt_0}}(\Om_{C_{-1}},\Om_{C_0})$ is one dimensional in every positive odd degree and $0$ in other degrees. Same holds for $RHom_{\Om_{\Tt_0}}(\Om_{C_{1}},\Om_{C_0})$. On the other hand,
$RHom_{\Om_{\Tt_0}}(\Om_{C_{0}},\Om_{C_0})$ can be calculated to be one dimensional in degree $0$ and two dimensional in positive even degrees. 

We can compute local hom's of $\pi_*\scrF$ and $\pi_*\scrF'$ on the \'etale chart \begin{equation}\pi:Spec(\C[X_0,Y_1]/(X_0Y_1))\rightarrow \T_0\end{equation} and see that $RHom_{\T_0}(\pi_*\Om_{C_0},\pi_*\Om_{C_0})$ is $1$ dimensional in degree $0$ and $2$ dimensional in higher degrees. Hence, degrees match up and it is hidden in the local computation that the map \begin{equation}RHom_{\Om_{\Tt_0}}(\Om_{C_{-1}},\Om_{C_0})\oplus RHom_{\Om_{\Tt_0}}(\Om_{C_{0}},\Om_{C_0})\oplus RHom_{\Om_{\Tt_0}}(\Om_{C_{1}},\Om_{C_0})\atop \rightarrow RHom_{\T_0}(\pi_*\Om_{C_0},\pi_*\Om_{C_0})  \end{equation} is an isomorphism.
\end{proof}
Notice, $\Om_{\T_0}$ is a $1$-spherical object of $D^b(Coh(\T_0))$ in the sense of \cite[Definition 2.9]{seidelthomas}. In other words, 
\begin{itemize}
\item $RHom_{\T_0}(\Om_{\T_0},\scrF)$ and $RHom_{\T_0}(\scrF,\Om_{\T_0})$ are finite dimensional for all $\scrF\in D^b(Coh(\T_0))$
\item $RHom_{\T_0}(\Om_{\T_0},\Om_{\T_0})=\C\oplus\C[-1]$
\item $RHom^j_{\T_0}(\scrF,\Om_{\T_0})\times RHom^{1-j}_{\T_0}(\Om_{\T_0},\scrF)\rightarrow RHom_{\T_0}(\Om_{\T_0},\Om_{\T_0})\cong \C$, composition map, is a non-degenerate pairing
\end{itemize}
The first and second conditions are immediate and the third one follows from Serre duality. Note, we are ignoring the condition they call (K1)) as it can be arranged by choosing an appropriate representative of $\Om_{\T_0}$ in the enhancement.

Hence, by \cite[Proposition 2.10]{seidelthomas}, there exists a (quasi-)equivalence $T_{\Om_{\T_0}}$ of (an enhancement of) $D^b(Coh(\T_0))$ fitting into an exact triangle \begin{equation}RHom_{\T_0}(\Om_{\T_0},\cdot)\otimes \Om_{\T_0}\rightarrow (\cdot)\rightarrow T_{\Om_{\T_0}}(\cdot)\rightarrow RHom_{\T_0}(\Om_{\T_0},\cdot)\otimes \Om_{\T_0}[1] \end{equation}
Thus, by Lemma \ref{categorycomparisonlemma}, there exists an object of $tw^\pi(M_{1_\C})$ corresponding to $\Om_{\T_0}$ and a self Morita equivalence of $M_\phi$, which we denote by $\Lambda_\Om$. This is the second twist we are looking for. Next we examine its effect on $\gamma_1$ and $\gamma_2$.
\begin{rk}\label{nodalactionremark}
The actions $\rho_1$ and $\rho_2$ induce $\G$-actions on $D^b(Coh(\T_0))$. $\rho_1$ is already induced by the geometric action in Remark \ref{actionremark}, hence its induced action on $D^b(Coh(\T_0))$ comes from the action $\G\curvearrowright \T_0$ making $\Tt_0\rightarrow \T_0$ equivariant. In other words, it is the action of $Aut^0(\T_0)\cong\G$. To describe induced $\rho_2$-action first note $\Lambda_p\curvearrowright D^b(Coh(\T_0))$ is simply tensoring with $\Om_{\T_0}(p)$, where we use $p$ to denote the image of $p_0\in \Tt_0\rightarrow\T_0$ as well. Hence, $\Lambda_p\circ\Lambda_{z.p}^{-1}$ acts by $\Om_{\T_0}(p-z.p)$. By Section \ref{subsec:firsttwist}, this action is the induced action of $\rho_2^\pm$. In other words, $\rho_2$-action induces the action of $Pic^0(\T_0)$ on $D^b(Coh(\T_0))$. In summary, $\rho_{uni}$ induces the action of $Aut^0(\T_0)\times Pic^0(\T_0)\cong\G\times \G$ on $D^b(Coh(\T_0))$.
\end{rk}
Consider $\rho_{1,z}\circ\Lambda_\Om\circ\rho_{1,z}^{-1}$. This is the twist by the $1$-spherical object $\rho_{1,z}(\Om_{\T_0})$. By Remark \ref{nodalactionremark}, $\rho_{1,z}(\Om_{\T_0})\simeq \Om_{\T_0}$; hence, $\rho_{1,z}\circ\Lambda_\Om\circ\rho_{1,z}^{-1}\simeq \Lambda_\Om$. In other words, $\rho_1$ commutes with $\Lambda_\Om$ and by taking deformation classes, we conclude the map induced by $\Lambda_\Om$ sends $\gamma_1$ to itself.

On the other hand, consider the commutator $\rho_{2,z}\circ\Lambda_\Om\circ\rho_{2,z}^{-1}\circ \Lambda_\Om^{-1}$. As $z$ varies, this gives a group like family, determined by a cocharacter of $\rho_{uni}$, thanks to Prop \ref{liftingcoroot}. We want to determine this cocharacter. A quick calculation shows that for any two smooth points $q,q'\in\T_0$, $\Lambda_\Om(\Om_{\T_0}(q-q'))=\Om_{\T_0}(q-q')$. 
Hence, \begin{equation}\rho_{2,z}\circ\Lambda_\Om\circ\rho_{2,z}^{-1}\circ \Lambda_\Om^{-1}(\Om_{\T_0})=\rho_{2,z}\circ\Lambda_\Om\circ\rho_{2,z}^{-1}(\Om_{\T_0})=\atop\rho_{2,z}\circ\Lambda_\Om(\Om_{\T_0}(p-z^{\mp}.p) )=\rho_{2,z}(\Om_{\T_0}(p-z^{\mp}.p) )=\Om_{\T_0} \end{equation}
This implies the second component of the cocharacter of $Aut^0(\T_0)\times Pic^0(\T_0)$ vanishes (since it fixes $\Om_{\T_0}$). On the other hand, for the smooth point $p\in \T_0$, $\Lambda_\Om(\Om_p)=\Om_{\T_0}(-p)[1]$. If we apply $\rho_{2,z}\circ\Lambda_\Om\circ\rho_{2,z}^{-1}\circ \Lambda_\Om^{-1}$ to $\Lambda_\Om(\Om_p)$ we obtain $\Om_{\T_0}(-z^{\pm}.p)[1]$. This shows, $Aut^0(\T_0)$ component is the cocharacter of weight $\pm 1$. This implies \begin{equation}\rho_{2,z}\circ\Lambda_\Om\circ\rho_{2,z}^{-1}\circ \Lambda_\Om^{-1}=\rho_1^\pm \end{equation} and thus $\Lambda_\Om\circ\rho_{2}\circ \Lambda_\Om^{-1}=\rho_2\circ \rho_1^{\mp}$. By taking the deformation classes we conclude:
\begin{cor}
The map induced by $\Lambda_\Om$ on $HH^1(M_{1_\C})$ sends $\gamma_1$ to $\gamma_1$ and $\gamma_2$ to $\gamma_2\mp \gamma_1$.
\end{cor}
To conclude the section, we have found two self-Morita equivalences $\Lambda_p$ and $\Lambda_\Om$ of $M_{1_{\A}}$ that acts on $HH^1(M_{1_\A})$ by the matrices $\left[\begin{smallmatrix}1&0\\\pm1&1\end{smallmatrix}\right]$ and $\left[\begin{smallmatrix}1&\mp1 \\0&1\end{smallmatrix}\right]$ respectively in $\{\gamma_1,\gamma_2 \}$ basis. The action of any self-Morita equivalence has to preserve the lattice $L(M_\phi):=L(M_\phi^R)|_{q=0}\subset HH^1(M_\phi)$. It is easy to show these matrices generate the group $SL(L(M_\phi))\cong SL_2(\Z)$. Indeed, it is a classical fact that $SL_2(\Z)$ is generated by $\left[\begin{smallmatrix}1&1 \\0&1\end{smallmatrix}\right]$ and $\left[\begin{smallmatrix}0&-1 \\1&0\end{smallmatrix}\right]$. See \cite{coursarithmetic} for instance. The latter matrix can easily be obtained as \begin{equation}\left[\begin{smallmatrix}0&-1 \\1&0\end{smallmatrix}\right]=\left[\begin{smallmatrix}1&0\\1&1\end{smallmatrix}\right]\left[\begin{smallmatrix}1&-1\\0&1\end{smallmatrix}\right]\left[\begin{smallmatrix}1&0\\1&1\end{smallmatrix}\right] \end{equation}
\begin{cor}\label{sl2cor}
The group of self-Morita equivalences of $M_{1_\A}$ act transitively on primitive vectors of the lattice $L(M_{1_\A})\cong \Z^2$.
\end{cor}
\section{Uniqueness of family of bimodules and the proof of the main theorem}\label{sec:unique}
In this section, we will use the previous sections to conclude the proof of Theorem \ref{mainthm}. In other words we will prove:
\begingroup
\def\thethm{\ref*{mainthm}}
\begin{thm}
Let $\A$ and $\phi$ be as in Section \ref{sec:intro}, i.e. satisfying \ref{C1}-\ref{C2} and so on. Assume further that  $HH^1(\A)=HH^2(\A)=0$. If $M_\phi$ is Morita equivalent to $M_{1_\A}$, then $\phi\simeq 1_\A$.
\end{thm}
\addtocounter{thm}{-1}
\endgroup
To prove this theorem, we will give a characterization of the family of bimodules $\scrG_R^{sf}$. Let us first work in a more general setting. Let $\B_0$ be an $A_\infty$- category and $\B$ be a curved deformation over $R=\C[[q]]$. Let $\fM$ be a family of bimodules over $\B$ parametrized by $Spf(A_R)$ and let $\gamma\in HH^1(\B)$. Consider the properties:
\begin{enumerate}[label=\textbf{G.\arabic*}]
\item\label{G1} The restriction $\fM|_{q=0}$ is a coherent family. This is equivalent to its representability by an object of $tw^\pi(\B_0\otimes\B_0^{op}\otimes ``\mathcal{C}oh(A)")$. See Definition \ref{compactfamdef}.
\item\label{G2} The restriction $\fM|_{t=1}$ is isomorphic to the diagonal bimodule over $\B$.
\item\label{G3} The family follows the class $1\otimes \gamma\in HH^1(\B^e)$.
\end{enumerate}
The semi-freeness of $\fM$ is implied by the family assumption. The property \ref{G1} is a technical one. However, notice the similarity of properties \ref{G2} and \ref{G3} to an initial value problem. More precisely, \ref{G3} is analogous to flow equation for a vector field, and \ref{G2} is analogous to setting an initial condition. 
We show
\begin{thm}\label{uniquenesstheorem}
Assume $\B_0$ is smooth, proper in each degree and $HH^0(\B_0)=\C$. Let $\fM$ and $\fM'$ be two families of bimodules satisfying \ref{G1}-\ref{G3}. Then $\fM$ and $\fM'$ are isomorphic up to $q$-torsion. In other words, there are maps \begin{equation}f_1:\fM\rightarrow\fM',f_2:\fM'\rightarrow\fM\end{equation} in the category $(\B^e)_{A_R}^{mod}$ of families of bimodules such that $f_2\circ f_1\simeq q^n1_{\fM}$ and $f_1\circ f_2\simeq q^n1_{\fM'}$.
\end{thm}
\begin{proof}
Let $Hom(\fM,\fM')$ denote $H^0((\B^e)_{A_R}^{mod})(\fM,\fM')$ throughout the proof. First, notice that it is finitely generated over $A_R$. To see this consider the complex \begin{equation}(\B_0^e)_{A}^{mod}(\fM|_{q=0},\fM'|_{q=0})\end{equation}of $A$-modules. Here, $(\B_0^e)_{A}^{mod}$ is the category of families of $\B_0$-bimodules parametrized by $Spec(A)$, which can be defined analogously. As stated in Lemma \ref{cptfamilylem}, the condition \ref{G1} implies that this complex has cohomology that is finitely generated over $A$ in each degree. Thus, by Lemma \ref{cptfamilycor} or \ref{semicont2} the same holds for the complex \begin{equation}(\B^e)_{A_R}^{mod}(\fM,\fM')\end{equation} and $Hom(\fM,\fM')$ is finitely generated. 

Second, by Proposition \ref{followsame}, the complex $(\B^e)_{A_R}^{mod}(\fM,\fM')$ admits a homotopy connection along $\precon_{A_R}$; thus, so is its cohomology. In particular, the $A_R$-modules $Hom (\fM,\fM')$, $Hom (\fM',\fM)$ and so on carry connections along $A_R$.

Applying Lemma \ref{fgcomplex} to this complex we see that \begin{equation}Hom(\fM,\fM')/(t-1)Hom(\fM,\fM')\cong H^0((\B^e)^{mod}(\fM|_{t=1},\fM'|_{t=1})) \end{equation}Here we are also using the fact that the restriction of $(\B^e)_{A_R}^{mod}(\fM,\fM')$ to $t=1$ gives the hom-complex $(\B^e)^{mod}(\fM,\fM')$ and this follows from semi-freeness of families over $A_R$. However, by condition \ref{G2}, $H^0((\B^e)^{mod}(\fM|_{t=1},\fM'|_{t=1}))$ is simply the self-endomorphisms of the diagonal; which is computed by Hochschild cohomology. Hence, the assumption $HH^0(\B_0)=\C$ implies $H^0((\B^e)^{mod}(\fM|_{t=1},\fM'|_{t=1}))\cong R$.

In summary $Hom(\fM,\fM')$ is a finitely generated $A_R$-module with a connection whose restriction to $t=1$ is isomorphic to $R$. Hence, by Proposition \ref{freeuptoqtors}, it is free of rank $1$, up to $q$-torsion. In other words, there is a map $Hom(\fM,\fM')\to A_R$ with $q$-torsion kernel and cokernel; hence, there exists an $f\in Hom(\fM,\fM')$ and $k\in\mathbb{N}$ satisfying the following: for every $x\in Hom(\fM,\fM')$, there exists a unique $a\in A_R$ such that $q^kx=af+y$ for some $q$-torsion element $y$ (by increasing $k$, we can ensure $y$ vanishes, assume this holds). The same is true for $Hom(\fM',\fM)$, $Hom(\fM,\fM)$ and $Hom(\fM',\fM')$. Choose such an elements $f\in Hom(\fM,\fM')$, $g\in Hom(\fM',\fM)$ with the same $k\in\mathbb{N}$.

Moreover, the composition map \begin{equation}\label{eq:comp111}Hom(\fM',\fM)\otimes_{A_R} Hom(\fM,\fM')\rightarrow Hom(\fM,\fM) \end{equation} (again the tensor product is $q$-adically completed) has kernel and cokernel that are $q$-torsion. To see this consider the cokernel $C$. By the compatibility of the connection with composition in Proposition \ref{followsame}, the image and the cokernel carry connections along $D_{A_R}$. Moreover, the restriction of composition map to $t=1$ gives the composition of families restricted to $t=1$; hence, it is an isomorphism and $C/(t-1)C=0$. Using Proposition \ref{qtorsmod} we see that $C$ is $q$-torsion. 

Hence, there exists an $m$ such that $q^m1_\fM$ is in the image of (\ref{eq:comp111}). By increasing $m$, we can ensure an element of the form $a(g\otimes f)$ maps to $q^m1_\fM$. 
Similarly, we can ensure there exists an element of the form $a'(f\otimes g)$ that maps to $q^m1_{\fM'}$ under composition. Hence, $q^mag=ag\circ f\circ a'g=q^ma'g$. Letting $f_1=f$, $f_2=q^mag$ proves the statement of the theorem.
\end{proof}
\begin{rk}
A version of Theorem \ref{uniquenesstheorem} for families over smooth complex curves is proven in \cite[Prop 1.21]{flux}. We follow a similar idea.
\end{rk}
\begin{lem}\label{semicont2}
Let $C^*$ be a complex of $q$-adically complete $A_R$-modules that are free of $q$-torsion. Assume the cohomology of the complex $C^*|_{q=0}=C^*/qC^*$ of $A$-modules is finitely generated over $A$ in each degree. Then, $H^*(C^*)$ is finitely generated over $A_R$ in each degree. 
\end{lem}
\begin{proof}(Sketch)
Pick $y_1,\dots,y_n\in C^i/qC^i$ that are closed and whose classes generate $H^i(C^*/qC^*)$ as an $A$-module. Consider the module $A<y_1,\dots,y_n>\subset C^i/qC^i$ and consider its submodule of elements $x$ such that there exists an $\tilde{x}\in C^i$ that deform $x$ and satisfying $d(\tilde x)=0$. This submodule is finitely generated over $A$ as well and we can find closed elements $\tilde{x}_1,\dots , \tilde{x}_m\in C^i$ whose restrictions to $q=0$ generate this submodule of deforming elements. Now, it is easy to see the cohomology classes of $\tilde{x}_1,\dots,\tilde{x}_m$ generate $H^i(C^*)$ over $A_R$.
\end{proof}
\begin{prop}
The family $\scrG_R^{sf}$ satisfies the conditions \ref{G1}-\ref{G3} for $\gamma=\gamma_\phi^R$.
\end{prop}
\begin{proof}
We have already shown \ref{G1} in Proposition \ref{compactfamily} and \ref{G3} in Corollary \ref{gfollows}. See also Remark \ref{grk}. To see \ref{G2}, notice $\cG_R|_{t=1}\subset \Tt_R\times \Tt_R$ is the diagonal by defining equations (\ref{eq:graph1}) and (\ref{eq:graph2}). Hence, it induces the diagonal bimodule of $\Om(\Tt_R)_{cdg}$, which descends to diagonal bimodule of $M_\phi^R$.
\end{proof}
\begin{rk}\label{rkforuone}
Similarly, by (\ref{eq:graph1}) and (\ref{eq:graph2}), $\cG_R|_{u=1}\subset\Tt_R\times\Tt_R$ is the graph of $\tr^{-1}$. Hence, the bimodule $\scrG_R^{pre}$ is quasi-isomorphic to 
\begin{equation}(\scrF,\scrF')\mapsto\Om(\Tt_R)_{cdg}(\scrF,\tr(\scrF')) \end{equation}
Since we take the smash product with action generated by $\tr\otimes \phi$, the bimodule induced on $M_\phi^R=(\Om(\Tt_R)_{cdg}\otimes \A)\#\Z$ is given by \begin{equation}(\scrF\otimes a,\scrF'\otimes a' )\mapsto M_\phi^R(\scrF\otimes a, (1\otimes \phi^{-1})(\scrF'\otimes a'))=M_\phi^R(\scrF\otimes a, \phi^{-1}_f(\scrF'\otimes a')) \end{equation}where $\phi_f$ is the ``fiberwise $\phi$'' functor, which will be defined in Section \ref{sec:another}.
\end{rk}
Before going back to main theorem, we state some lemmas in homological algebra and abstract deformation theory:
\begin{lem}\label{lem:normalsupercategory}
Let $\B$ and $\B'$ be Morita equivalent $A_\infty$-categories over $\C$. Then, there exists an $A_\infty$-category $\tilde B$ that contains both $\B$ and $\B'$ as full, split generating $A_\infty$-subcategories with disjoint sets of objects. Moreover, this category can be chosen such that $\tilde \B$, considered as an $A_\infty$-bimodule over $\B$-$\B'$ realizes any given Morita equivalence.
\end{lem}
\begin{proof}
A fixed Morita equivalence induces a quasi-equivalence $perf(\B)\to perf(\B')$. Moreover, the restriction of $perf(\B)$-$perf(\B')$-bimodule corresponding to this quasi-equivalence along the Yoneda embeddings is quasi-isomorphic to the initial Morita equivalence. Therefore, the categories $\B$ and $\B'$ admit extensions (within their split closed twisted envelope for instance) that are quasi-equivalent and such that the bimodule induced by the quasi-equivalence restricts to the initial Morita equivalence. Hence, without loss of generality, we can assume $\B$ and $\B'$ are quasi-equivalent, and the Morita equivalence is induced by this quasi-equivalence (which we call $f:\B\to \B'$). 

We start by constructing a category $\underline{\B}$ and faithful, cohomologically fully faithful functors $i:\B\to \underline{\B}$, $i':\B'\to\underline{\B}$ with split generating images. Let $\B_0$ be another category with quasi-equivalences $j:\B_0\to \B$ and $j':\B_0\to \B'$ such that $f\circ j\simeq j'$. For instance, one can let $\B_0=\B$, $j=1_\B$, $j'=f$. One can construct a category $Groth (\B\leftarrow \B_0\rightarrow \B')$, the Grothendick construction, as given in \cite{GPS2}. The object set of this category is given by $ob(\B)\sqcup ob(\B')\sqcup ob(\B_0)$, and it contains $\B$, $\B'$ and $\B_0$ as full subcategories. Given $L_0\in ob(\B_0)$, $L\in ob(\B)$, $L'\in ob(\B')$, the hom-complexes from $L$ to $L'$, $L$ to $L_0$ and $L'$ to $L_0$ are defined to be $0$. The hom-complex $L_0$ to $L$ is given by $\B(j(L_0),L)$, and it is similar with $L_0$ to $L'$. Localization of $Groth (\B\leftarrow \B_0\rightarrow \B')$ at all $1_{j(L_0)}$, considered as a morphism in $Groth (\B\leftarrow \B_0\rightarrow \B')$ from $L_0$ to $L=j(L_0)$, and $1_{j'(L_0)}$, considered as a morphism from $L_0$ to $L'=j'(L_0)$,  gives the homotopy push-out of the same diagram. The explicit model of localization given in \cite{GPS2} (as a quotient as in \cite{lyubaquot}) contains the Grothendick construction as a non-full subcategory; therefore, it contains $\B$ and $\B'$ as non-full subcategories. It is also easy to see their inclusions are cohomologically full and faithful and split generate the localization. Denote the localization by $\underline{\B}$. If $i:\B\to \underline{\B}$, $i':\B'\to \underline{\B}$ denote the inclusion functors, one can also check using $f\circ j\simeq j'$ (or $f\simeq j'\circ j^{-1}$ for an inverse to $j$) that $i'\circ f\simeq i$ as functors. This implies the bimodule condition in the statement of the lemma.

In summary, $\underline{\B}$ satisfies everything asked in the statement of the lemma, except fullness: $\B$ and $\B'$ are non-full subcategories. Consider the inclusion of $i:\B\to \underline{\B}$, which is an $A_\infty$-functor with vanishing higher maps. The explicit construction of localization allows one to choose chain level left inverses $p$ to each chain map $i:\B(L_0,L_1)\to \underline{\B}(i(L_0),i(L_1))$ (satisfying $p\circ i=1$) and chain homotopy $h$ between $i\circ p$ and the identity. 
Moreover, $h$ can be assumed to vanish at the image of $i$. This homotopy transfer data as in \cite{markltransfer}, normally allows one to transfer $A_\infty$-structure on $\underline{\B}$ to $\B$. On the other hand, the properties $p,i,h$ satisfy imply that the transferred structure is the same as original $A_\infty$-structure on $\B$. Moreover, the $A_\infty$-functor $\B\to\underline{\B}$ that this process creates is the same as $i$. The same applies to $\B'$ as well. In other words, one can extend $i':\B'\to \underline{\B}$ to a transfer data such that the transferred and the original $A_\infty$-structure on $\B'$ coincide. 

Now, construct a category $\tilde \B$ with objects $ob(\B)\sqcup ob(\B')$. Let the hom-complexes between objects of $\B$, resp. $\B'$ be the same as hom-complexes of $\B$ resp. $\B'$. Let $\tilde \B(L_0,L_1):=\underline{\B}(i(L_0),i'(L_1))$, when $L_0\in ob(\B)$ and $L_1\in ob(\B')$, and same in the other direction. So far, we have not specified compositions, but there is an inclusion of each complex of $\tilde \B$ into $\underline{\B}$. Choose homotopy transfer data for this inclusion (i.e. left inverses to chain maps and appropriate homotopies) that extend previously chosen transfer data on $\B$ and $\B'$. Then, one can transfer the $A_\infty$-structure along this data to $\tilde \B$. By definition, the $A_\infty$-category is quasi-equivalent to its essential image, and the transferred $A_\infty$-structures on $\B$ and $\B'$ coincide with the original structures. It is easy to check that the properties in the lemma are satisfied. 
\end{proof}
\begin{lem}\label{defolem1}
Let $\B$ and $\B'$ be Morita equivalent $A_\infty$ categories over $\C$. Let $\B_R$ be a (possibly curved) deformation of $\B$ over $R=\C[[q]]$.Then there exists a (possibly curved) deformation $\B'_R$ of $\B'$ over $R$ such that the initial Morita equivalence extends to a Morita equivalence of $\B_R$ and $\B_R'$.
\end{lem}
\begin{proof}
Let $\tilde\B$ be as in Lemma \ref{lem:normalsupercategory}. Since $\B$ and $\B'$ are contained as full subcategories, there are restriction maps $CC^*(\tilde \B,\tilde \B)[1]\to CC^*(\B,\B)[1]$ and $CC^*(\tilde \B,\tilde \B)[1]\to CC^*(\B',\B')[1]$, and split generation implies they are both quasi-isomorphisms of dgla's. By \cite[Corollary V.52]{manettilectures}, quasi-isomorphism of dgla's induce isomorphism of deformation functors; thus, there exists a deformation of $\tilde \B$ that lifts the one on $\B$ (up to gauge equivalence, but this subtlety can be ignored by composing the inclusion of induced deformation of $\B$ into that of $\tilde \B$ with the gauge equivalence). Then, one can restrict the deformation to objects of $\B'$ to obtain a deformation of it. 

By Lemma \ref{lem:normalsupercategory}, one can choose $\tilde \B$ and inclusions of $\B$, $\B'$ such that, $\tilde\B$, considered as a $\B$-$\B'$ and $\B'$-$\B$-bimodule, is quasi-isomorphic to initial Morita equivalence and its inverse. Note that the $\B$-$\B$-bimodule homomorphism $\tilde \B\otimes_{\B'} \tilde \B\to \B$ is simply given by applying the $A_\infty$-product. The same is true for the convolution in other direction too; thus, these bimodule homomorphisms deform as well, and remain quasi-isomorphisms by Lemma \ref{semicont}. Hence, the deformation of $\tilde \B$ is still invertible, as a $\B$-$\B'$ (or $\B'$-$\B$-bimodule), and the deformations of $\B$ and $\B'$ are still Morita equivalent. The deformed Morita equivalence restricts to the initial Morita equivalence (since the gauge transformation among deformations of $\B$ deforms identity functor, the restriction of the bimodule is not effected it).
\end{proof}
Next result is a versality statement, which is a version of \cite[Lemma 3.5]{seidelK3} and indeed follows from \cite[Lemma 3.9]{seidelK3}.
\begin{lem}\label{defolem2}
Let $\B$ be an $A_\infty$-category such that $HH^2(\B)=\C$. Then any two (curved) deformations $\B_1$ and $\B_2$ of $\B$ over $R=\C[[q]]$ that are non-trivial in the first order are related by a base change by an automorphism $f_q$ of $R$ that specialize to identity at $q=0$. In other words, $\B_1=f_q^*\B_2$.
\end{lem}
\begin{cor}\label{defocor}
Assume $\B$ and $\B'$ are Morita equivalent. Let $\B_R$ and $\B_R'$ be respective curved deformations over $R=\C[[q]]$ that are non-trivial in the first order. Assume $HH^2(\B)\cong HH^2(\B')\cong \C$. Then there exists an automorphism $f_q$ of $R$ specializing to identity at $q=0$ such that initial Morita equivalence extends to a Morita equivalence of $\B_R$ and $f_q^*\B'_R$.
\end{cor}
\begin{proof}
This follows from Lemma \ref{defolem1} and \ref{defolem2}. Cf. \cite[Cor 3.6]{seidelK3}.
\end{proof}
Let us go back to proof of main theorem:
\begingroup
\def\thethm{\ref{mainthm}}
\begin{thm}Let $\A$ be as in Section \ref{sec:intro} and assume further that $HH^1(\A)=HH^2(\A)=0$. Assume $M_\phi$ is Morita equivalent to $M_{1_\A}$. Then, $\phi\simeq 1_\A$.
\end{thm}
\addtocounter{thm}{-1}
\endgroup
\begin{proof}
The Morita equivalence gives an isomorphism $HH^1(M_\phi)\cong HH^1(M_{1_\A})$. Moreover, it gives a correspondence of group-like families parametrized by $\G$ and the correspondence is compatible with taking deformation classes. This implies that the isomorphism carries $L(M_\phi):=L(M_\phi^R)|_{q=0}\subset HH^1(M_\phi)$ onto $L(M_{1_\A})\subset HH^1(M_{1_\A})$. The primitive class $\gamma_\phi\in L(M_\phi)\cong\Z^2$ is carried to another primitive class in $L(M_{1_\A})$. By Corollary \ref{sl2cor}, there exists a self-Morita equivalence of $M_{1_\A}$ that carries every primitive class to every other primitive class. In particular, we can find one that carries image of $\gamma_\phi$ to $\gamma_{1_\A}$ and composing the initial Morita equivalence with the latter, we can assume the isomorphism of Hochschild cohomologies induced by the equivalence $M_\phi\simeq M_{1_\A}$ maps $\gamma_\phi$ to $\gamma_{1_\A}$. 

By Corollary \ref{defocor}, the Morita equivalence extends to a Morita equivalence of $M_\phi^R$ and $f_q^* M_{1_\A}^R$ for some automorphism $f_q$ of $R$ specializing to identity at $q=0$. For simplicity assume $f_q=1_R$.

Under this equivalence $\gamma_\phi^R$ corresponds to a deformation of $\gamma_{1_\A}$ (i.e. to an element $\gamma_{1_\A}^R+O(q)$). This element also has to be in the discrete lattice $\Z^2\cong L(M_{1_\A}^R)\subset HH^1(M_\phi^R)$; hence, it is $\gamma_{1_\A}^R$. 

Consider two families of bimodules over $M_\phi^R$ and $M_{1_\A}^R$, which we denoted by $\scrG_R^{sf}$. To avoid confusion, let us now denote them by $\scrG_\phi$ and $\scrG_1$ respectively. They both satisfy the properties \ref{G1}-\ref{G3} on their domains (\ref{G3} is satisfied for the class $\gamma_\phi^R$ and $\gamma_{1_\A}^R$ respectively). The Morita equivalence gives rise to a correspondence of bimodules and families of bimodules. See (\ref{eq:conv3}) in Section \ref{subsec:familyreview}. By Corollary \ref{defmorita} and Remark \ref{rkmorita}, the family over $M_{\phi}^R$ corresponding to $\scrG_1$ satisfies \ref{G3} for the class corresponding to $\gamma_{1_\A}^R$, i.e. for $\gamma_\phi^R$ by the paragraph above. In other words, it follows $1\otimes \gamma_\phi^R$. Denote this family over $M_\phi^R$ by $\scrG_1'$. That it satisfies \ref{G1} essentially follows from the fact that the Morita equivalence between $M_\phi$ and $M_{1_\A}$ induces a quasi-equivalence between $tw^\pi(M_\phi)$ and $tw^\pi(M_{1_\A})$. That it satisfies \ref{G2} is clear.

Hence, by Theorem \ref{uniquenesstheorem} the families $\scrG_\phi$ and $\scrG_1'$ are the same up to $q$-torsion. In particular, consider their restriction to $R$-point $u=1$ of $Spf(A_R)$. By Remark \ref{rkforuone}, $\scrG_1$ restricts to diagonal; hence, $\scrG_1'$ restricts to diagonal of $M_\phi^R$. By the same remark, $\scrG_\phi$ restricts to kernel $\Phi_f^{-1}$ of the $\phi_f^{-1}$ that will be defined more carefully in Section \ref{sec:another}. Hence, $\Phi_f^{-1}$ is quasi-isomorphic to diagonal bimodule of $M_\phi^R$ up to $q$-torsion (thus, so is $\Phi_f$ by invertibility). 

Let $p=p_0\in \Tt_R$ be a smooth $R$-point supported on $C_0$. As remarked in the proof of Lemma \ref{uniinjective}, there exist an unobstructed object of $tw^\pi(\Om(\Tt_R)_{cdg} )$ given as a deformation of a cone of $``\Om_{C_0}(-1)\rightarrow \Om_{C_0}"$. Note that it is easier to define as an unobstructed module rather than a twisted complex. Hence, we have an unobstructed object $``\Om_p\otimes a"\in tw^\pi(M_\phi^R)$ for each $a\in ob(\A)$ and a full (uncurved) subcategory $\{\Om_p \}\otimes \A\subset tw^\pi(M_\phi^R)$. 
$\phi_f$ acts on this subcategory and the restriction of the bimodule $\Phi_f$ to it is given by \begin{equation}(\Om_p\otimes a,\Om_p\otimes a')\mapsto M_\phi^R(\Om_p\otimes a, \phi_f(\Om_p\otimes a') )\end{equation}In other words, it is the bimodule corresponding to action of $\phi_f$ on $\{\Om_p \}\otimes \A$. By above, it is quasi-isomorphic to diagonal bimodule up to $q$-torsion.

As these are uncurved categories, we can invert $q$. Once $q$ is inverted, the category $\{\Om_p \}\otimes \A$ becomes $\{\Om_p \}_K\otimes \A\simeq K[t]\otimes \A$, where $K=\C((q))$ and $t$ is a variable of degree $1$. Note slight sloppiness of notation about $q$-adic completions of  $\{\Om_p \}\otimes \A$.
On this category the diagonal $K[t]\otimes \A$ and $K[t]\otimes\Phi$ acts the same way. Lemma \ref{lemmarestr} concludes the proof.
\end{proof}	
\begin{lem}\label{lemmarestr}Let $\Phi$ and $\Psi$ be self Morita equivalences (we can assume $\Phi$ and $\Psi$ are just bimodules over $\A$, not necessarily invertible).
Assume $K[t]\otimes \Phi$ and $K[t]\otimes \Psi$ are quasi-isomorphic as $K[t]\otimes \A$-bimodules (where $deg(t)=1$). Then $\Phi$ and $\Psi$ are quasi-isomorphic.
\end{lem}
\begin{proof} Consider the algebra maps $K\xrightarrow{i} K[t]\xrightarrow{p}K$. We have a functor \begin{equation}Bimod(K[t],K[t])\rightarrow Bimod(K,K)\atop M\mapsto K\overset{L}{\otimes}_{K[t]}M \end{equation}where the bimodule structure on the right is induced by the inclusion. Geometrically this map would be $``(p^*,i_*)"$. It sends the diagonal bimodule of $K[t]$ to diagonal bimodule of $K$. We can define a similar functor \begin{equation}Bimod(K[t]\otimes \A,K[t]\otimes \A)\rightarrow Bimod(\A,\A)\atop \fM\longmapsto K\overset{L}{\otimes}_{K[t]}\fM \end{equation}
sending $K[t]\otimes \Phi$ and $K[t]\otimes \Psi$ to $\Phi$ and $\Psi$ respectively. This finishes the proof.
\end{proof}
\section{Growth rates and another dynamical invariant}\label{sec:another}
In the previous section, we have exploited the uniqueness of the family $\scrG_R^{sf}$ to distinguish trivial mapping tori from the others. However, $\scrG_R^{sf}$ encodes more and we can use it to produce more invariants of the tori. As mentioned in Remark \ref{rkforuone}, we can extract ``fiberwise $\phi$'' by restricting this family to $R$-point $u=1$. Let us define it more carefully.

Let $\psi$ be an auto-equivalence of $\A$ that commutes with $\phi$. For simplicity assume $\psi$ is a strict dg autoequivalence (i.e. acts bijectively on objects and hom-sets and its higher components vanish) and it commutes with $\phi$ strictly.  
\begin{defn}
Under these assumptions, $\psi$ induces auto-equivalences of $M_\phi$ and $M_\phi^R$ (again bijective on objects and hom-sets) given by descent of $1\otimes\psi$ acting on $\Om(\Tt_0)_{dg}\otimes\A$, resp. $\Om(\Tt_R)_{cdg}\otimes \A$ to their smash product with $\Z$, namely $M_\phi$, resp. $M_\phi^R$. Denote this autoequivalence by $\psi_f$ and corresponding $M_\phi$ resp. $M_\phi^R$-bimodule by $\Psi_f$.
\end{defn}
Intuitively, this autoequivalence corresponds to application $\psi$ on each fiber of ``the fibration $M_\phi\rightarrow\T_0$"; hence the name fiberwise $\psi$. This Section is about description of growth of $HH^*(M_\phi^R,\Phi^k_f)$.
\begin{rk}
Let $\Psi$ be the $\A$-bimodule corresponding to $\psi$, i.e. \begin{equation}\Psi: (a,a')\mapsto \A(a,\psi (a'))\end{equation}
Due to the strict commutation assumption, $\Psi$ is naturally $\Z_\Delta$-equivariant with the action is generated by \begin{equation}\A(a,\psi(a') )\rightarrow \A(\phi(a),\phi(\psi(a')) )=\A(\phi(a),\psi(\phi(a')) ) \end{equation}
We can obtain $\Psi_f$ by descent of $\Om(\Tt_0)_{dg}\otimes \Psi$, resp. $\Om(\Tt_R)_{cdg}\otimes \Psi$. In particular, we can assume $\Phi_f=(1\otimes\Phi)\#\Z$ in the sense of Section \ref{sec:construction}.
\end{rk}
\begin{lem}\label{baliklemma}
Assume $\A$	is a smooth dg category. Then \begin{equation}CC^*(M_\phi,\Psi_f)\simeq cocone\big(CC^*(\Om(\Tt_0)_{dg},\Om(\Tt_0)_{dg} )\otimes CC^*(\A,\Psi)\atop\xrightarrow{\tr_*\otimes \phi_*-1} CC^*(\Om(\Tt_0)_{dg},\Om(\Tt_0)_{dg} )\otimes CC^*(\A,\Psi) \big) \end{equation}i.e. the derived invariants of \begin{equation}CC^*(\Om(\Tt_0)_{dg},\Om(\Tt_0)_{dg} )\otimes CC^*(\A,\Psi)\end{equation} 
\end{lem}
\begin{proof}
The proof of Prop \ref{CCtorus} works in this case. Namely, we would need to replace $CC^*(\A,\A)$ by $CC^*(\A,\Psi)$ and so on. 
\end{proof}
\begin{cor}\label{invhf}
Assume $\A$	is a smooth dg category. Then \begin{equation}CC^*(M^R_\phi,\Psi_f)\simeq cocone\big(CC^*(\Om(\Tt_R)_{cdg},\Om(\Tt_R)_{cdg} )\otimes CC^*(\A,\Psi)\atop\xrightarrow{\tr_*\otimes \phi_*-1} CC^*(\Om(\Tt_R)_{cdg},\Om(\Tt_R)_{cdg} )\otimes CC^*(\A,\Psi) \big) \end{equation}i.e. the derived invariants of \begin{equation}CC^*(\Om(\Tt_R)_{cdg},\Om(\Tt_R)_{cdg} )\otimes CC^*(\A,\Psi)\end{equation}
\end{cor}
\begin{proof}
All the maps leading to quasi-isomorphism in Lemma \ref{baliklemma} can be written over $R$, as remarked in Section \ref{sec:hoch}. Hence, the result follows from Lemma \ref{semicont} and \ref{baliklemma}. Note we use $q$-adically completed tensor products again.
\end{proof}
By Proposition \ref{HHOmfinal} and K\"unneth formula, the cohomology of \begin{equation}CC^*(\Om(\Tt_R)_{cdg},\Om(\Tt_R)_{cdg} )\otimes CC^*(\A,\Psi)\end{equation} is isomorphic to \begin{equation}(R\oplus R[-1]\oplus HH^{\geq 2 }(\Om(\Tt_R)_{cdg},\Om(\Tt_R)_{cdg}) )\otimes HH^*(\A,\Psi) \end{equation}
up to $q$-torsion. $HH^{\geq 2 }(\Om(\Tt_R)_{cdg},\Om(\Tt_R)_{cdg})$ is also $q$-torsion by Proposition \ref{HHOmfinal}. Hence, it is $HH^*(\A,\Psi)\oplus HH^*(\A,\Psi)[-1]$. By Corollary \ref{invhf}, its derived invariants compute $HH^*(M_\phi^R,\Psi_f)$. In other words
\begin{align*}
HH^*(M_\phi^R,\Psi_f)\cong R\otimes HH^*(\A,\Psi)^\phi\oplus\\
R\otimes (HH^*(\A,\Psi)^\phi\oplus HH^*(\A,\Psi)/(\phi-1) )[-1]\oplus\\
 R\otimes HH^*(\A,\Psi)/(\phi-1)[-2] 
\end{align*}
Letting $\psi=\phi^k$, this relates growth of $HH^*(M_\phi^R)$ to the growth of invariant part of $HH^*(\A,\Phi^k)$. We also recover 
\begin{align*}
 HH^*(\A,\Psi)^\phi\oplus\\
 (HH^*(\A,\Psi)^\phi\oplus HH^*(\A,\Psi)/(\phi-1) )[-1]\oplus\\
 HH^*(\A,\Psi)/(\phi-1)[-2] 
\end{align*}
out of $HH^*(M_\phi^R,\Psi_f)$ up to $q$-torsion (the input is up to $q$-torsion). By Section \ref{sec:unique} this data is an invariant of the pair $(M_\phi,\gamma_\phi)$, when $\psi=\phi^k$, where $k\in\Z$. Hence, we obtain:
\begin{prop}
\begin{align*}
HH^*(\A,\Phi^k)^\phi\oplus\\
(HH^*(\A,\Phi^k)^\phi\oplus HH^*(\A,\Phi^k)/(\phi-1) )[-1]\oplus\\
HH^*(\A,\Phi^k)/(\phi-1)[-2] 
\end{align*}
is an invariant of the pair $(M_\phi,\gamma_\phi)$. In other words, if $M_\phi$ is Morita equivalent to $M_{\phi'}$ such that $\gamma_\phi$ corresponds to $\gamma_{\phi'}$ under the induced isomorphism between Hochschild cohomologies, then the graded vector spaces given above are isomorphic. Note $\Phi^k$ denotes the self-convolution of the bimodule $\Phi$ $k$-times.
\end{prop}

\appendix
\section{Modules over $A_R$}\label{sec:modules}
Recall $A_R=\C[u,t][[q]]/(ut-q)$ and $A=A_R/(q)=\C[u,t]/(ut)$.
\begin{defn}
	Let $D_A$, resp. $D_{A_R}$ denote the derivation $t\partial_t-u\partial_u$ on $A$, resp. $A_R$. 
\end{defn}
Note that $D_{A_R}$ is $R$-linear on $A_R$ and it can be seen as the infinitesimal action of the $\G$-action given by $z:t\mapsto zt, u\mapsto z^{-1}u $.
\begin{defn}\label{connectionalong}Let $M$ be a topological module over $A_R$. A connection on $M$ along $D_{A_R}$ is a linear map $D_M:M\rightarrow M$ satisfying $D_M(fm)=D_{A_R}(f)m+fD_M(m)$. We will often refer to it simply as a connection on $M$. 
\end{defn}
A connection can be seen as an infinitesimal version of an equivariant structure with respect to the action above. 
\begin{rk}$A_R$ is a Noetherian ring. Moreover, finitely generated modules over $A_R$ are automatically complete with respect to $q$-adic topology. 
\end{rk}
Let us first prove:
\begin{prop}\label{qtorsmod}
	Let $M$ be a finitely generated module over $A_R$, which can be endowed with a connection $D_M$.	Assume $M/(t-1)M$ is a $q$-torsion module over $A_R/(t-1)\cong R$. Then, $M$ is $q$-torsion.
\end{prop}
\begin{proof}
	We can assume $M/(t-1)M=0$, by replacing $M$ by $q^iM$ for $i\gg 0$. Consider $M/uM$. It is a finitely generated module over $\C[t]$ and carries a connection along the derivation $t\partial_t$ on $\C[t]$. As it vanishes at $t=1$, it has to be torsion over $\C[t]$ and as it carries a connection its annihilator is invariant under the action $z\in\G :t\mapsto zt$. Hence, $ann_{\C[t]}(M/uM)=(t^{n-1})$ for some $n$, implying $t^nM\subset utM=qM\subset tM$. This shows $t$-adic and $q$-adic topologies on $M$ coincide and $M$ is $t$-adically complete as well. Hence, we can see $M$ as a module over $\C[u][[t]]$ that is finitely generated over $A_R=\C[u,t][[ut]]\subset \C[u][[t]]$.
	
	Given $s\in M$, consider $\C[u][[t]].s\subset M$. It is an $A_R$ submodule and $A_R$ is Noetherian; hence, $\C[u][[t]].s\cong \C[u][[t]]/ann(s)$ is finitely generated over $A_R$ as well, where $ann(s):=ann_{\C[u][[t]]}(s)$. Now dividing by $u$ again, this implies \begin{equation}\C[u][[t]]/(ann(s)+u\C[u][[t]])\end{equation} is a module over $\C[[t]]$ that is finitely generated over $\C[t]=A_R/(u)$. By the classification of finitely generated modules over PIDs, we see that this module is indeed $t$-torsion, i.e. there exists $N'$ such that \begin{equation}t^{N'}\in ann(s)+u\C[u][[t]]\end{equation} In other words $ann(s)$ contains an element that is of the form $t^{N'}+O(u).$
	
	Now, pick a set of generators $s_1,\dots ,s_r$ and $N_i$ such that there exists an element in $ann(s_i)$ that is of the form $t^{N_i}+O(u).$ The product of these annihilating elements is of the form $t^N+O(u)\in \C[u][[t]]$ and annihilates $M.$ 
	
	Let $\hat{M}$ denote the completion of $M$ with respect to the ideal $(u)\subset \C[u][[t]].$ It is a finitely generated module over $\C[[u,t]]$ and it also carries a connection along $t\partial_t-u\partial_u$. Hence, its annihilator $J$ over $\C[[u,t]]$ satisfy the conditions of Lemma \ref{curvesatorigin} below. Moreover, by the paragraph above, an element of the form $t^N+O(u)$ is in $J$. Hence, $J$ cannot be contained in $(0)$ or $(u)$ and the prime ideals belonging to $J$ contain $t$. This implies there exists $N_1$ such that $t^{N_1}\in J.$ In other words, $t^{N_1}$ annihilates $\hat{M}$. 
	
	We want to use this to show $t^{N_1}M=0$, implying $M$ is q-torsion.
	
	Our approach is using \cite[Theorem 10.17]{comalg}, namely the kernel of the (u-adic) completion map $M\rightarrow \hat{M}$ is the set of elements of $M$ annihilated by some element of $1+(u)$. By above, $t^{N_1}M$ is in the kernel and we see that $M$ is annihilated by an element of the form $t^{N_1}(1+O(u))$. Consider $t^{N_1}M/t^{N_1+1}M$. It is annihilated by $t$ and an element of the form $1+O(u)$. Hence, we can get rid of multiples of $t$ in $1+O(u)$ and see that there exists a polynomial $f(u)$ such that $1+uf(u)$ annihilates $t^{N_1}M/t^{N_1+1}M$. $t^{N_1}M/t^{N_1+1}M$ is a finitely generated module over $\C[u]$ with a connection along $-u\partial_u$. Hence, by the classification of modules over PIDs, it has to be finite direct sum of copies of $\C[u]$ and of $\C[u]/(u^l)$, for various $l\geq 1$. Thus, that it is annihilated by $1+uf(u)$ implies $t^{N_1}M/t^{N_1+1}M=0$. In other words, \begin{equation}t^{N_1}M=t^{N_1+1}M=t^{N_1+2}M=\dots\end{equation} But recall $M$ is complete in $t$-adic topology i.e. the completion map \begin{equation}M\rightarrow \lim\limits_{\leftarrow} M/t^nM \end{equation} is an isomorphism. Thus, $M\simeq M/t^{N_1}M$ and $t^{N_1}M=0$. This implies $q^{N_1}M=0$, finishing the proof. 
\end{proof}
\begin{lem}\label{curvesatorigin} Consider the action of the group $\C^*$ on $\C[[u,t]]$, where $z\in\C$ acts by $t\mapsto zt, u\mapsto z^{-1}u$. Consider an ideal $J$ that is invariant under the action, or equivalently $(t\partial_t-u\partial_u)(J)\subset J$. If $J$ is a prime ideal, then it is one of $(0),(u),(t),(u,t)$. If $J$ is an arbitrary invariant ideal, then the prime ideals belonging to $J$ (its prime components) are among $(0),(u),(t),(u,t)$ (thus, $\sqrt{J}$ is the intersection of some of $(0),(u),(t),(u,t)$).
\end{lem}
\begin{proof}
	The second statement is an easy corollary of the first: namely assume $J=\bigcap \mathfrak{q_i}$ is a a minimal primary decomposition, which exist by \cite[Theorem 7.13]{comalg}. Then prime radicals $\mathfrak{p_i}=\sqrt{\mathfrak{q_i}}$ are invariants of $J$, by \cite[Theorem 4.5]{comalg}.
	$J$ is fixed under the action of $\C^*$ on $\C[[u,t]]$. Hence, so are its prime components.
	
	Now assume $J$ is prime. As the Krull dimension of $\C[[u,t]]$ is $2$, the height of $J$ can be 0,1 or 2. If it is 0, $J=(0)$. If it is 2, $J=(u,t)$ as $\C[[u,t]]$ is a local ring. Hence, assume $J$ has height 1. As $\C[[u,t]]$ is a UFD, $J$ is principal (see \cite[Prop 1.12A]{hartshorne}). Take a prime $f\in J$ such that $J=(f)$. The invariance of $J$ under $\C^*$-action implies that for all $z\in\C^*$, $z.f$ is a generator as well; hence, it differs from $f$ by a unit of $\C[[u,t]].$
	
	Partially order the monomials as \begin{equation}t^au^b\leq t^c u^d\text{ iff } a\leq c \text{ and } b\leq d \end{equation} As the units of $\C[[u,t]]$ are of the form $\alpha+O(u,t)$, where $\alpha\in \C^*$, the set of non-zero monomials of $f$ that are minimal with respect to this order does not change when we multiply it with a unit. 
	Moreover, the coefficients of the minimal monomials are multiplied by the same constant $\alpha\in\C^*$. On the other hand, the $\C^*$-action acts on the monomial $t^iu^j$ by $z^{i-j}.$ Thus, the difference $i-j$ has to be the same for all minimal monomials. But if $i-j=i'-j'$, then either $t^iu^j\leq t^{i'}u^{j'}$ or $t^{i'}u^{j'}\leq t^{i}u^{j}$. Hence, there can be only one minimal non-zero monomial of $f$. Call it $t^iu^j$. As all the other monomials are divisible by it, $t^iu^j$ differs from $f$ by a unit; hence, $J=(t^iu^j)$. As $J$ is prime, it is either $(t)$ or $(u)$. This finishes the proof of the lemma.
\end{proof}
We wish to use Proposition \ref{qtorsmod} to prove some properties of modules of higher rank. For that, we need another lemma:
\begin{lem}\label{freevee} Let $M$ be a finitely generated module over $A_R$, which can be endowed with a connection $D_M$. Then $M^{\vee}=Hom_{A_R}(M,A_R)$ is free over $A_R$. 
\end{lem}
\begin{proof}
	$M^\vee$ is finitely generated and admits a connection as well. Assume $M^\vee\neq 0$. Consider the local ring $(A_R)_{(u,t)}\subset \C[[u,t]]$. Note, we do not take its $q$-completion.  This is a Noetherian local ring whose $(u,t)$-adic completion is $\C[[u,t]]$. Hence, by \cite[Cor 11.19]{comalg}, they have the same Krull dimension, which is $2$. Thus, \begin{equation}depth((A_R)_{(u,t)})\leq dim((A_R)_{(u,t)})=2\end{equation} by \cite[Prop 18.2]{comalgeis}. As $u,t$ is a regular sequence (i.e. $u$ is not a zero divisor on $A_R$ and $t$ is not a zero divisor on $A_R/(u)$), $depth((A_R)_{(u,t)})=2$. 
	
	Moreover, $(A_R)_{(u,t)}$ is a regular local ring, 
	hence it has finite global dimension (see \cite[Cor 19.6]{comalgeis}). Thus, we can apply the Auslander-–Buchsbaum formula (\cite[Thm 19.9]{comalgeis}) to every finitely generated module and obtain \begin{equation}depth(M')+pd(M')=depth((A_R)_{(u,t)})=2 \end{equation} where $pd,depth$ are over $(A_R)_{(u,t)}$. Let $M':=(M^\vee)_{(u,t)}$. Clearly, $u,t$ is a regular sequence for $M'\cong Hom_{A_R}(M,(A_R)_{(u,t)})$; hence, if $M'\neq 0$, it has depth 2 and projective dimension 0. Thus, it is projective. As $(A_R)_{(u,t)}$ is local, this implies $M'=(M^\vee)_{(u,t)}$ is free.
	
	In particular, this implies that the $A=\C[u,t]/(ut)$-module $(M^\vee)_0=M^\vee/qM^\vee$ is free around $(0,0)$, i.e. $(M^\vee/qM^\vee)_{(u,t)}$ is free over $A_{(u,t)}$. Using the connection on $M^\vee$, one can show the freeness of $((M^\vee)_0)_t$, resp. $((M^\vee)_0)_u$ over $\C[t,t^{-1}]$, resp. $\C[u,u^{-1}]$. This is sufficient to conclude $(M^\vee)_0$ is free. 
	
	This implies the freeness of $M^\vee$ as well: choose a basis $A^n\xrightarrow{\cong} (M^\vee)_0$ and lift it to a linear map $A^n_R\rightarrow M^\vee$. A simple semi-continuity argument would show this is an isomorphism as well, finishing the proof.
\end{proof}
\begin{rk}The proof implies the freeness of any finitely generated module with connection for which $u,t$ is a regular sequence. However, we do not need this. 
\end{rk}
We can use Proposition \ref{qtorsmod} and Lemma \ref{freevee} to prove:
\begin{prop}\label{freeuptoqtors}Let $M$ be a finitely generated module over $A_R$ that can be endowed with a connection. Then, $M$ is free up to $q$-torsion.
\end{prop}
\begin{proof}
	Consider the natural map $M\rightarrow M^{\vee\vee}$. This map is compatible with the connections; thus, both its kernel and cokernel are $q$-torsion by Proposition \ref{qtorsmod}. $M^{\vee\vee}$ is free by Lemma \ref{freevee}, implying $M$ is free up to $q$-torsion.
\end{proof}
\begin{rk}One can apply proof of Proposition \ref{freeuptoqtors} to produce a map $M^{\vee\vee}\hookrightarrow M$ with $q$-torsion cokernel.
\end{rk}

\bibliographystyle{alpha}
\bibliography{biblioforspefibre}
\end{document}